\pdfoutput=1
\documentclass{scrartcl}
\usepackage{topzeta}
\ifpdf
\hypersetup{pdftitle={Computing topological zeta functions of groups, algebras, and modules, I}, pdfauthor={Tobias Rossmann}}
\fi

\title{Computing topological zeta functions of groups, algebras, and modules, I}
\author{Tobias Rossmann} \affil{\small
  Fakult\"at f\"ur Mathematik, Universit\"at Bielefeld, D-33501 Bielefeld,
  Germany}

\date{May 2014}
\begin{document}
\maketitle

\begin{abstract}
  \small We develop techniques for computing zeta functions associated with
  nilpotent groups, not necessarily associative algebras, and modules, as well
  as Igusa-type zeta functions.  At the heart of our method lies an explicit
  convex-geometric formula for a class of $p$-adic integrals under
  non-degeneracy conditions with respect to associated Newton polytopes.  Our
  techniques prove to be especially useful for the computation of topological
  zeta functions associated with algebras, resulting in the first systematic
  investigation of their properties.
\end{abstract}

\blankfootnote{\indent{\itshape 2000 Mathematics Subject Classification.}
  11M41, 20F69, 14M25.

  This work is supported by the DFG Priority Programme ``Algorithmic and
  Experimental Methods in Algebra, Geometry and Number Theory'' (SPP 1489).}

\tableofcontents

\section{Introduction}
\label{s:intro}

\paragraph{Zeta functions of nilpotent groups.}
For a finitely generated group $G$, the number $a_n(G)$ of subgroups of $G$ of
index $n$ is finite for all $n\ge 1$.  Suppose that $G$ is torsion-free and
nilpotent.  A powerful tool to analyse the arithmetic properties of the sequence
$a_1(G),a_2(G),\dotsc$ is the \emph{subgroup zeta function} of $G$ introduced by
Grunewald, \mbox{Segal}, and Smith \cite{GSS88}.  It is given by the Dirichlet
series $\zeta_G(s) = \sum_{n=1}^\infty a_n(G) n^{-s}$ which defines an analytic
function in some complex right half-plane.  As a consequence of nilpotency,
$\zeta_G(s)$ admits an Euler product decomposition $\zeta_G(s) = \prod_p
\zeta_{G,p}(s)$ into local factors $\zeta_{G,p}(s) = \sum_{d=0}^\infty
a_{p^d}(G) \dtimes p^{-ds}$ indexed by primes $p$, see \cite[Prop.\ 1.3]{GSS88}.
A much deeper fact \cite[Thm\ 1]{GSS88} is that each $\zeta_{G,p}(s)$ is a
rational function in $p^{-s}$ over $\QQ$.  By considerably advancing the
techniques pioneered in \cite{GSS88}, du Sautoy and Grunewald \cite[Thm\
1.1(1)]{dSG00} derived a fundamental result on nilpotent groups: the degree of
polynomial subgroup growth of $G$, a priori a positive real number, is in fact
rational.  Their proof using zeta functions constitutes a crucial step in the
development of the theory of zeta functions of groups and rings as a subject of
independent interest.

\paragraph{Zeta functions of groups, rings, and modules.}
Two key insights of \cite{GSS88}, some details of which we will recall in
\S\ref{s:groups}, have been central to subsequent developments in the theory of
zeta functions of nilpotent groups.  First, the enumeration of subgroups of such
a group is essentially reduced, via linearisation, to the enumeration of
subrings of an associated nilpotent Lie ring.  Second, the corresponding local
subring zeta functions $\zeta_{\cL,p}(s)$ of such a ring $\cL$ admit explicit
expressions in terms of $p$-adic integrals, thus rendering the subject amenable
to a range of powerful techniques from arithmetic geometry and model theory, in
particular.  By invoking a result of Denef on the rationality of certain
definable integrals, the first major result \cite[Thm\ 3.5]{GSS88} in this
direction is that each $\zeta_{\cL,p}(s)$ is rational in $p^{- s}$ over $\QQ$.
As it turns out, many fundamental properties of $\zeta_{\cL,p}(s)$ remain intact
when the requirement that $\cL$ be a nilpotent Lie ring is relaxed whence the
study of zeta functions of non-associative rings evolved into an area of its
own.

As we will explain in \S\ref{s:groups}, the above framework also accommodates
Solomon's zeta functions \cite{Sol77} enumerating submodules associated with
integral representations; see \cite{BR80} for applications of $p$-adic
integration in this context.  Finally, we note that the theory of representation
growth, a presently very active branch of asymptotic group theory (see e.g.\
\cite{AKOV13}), incorporates much of the machinery alluded to above.

\paragraph{Igusa's local zeta function.}
Following the introduction of techniques from $p$-adic integration, subsequent
major advances in the theory of zeta functions of groups and rings, for example
\cite{dSG00,Vol10}, drew heavily upon results surrounding Igusa's local zeta
function.  We refer to \cite{Den91a,Nic10} for introductions to the area.  Write
$\ZZ_p$ for the ring of $p$-adic integers.  In its simplest incarnation, Igusa's
local zeta function $\Zeta_{f,p}(s)$ associated with a polynomial $f\in
\ZZ_p[X_1,\dotsc,X_n]$ is given by $\Zeta_{f,p}(s) = \int_{\ZZ_p^n}
\abs{f(\xx)}^s \dd\mu(\xx)$, where $\abs{\dtimes}$ denotes the usual $p$-adic
absolute value, $\mu$ is the additive Haar measure on $\ZZ_p^n$ with
$\mu(\ZZ_p^n) = 1$, and $s\in \CC$ satisfies $\Real(s) > 0$.  Igusa'a local zeta
function $\Zeta_{f,p}(s)$ enumerates solutions of congruences $f(\xx) \equiv 0
\bmod p^{d}$ via
\[
\frac{1-p^{-s} \Zeta_{f,p}(s)}{1-p^{-s}} = \sum_{d=0}^\infty \noof{\{ \bar\xx\in
  (\ZZ/p^d)^n : f(\bar\xx) \equiv 0 \bmod {p^d}\}} \dtimes p^{-(s+n)d}.
\]

Igusa proved (see \cite[Thm\ 8.2.1]{Igu00}) that $\Zeta_{f,p}(s)$ is always
rational in $p^{-s}$ over $\QQ$.  For a polynomial $f\in \ZZ[X_1,\dotsc,X_n]$,
we may consider $\Zeta_{f,p}(s)$ for all primes $p$. Denef \cite{Den87} gave a
formula for $\Zeta_{f,p}(s)$ that, in particular, explains the dependence on the
prime $p$.  A fundamental insight of his was that $\Zeta_{f,p}(s)$ can, for
almost all $p$, be expressed explicitly in terms of an embedded resolution of
singularities of $\Spec(\QQ[X_1,\dotsc,X_n]/f)\subset \AA^n_{\QQ}$.

It is evident from Igusa's rationality result what it means to compute
$\Zeta_{f,p}(s)$.  The existence of algorithms for resolution of singularities
in characteristic zero \cite{Vil89,BM97} implies that Denef's approach can, at
least in principle, be used to carry out such computations.  However, the rather
unfavourable complexity estimates of known algorithms for resolving
singularities (cf.\ \cite{BGMW11}) as well as practical experience show that
this approach is only useful in very small dimensions---in fact, often only for
$n \le 3$.

Among the more practically-minded methods for computing $\Zeta_{f,p}(s)$, a
central place is taken by ``toroidal'' ones that, under suitable non-degeneracy
assumptions, yield explicit formulae for $\Zeta_{f,p}(s)$ in terms of the Newton
polytope (or Newton polyhedron) of $f$; see \cite{DH01,VZG08} and the references
in \S\ref{ss:ndlit}.  Such methods go back to work of Khovanskii \cite{Kho77}
and others \cite{Kus76,Var76} on ``toroidal compactifications'' of complex
varieties.  These results have been a major source of inspiration for the work
described in the present article.

\paragraph{Results I: non-degenerate $p$-adic integrals.}
In Sections~\ref{s:polyzeta}--\ref{s:padic_eval}, we study multivariate local
zeta functions defined by $p$-adic integrals attached to globally defined
collections of polynomials.  The rather general class of integrals that we
consider specialises to local subring, subgroup, and submodule zeta functions as
well as to various Igusa-type zeta functions.  Our main result
(Theorem~\ref{thm:evaluate}) is an explicit convex-geometric formula for the
aforementioned integrals under non-degeneracy assumptions of the defining
polynomials with respect to their Newton polytopes.  When applied in the context
of Igusa-type zeta functions, our formula generalises a result of Denef and
Hoornaert \cite{DH01} and provides an alternative to work of Veys and
\Zuniga-Galindo \cite{VZG08}.  To the author's knowledge, our formula is the
first of its kind that applies to subring and submodule zeta functions.

\paragraph{Topological zeta functions.}
The \emph{topological zeta function} $\Zeta_{f,\topo}(s) \in \QQ(s)$ associated
with a polynomial $f\in \ZZ[X_1,\dotsc,X_n]$ was introduced by Denef and Loeser
\cite{DL92}.  It can be thought of as a limit ``$p\to 1$'' of the local zeta
functions $\Zeta_{f,p}(s)$ from above.  For example, it is easy to see that
$\Zeta_{X,p}(s) = \int_{\ZZ_p}\abs{x}^s\dd\mu(x) = \frac{1-p^{-1}}{1-p^{-s-1}}$.
Informally, we obtain $\Zeta_{X,\topo}(s)$ as the constant term of
$\Zeta_{X,p}(s)$, expanded as a series in $p-1$.  In our example, disregarding
questions of convergence, using the binomial series, we find that
$\Zeta_{X,p}(s) = 1/(s+1) + \cO(p-1)$ whence $\Zeta_{X,\topo}(s) = 1/(s+1)$.
While our informal approach fails to be rigorous in a number of ways, the
results of \cite{DL92} serve to fill these gaps.  Using a different approach,
Denef and Loeser \cite{DL98} later obtained $\Zeta_{f,\topo}(s)$ as a
specialisation of the motivic zeta function associated with $f$, another
invention of theirs.  As already observed in \cite{DL92}, topological zeta
functions associated with polynomials retain interesting properties of their
more complicated $p$-adic ancestors.  At the same time, they tend to be more
amenable to computations.  Important contributions to their study have, in
particular, been made by Veys and his collaborators, see e.g.\
\cite{Vey95,LV99,LSV06}.

In \cite{dSL04}, du Sautoy and Loeser introduced motivic subalgebra zeta
functions.  As a by-product, they defined associated topological zeta functions
and gave a number of examples in cases where the motivic zeta function has been
computed.  Informally, if $\cL$ is a possibly non-associative ring of additive
rank $d$, then its topological subring zeta function $\zeta_{\cL,\topo}(s)$ is
the constant term of $(1-p^{-1})^d \zeta_{\cL,p}(s)$ as a series in $p-1$.  For
instance, if $(\ZZ^d,0)$ denotes $\ZZ^d$ endowed with the zero multiplication,
then, as is well-known, $\zeta_{(\ZZ^d,0),p}(s) =
\zeta_p(s)\zeta_p(s-1)\dotsb\zeta_p(s-(d-1))$, where $\zeta_p(s) = 1/(1-p^{-s})$
is the $p$-local factor of the Riemann zeta function.  We thus expect
$\zeta_{(\ZZ^d,0),\topo}(s) = \frac 1 {s(s-1)\dotsb (s-(d-1))}$ which is indeed
the case using the definition of $\zeta_{\cL,\topo}(s)$ given in the present
article (but see Remark~\ref{r:shift}).  For a rigorous treatment, in
\S\ref{s:topzeta}, we give a self-contained introduction to topological zeta
functions, including those arising from the enumeration of subrings, based on
the original approach from \cite{DL92}.

\paragraph{Results II: computing topological zeta functions.}
In \S\ref{s:top_eval}, we continue our investigation of the $p$-adic integrals
in \S\ref{s:padic_eval} from the point of view of topological zeta functions.
As our main result (Theorem~\ref{thm:topeval}), we give explicit and purely
combinatorial formulae for the topological zeta functions associated with our
integrals under non-degeneracy conditions.

In \S\ref{s:app}, we illustrate our formulae by explicitly computing examples of
zeta functions of groups, rings, and modules, old and new.  While such
computations can be carried out by hand in small cases, the true strength of our
approach (the topological part, in particular) lies in its machine-friendly
form.  Indeed, based on Theorem~\ref{thm:topeval}, the author has developed a
practical method for computing topological zeta functions associated with
nilpotent groups, not necessarily associative rings, modules, as well as various
Igusa-type zeta functions under non-degeneracy assumptions.  A detailed account
of the computational techniques and extensions needed to transform the present
theoretical work into such a practical method will be the subject of a separate
article \cite{topzeta2}.  As an illustration, our method allows us to compute
the previously unknown topological subring zeta function of the nilpotent Lie
ring $\Fil_4$, see \eqref{eq:Fil4}. In contrast, according to \cite[\S
2.13]{dSW08}, the $p$-adic subring zeta functions of $\Fil_4$ have so far
resisted ``repeated efforts'' to compute them.  As explained in \S\ref{ss:Fil4},
this particular example concludes the determination of the topological subgroup
zeta functions of all nilpotent groups of Hirsch length at most $5$.

Based on considerable experimental evidence, in \S\ref{s:conjectures}, we state
a number of intriguing conjectures.  It would seem that some of these
conjectures are in fact topological shadows of conjectural properties of
$p$-adic zeta functions that went previously unnoticed.  For example, if $\cL$
is a nilpotent Lie ring of additive rank $d$, then
Conjecture~\ref{conj:comagic_P} predicts that the meromorphic continuations of
$\zeta_{\cL,p}(s)$ and $\zeta_{(\ZZ^d,0),p}(s)$ (where $(\ZZ^d,0)$ denotes
$\ZZ^d$ endowed with the zero multiplication) agree at $s = 0$ in the sense that
$\frac{\zeta_{\cL,p}(s)}{\zeta_{(\ZZ^d,0),p}(s)}\Big\vert_{s=0} = 1$.

\subsection*{Acknowledgements}

I would like to express my gratitude to Christopher Voll.  This project grew out
of numerous discussions with him and continues to benefit greatly from his
insight and support.  Furthermore, I would like to thank Wim Veys for
interesting discussions on topological zeta functions.

\subsection*{{\normalfont \it Notation}}

We let ``$\subset$'' signify not necessarily proper inclusion.  The cardinality
of a set $A$ is denoted by $\noof A$ or $\card A$.  We write $\NN = \{ 1,
2,\dotsc\}$ and $\NN_0 = \NN \cup \{ 0\}$.  The ideal generated by a set or
family $S$ within a given ambient ring is denoted by $\langle S\rangle$.  For a
family $\ff = (f_i)_{i\in I}$ of polynomials, we write $\ff(\xx) =
\bigl(f_i(\xx)\bigr)_{i\in I}$ and similarly for sets of polynomials.  We often
write $\one = (1,\dotsc,1)$.  Base change is usually denoted using subscripts.
We let $\Torus^n = \Spec\bigl(\ZZ[\lambda_1^{\pm 1},\dotsc,\lambda_n^{\pm
  1}]\bigr)$ and we identify $\Torus^n(R) = (R^\times)^n$.  Given a valuation
$\nu$ on a field $K$, we write $\nu(\xx) = (\nu(x_1),\dotsc,\nu(x_n))$ for
$\xx\in K^n$.

\section{Global and local zeta functions of algebras and modules}
\label{s:groups}

In this section, we elaborate on the relationship between local zeta functions
of groups and rings and $p$-adic integration.

\subsection{Subalgebra and submodule zeta functions}
Let $R$ be the ring of integers in a global or a non-Archimedean local field of
characteristic zero.  We consider slight generalisations of subring \cite{GSS88}
and submodule \cite{Sol77} zeta functions.  By a \emph{non-associative algebra},
we always mean a not necessarily associative algebra.

\begin{defns}
  \label{d:subzeta}
  \quad
  \begin{enumerate}
  \item
    \label{d:subzeta1}
    Let $\cA$ be a non-associative $R$-algebra. Suppose that the underlying
    $R$-module of $\cA$ is free of finite rank.  The \emph{subalgebra zeta
      function} of $\cA$ is
    \[
    \zeta_{\cA}(s) = \sum_{n=1}^\infty \noof{\{\cU : \text{$\cU$ is an
        $R$-subalgebra of $\cA$ with $\idx{\cA:\cU} = n$}\}} \dtimes n^{-s},
    \]
    where $\idx{\cA:\cU}$ denotes the cardinality of the $R$-module $\cA/\cU$.
  \item
    \label{d:subzeta2}
    Let $M$ be a free $R$-module of finite rank and let $\cE$ be a subalgebra of
    the associative $R$-algebra $\End_R(M)$.  The \emph{submodule zeta function}
    of $\cE$ acting on $M$ is
    \[
    \zeta_{\cE\acts M}(s) = \sum_{n=1}^\infty \noof{\{U : \text{$U$ is an
        $\cE$-submodule of $M$ with $\idx{M:U} = n$}\}} \dtimes n^{-s}.
    \]
  \end{enumerate}
\end{defns}
\begin{rems}
  \label{r:algmod}
  \quad
  \begin{enumerate}
  \item The subalgebra zeta function specialises to the global ($R = \ZZ$) and
    local ($R = \ZZ_p$) subring zeta functions from \cite[\S 3]{GSS88}.  If $R =
    \ZZ$ and $\cE\otimes \QQ$ is semisimple over $\QQ$, then $\zeta_{\cE\acts
      M}(s)$ is the zeta function introduced by Solomon in \cite{Sol77}; it
    admits an Euler factorisation and $\cE\otimes \ZZ_p$ acting on $M \otimes
    \ZZ_p$ gives rise to its local factors.
  \item
    \label{r:algmod2}
    The \emph{ideal zeta function}
    \[
    \zeta_{\cA}^\normal(s) := \sum_{n=1}^\infty \noof{\{\cU\normal \cA :
      \idx{\cA:\cU} = n\}} \dtimes n^{-s}
    \]
    of $\cA$ defined in \cite[\S 3]{GSS88} (for $R = \ZZ$), itself a
    generalisation of the Dedekind zeta function of a number field, is an
    instance of a submodule zeta function---indeed, $\zeta_{\cA}^\normal(s) =
    \zeta_{\cE\acts \cA}^{\phantom o}(s)$ where $\cE$ is the algebra generated
    by all $R$-module endomorphisms $x\mapsto ax$ and $x\mapsto xa$ of $\cA$.
  \item Since the number of all $R$-submodules of $R^d$ of index $n$ is bounded
    polynomially as a function of $n$,
    each of the formal Dirichlet series above defines an analytic function in
    some complex right half-plane, a property that is in fact enjoyed by a much
    larger class of rings $R$, see \cite{Seg97}.
  \end{enumerate}
\end{rems}

By a \emph{prime} of a number field, we always mean a non-zero prime ideal of
its ring of integers.  Using primary decomposition over Dedekind domains, the
Euler product factorisations in \cite{GSS88,Sol77} take the following form in
the present setting.

\begin{lemma}
  Let $\fo$ be the ring of integers of a number field $k$.  Let $\cA$, $M$, and
  $\cE$ be as in Definition~\ref{d:subzeta} with $R = \fo$.  Then
  $\zeta_{\cA}(s) = \prod_{\fp} \zeta_{\cA\otimes_{\fo}\fo_{\fp}}(s)$ and
  $\zeta_{\cE\acts M}(s) = \prod_{\fp}\zeta_{\cE \otimes_{\fo} \fo_{\fp} \acts M
    \otimes_{\fo} \fo_{\fp}}(s)$, the products being taken over the primes of
  $k$.  \qed
\end{lemma}

Here, $\fo_{\fp}$ denotes the completion of the localisation of $\fo$ at $\fp$
or, equivalently, the valuation ring of the $\fp$-adic completion $k_{\fp}$ of
$k$, and we regarded $\cA\otimes_{\fo} \fo_{\fp}$ and
$\cE\otimes_{\fo}\fo_{\fp}$ as $\fo_{\fp}$-algebras and $M\otimes_{\fo}
\fo_{\fp}$ as an $\fo_{\fp}$-module.

\subsection{Zeta functions of nilpotent groups: linearisation}

The \emph{normal subgroup zeta function} $\zeta_G^\normal(s)$ of a finitely
generated torsion-free nilpotent group $G$ is defined similar to
$\zeta_G^{\phantom o}(s)$ from the introduction by enumerating normal subgroups
of finite index of $G$. The function $\zeta_G^\normal(s)$ shares many properties
with $\zeta_G^{\phantom o}(s)$, in particular the existence of canonical Euler
products and rationality of local factors.

\begin{thm}[{\cite[\S 4]{GSS88}}]
  \label{thm:nilpotent}
  Let $G$ be a finitely generated torsion-free nilpotent group and let $\fL(G)$
  be the nilpotent $\QQ$-Lie algebra attached to $G$ via the Mal'cev
  correspondence.  Let $\cL\subset \fL(G)$ be an arbitrary Lie subalgebra over
  $\ZZ$ which is finitely generated as a $\ZZ$-module and which spans $\fL(G)$
  over $\QQ$.  Then $\zeta_{G,p}^{\phantom o}(s) = \zeta_{\cL,p}^{\phantom
    o}(s)$ and $\zeta_{G,p}^\normal(s) = \zeta_{\cL,p}^\normal(s)$ for almost
  all primes $p$.
\end{thm}

Throughout this article, we will ignore finite sets of exceptional primes.  From
our point of view, the study of local zeta functions of nilpotent groups is thus
subsumed by the cases of algebras and modules.  Passing from nilpotent groups to
$\ZZ$-algebras allows us to perform base changes to finite extensions of $\ZZ$
and $\ZZ_p$.  This entirely natural operation on the level of algebras and
modules will play a central role throughout this article.

\subsection{Local zeta functions and $\fP$-adic integration}
\label{ss:coneint}

The following notational conventions will be used at various points in this
article.
\begin{notation}
  \label{not:local}
  For a a non-Archimedean local field $K$, let $\fO$ be the valuation ring and
  $q$ be the residue field size of $K$.  Let $\fP$ be the maximal ideal of
  $\fO$.  Choose a uniformiser $\pi$ and let $\nu$ be the valuation on $K$ with
  $\nu(\pi) = 1$.  Let $\abs{\dtimes}$ be the absolute value on $K$ with
  $\abs{x} = q^{-\nu(x)}$.  For a non-empty set or family $M \subset K$, we
  write $\norm{M} = \sup(\abs{x}:x\in M)$.  Let $\mu$ be the Haar measure on $K$
  normalised such that $\mu\left(\fO\right) = 1$; we use the same symbol to
  denote the product measure on $K^n$.
\end{notation}

We often use subscripts to denote base change.  Let $\Tr_d(R)$ be the ring of
upper triangular $d\times d$-matrices over a ring $R$.
\begin{thm}[{\cite[\S 5]{dSG00}}]
  \label{thm:coneint}
  Let $k$ be a number field with ring of integers $\fo$.
  \begin{enumerate}
  \item
    \label{thm:coneint1}
    Let $\cA$ be a non-associative $\fo$-algebra which is free of rank $d$ as an
    $\fo$-module.  Then there exists a finite family $\ff = (f_i)_{i\in I}$ of
    non-zero Laurent polynomials $f_i \in \fo[X_{ij}^{\pm 1} : 1\le i \le j\le
    d]$ with the following property: if $K \supset k$ is a $p$-adic field, then
    \begin{equation}
      \label{eq:coneint}
      \zeta_{\cA_{\fO}}(s) =
      (1-q^{-1})^{-d}
      \int_{\{ \xx\in \Tr_d(\fO) : \norm{\ff(\xx)}\le 1\}}
      \abs{x_{11}}^{s-1}\dotsb \abs{x_{dd}}^{s-d} \dd\mu(\xx),
    \end{equation}
    where we identified $\Tr_d(K) \approx K^{\binom{d+1}2}$ and followed the
    conventions in Notation~\ref{not:local}.
  \item
    \label{thm:coneint2}
    Let $\cE$ be an associative $\fo$-subalgebra of $\End_{\fo}(M)$, where $M$
    is a free $\fo$-module of rank $d$.  Then there are Laurent polynomials as
    in (\ref{thm:coneint1}) (but usually different from those) such that the
    conclusion of (\ref{thm:coneint1}) holds for $\zeta_{\cE_{\fO} \acts
      M_{\fO}}(s)$ in place of $\zeta_{\cA_{\fO}}(s)$.
  \end{enumerate}
\end{thm}
\begin{rems}
  \label{r:coneint}
  \quad
  \begin{enumerate}
  \item Strictly speaking, the proof given in \cite{dSG00} (which relies on
    \cite[Prop.\ 3.1]{GSS88}) only covers the case that $\cA$ is a Lie ring over
    $\ZZ$ and $K = \QQ_p$; however, the same arguments carry over, essentially
    verbatim, to the present setting.
  \item
    \label{r:coneint2}
    For the benefit of the reader who wishes to work through the explicit
    examples in \S\ref{s:app}, we recall how to find a suitable $\ff$ in
    (\ref{thm:coneint1}), part (\ref{thm:coneint2}) being similar.  Thus, by
    choosing a basis, we may identify $\cA$ and $\fo^d$ as $\fo$-modules.  Let
    $R := \fo[X_{ij} :1\le i\le j \le d]$ and let $C:=[X_{ij}]_{i\le j} \in
    \Tr_d(R)$ with rows $C_1,\dotsc,C_d$.  Now let $\ff$ consist of all non-zero
    components of all $d$-tuples $\det(C)^{-1} (C_mC_n) \adj(C)$ for $1\le
    m,n\le d$, where the products $C_mC_n$ are taken in $\cA_R$ and $\adj(C)$
    denotes the adjugate matrix of $C$.  We note that $\ff$ in general very much
    depends on our choice of an $\fo$-basis of $\cA$.  Moreover, the description
    of $\ff$ given here is often highly redundant.  For actual computations such
    as those in \S\ref{s:app}, we tacitly apply some elementary simplifications.
  \end{enumerate}
\end{rems}

Keeping the notation of Theorem~\ref{thm:coneint}, by adapting and extending a
result due to Denef \cite[Thm\ 3.1]{Den87}, du Sautoy and Grunewald \cite[\S\S
2--3]{dSG00} (see also \S\ref{ss:topsub}) derived an explicit formula for
$\zeta_{\cA_{\fO}}(s)$ or $\zeta_{\cA_{\fO}\acts M_{\fO}}(s)$, respectively, in
terms of numerical data extracted from an embedded resolution of singularities.
As in the case of Igusa's local zeta function, this approach is primarily of
theoretical interest due to the infeasibility of constructing a resolution of
singularities in practice.  Our first major goal, to be accomplished in
Theorem~\ref{thm:evaluate}, is to find more practical means of computing
integrals such as those in Theorem~\ref{thm:coneint} under suitable
non-degeneracy assumptions.

\section{Zeta functions associated with cones and polytopes}
\label{s:polyzeta}

Let $K$ be a $p$-adic field with associated objects as in
Notation~\ref{not:local}.  For a polytope $\cP \subset \RR^n$ and $\xx\in
(K^\times)^n$, we write $\cP(\xx) = \{ \xx^\alpha : \alpha \in \cP \cap
\ZZ^n\}$, where $\xx^{\alpha} = x_1^{\alpha_1}\dotsb x_{n\phantom
  1\!\!}^{\alpha_n}$.  For $\xx\in K^n$, write $\nu(\xx) :=
(\nu(x_1),\dotsc,\nu(x_n))$.  Note that $\nu(\xx^{\omega}) = \bil {\nu(\xx)}
\omega$ for $\omega \in \ZZ^n$ and $\xx \in (K^\times)^n$, where
$\bil{\dtimes\,}{\dtimes\,}$ denotes the standard inner product.  Let
$\cC_0\subset \Orth^n$ be a half-open rational cone and let
$\cP_1,\dotsc,\cP_m\subset\Orth^n$ be lattice polytopes.  In this section, we
give an explicit convex-geometric formula (see
Proposition~\ref{prop:monomial_integral}) for the ``zeta function''
\begin{equation}
  \label{eq:cone_polytope_zeta}
  \int_{\left\{ \xx \in K^n : \nu(\xx)\in \cC_0 \right\}}
  \norm{\cP_1(\xx)}^{s_1} \dotsb \norm{\cP_m(\xx)}^{s_m} \dd\mu(\xx),
\end{equation}
where $s_1,\dotsc,s_m\in \CC$ with $\Real(s_j) \ge 0$.  Zeta functions of this
shape constitute the building blocks of our explicit formulae for more
complicated $\fP$-adic integrals in \S\ref{s:padic_eval}.  They also generalise
Igusa-type zeta functions associated with monomial ideals previously considered
in the literature.  Indeed, if $\cC_0 = \Orth^n$, $m = 1$, and $\cP_1 =
\conv(\alpha_1,\dotsc,\alpha_r)$ for $\alpha_1,\dotsc,\alpha_r\in \NN_0^n$, then
(\ref{eq:monomial_integral}) coincides with the zeta function associated with
the ideal, $I$ say, generated by $\XX^{\alpha_1},\dotsc,\XX^{\alpha_r}$ over
$\fo$ in the sense of \cite{HMY07}.  In this special case, our
Proposition~\ref{prop:monomial_integral} is analogous to \cite[Prop.\
2.1]{HMY07}, the main difference being that in \cite{HMY07}, the polytope
$\cP_1$ is replaced by the polyhedron $\conv\bigl(\alpha \in \NN_0^n: \XX^\alpha
\in I\bigr) \subset \Orth^n$; cf.\ \S\ref{ss:ndlit}.

\subsection{Background on cones and their generating functions}
\label{ss:cones}

We summarise some standard material; for details see e.g.\
\cite{Bar08,Bar02,CLS11,Zie95}.

\paragraph{Cones.}
By a \emph{cone} in $\RR^n$ we mean a polyhedral cone, i.e.\ a finite
intersection of closed linear half-spaces in $\RR^n$.  If such a collection of
half-spaces can be chosen to be defined over $\QQ$, then the cone is
\emph{rational}.  Equivalently, cones in $\RR^n$ are precisely the sets of the
form $\cone(P) := \{ \sum\limits_{\varrho \in P} \lambda(\varrho) \varrho :
\lambda(\varrho) \in \Orth \}$ for finite $P \subset \RR^n$.  A cone $\cC
\subset \RR^n$ is rational if and only if $\cC = \cone(P)$ for some finite
$P\subset \ZZ^n$.  The \emph{dual cone} $\cP^* := \{ \omega \in \RR^n :
\bil\alpha\omega \ge 0 \text{ for all } \alpha\in \cP\}$ of a polytope $\cP
\subset \RR^n$ is a cone in our sense.  A cone $\cC \subset \RR^n$ is
\emph{pointed} if it does not contain any non-trivial linear subspace of
$\RR^n$.  The following notion appears less frequently in the literature.  Thus,
a (relatively) \emph{half-open cone} in $\RR^n$ is a set of the form
$\cC\setminus(\cC_1\cup\dotsb\cup\cC_r)$, where $\cC \subset \RR^n$ is a cone
and $\cC_1,\dotsc,\cC_r$ are faces of $\cC$.  Equivalently, half-open cones are
finite intersections of closed linear half-spaces and open ones.  We say that a
half-open cone is rational if it is either empty or if its closure is a rational
cone.  A \emph{relatively open cone} is a non-empty half-open cone which
coincides with the relative interior of its closure within the ambient Euclidean
space.

\paragraph{Polytopes and fans.}
For a non-empty polytope $\cP \subset \RR^n$ and $\omega\in \RR^n$, let
$\face_\omega(\cP)$ denote the face of $\cP$ where the function $\cP \to \RR,
\alpha \mapsto \bil{\alpha}{\omega}$ attains its minimum; we do not regard
$\emptyset$ as a face of $\cP$.  If $\cQ \subset \RR^n$ is another non-empty
polytope, then $\face_\omega(\cP + \cQ) = \face_\omega(\cP) +
\face_\omega(\cQ)$.  The (inner, relatively open) \emph{normal cone} of a face
$\tau \subseteq \cP$ is $\NormalCone_\tau(\cP) = \{ \omega \in \RR^n :
\face_\omega(\cP) = \tau \}$.  We have $n = \dim(\tau) +
\dim(\NormalCone_\tau(\cP))$.  It is well-known that $ \RR^n = \bigcup_{\tau}
\NormalCone_{\tau}(\cP) $ is a partition of $\RR^n$ into relatively open cones.
The set $\{ \overline{\NormalCone_\tau(\cP)} : \tau \text{ is a face of } \cP\}$
constitutes a fan, called the (inner) \emph{normal fan} of $\cP$.

\paragraph{Visibility.}
\label{p:visibility}
The following terminology is non-standard.
\begin{defn}
  \label{d:visible}
  Let $\cC_0\subset \RR^n$ be a half-open cone and let $\cP \subset \RR^n$ be a
  non-empty polytope.  We say that a face $\tau \subset \cP$ is
  \emph{$\cC_0$-visible} if $\NormalCone_\tau(\cP) \cap \cC_0 \not= \emptyset$.
\end{defn}
\begin{lemma}
  \label{lem:visible}
  Let $\cC\subset \RR^n$ be a full-dimensional cone and let $\cP \subset \RR^n$
  be a non-empty polytope.  Write $\interior(\cC)$ for the interior of $\cC$ and
  $\cC^*$ for the dual cone of $\cC$.  Then the $\interior(\cC)$-visible faces
  of $\cP$ are precisely the non-empty compact faces of $\cP + \cC^*$.
\end{lemma}
\begin{proof}
  For $\omega\in \RR^n$, we have $\face_\omega(\cP + \cC^*) = \face_\omega(\cP)
  + \face_\omega(\cC^*)$ so $\face_\omega(\cP + \cC^*)$ is non-empty and compact
  if and only if $\face_\omega(\cC^*) = \{ 0\}$.  By \cite[Ex.\
  1.2.2(a)]{CLS11}, the latter condition is equivalent to $\omega\in
  \interior(\cC)$.
\end{proof}

\paragraph{Generating functions.}
Let $\lambda_1,\dotsc,\lambda_n$ be algebraically independent over $\QQ$ and
write $\bm\lambda =(\lambda_1,\dotsc,\lambda_n)$.  Let $\Torus^n =
\Spec(\ZZ[\lambda_1^{\pm 1},\dotsc,\lambda_n^{\pm 1}])$.  Unless otherwise
mentioned, by a commutative ring or algebra, we always mean an associative,
commutative, and unital ring or algebra, respectively.  Given a commutative ring
$R$, we identify $\TT^n(R) = (R^\times)^n$.

\begin{defn}
  \label{d:ConeRegion}
  For a rational cone $\cC \subset \RR^n$, let
  $$\ConeRegion(\cC) := \{ \xx\in \Torus^n(\CC) : \abs{\xx^\omega} < 1 \text{ for }
  0 \not= \omega \in \cC \cap \ZZ^n\}.$$
\end{defn}
Note that if $\cC = \cone(\varrho_1,\dotsc,\varrho_r)$ for $0\not= \varrho_i\in
\ZZ^n$, then each $\omega \in \cC \cap \ZZ^n$ is a $\QQ_{\ge 0}$-linear
combination of the $\varrho_i$ whence $\ConeRegion(\cC) = \{ \xx \in
\Torus^n(\CC) : \abs{\xx^{\varrho_i}} < 1 \text{ for } i=1,\dotsc,r\}$ follows.

The next result is well-known.
\begin{thm}[{Cf.\ \cite[Ch.\ 13]{Bar08}}]
  \label{thm:genfun}
  Let $\cC\subset \RR^n$ be a pointed rational cone and let $\cC_0 \subset \cC$
  be a half-open cone with $\overline{\cC_0} = \cC$.
  \begin{enumerate}
  \item $\ConeRegion(\cC)$ is a non-empty open subset of $\Torus^n(\CC)$.
  \item $\sum_{\omega\in {\cC_0} \cap \ZZ^n} \xx^\omega$ converges absolutely
    and compactly on $\ConeRegion(\cC)$.
  \item There exists a unique $\genfun{\cC_0}\in \QQ(\bm \lambda)$ with
    $\genfun{\cC_0}(\xx) = \sum_{\omega\in \cC_0\cap \ZZ^n} \xx^\omega$ for
    $\xx\in \ConeRegion(\cC)$.
  \end{enumerate}
\end{thm}

Using the inclusion-exclusion principle, Theorem~\ref{thm:genfun} reduces to the
case $\cC_0 = \cC$, the situation usually considered in the literature.  Let
$\QQ[\cC\cap\ZZ^n]$ denote the $\QQ$-subalgebra of $\QQ(\bm\lambda)$ spanned
(over $\QQ$) by the Laurent monomials $\bm\lambda^\omega$ for $\omega\in \cC
\cap \ZZ^n$.  Then, as is well-known, $\genfun{\cC}(\bm\lambda)$ can be written
as a finite sum of rational functions of the form
$f(\bm\lambda)/\prod_{j=1}^d(1-\bm\lambda^{\alpha_j})$, where $f(\bm\lambda)\in
\QQ[\cC\cap\ZZ^n]$, $d\le \dim(\cC)$, and $0\not= \alpha_j\in \cC\cap \ZZ^n$.
For the sake of completeness, we set $\genfun{\emptyset}(\bm\lambda) = 0$.

\subsection{Monomial substitutions}
\label{ss:monomial_sub}

Let $A\in \Mat_{n\times m}(\ZZ)$.  We also let $A$ denote the induced linear map
$\RR^n \to \RR^m$ acting by right-multiplication on row vectors.  We further
obtain an induced ring homomorphism $(\blank)^A\colon \ZZ[\lambda_1^{\pm
  1},\dotsc,\lambda_n^{\pm 1}] \to \ZZ[\xi_1^{\pm 1},\dotsc,\xi_m^{\pm 1}]$
given by $(\bm\lambda^\alpha)^A = \bm\xi^{\alpha A}$ for $\alpha \in \ZZ^n$ and
an induced morphism $A(\blank)\colon \Torus^m \to \Torus^n$ characterised by
$f(A\yy) = f^{A}(\yy)$ for $f\in \ZZ[\bm\lambda^{\pm 1}]$ and $\yy\in
\Torus^m(R)$, where $R$ is a commutative ring.  Thus, if $A_1,\dotsc,A_n$ denote
the rows of $A$, then $A\yy = (\yy^{A_1},\dotsc,\yy^{A_n})$ for
$\yy\in\Torus^m(R)$.

Let $\cC \subset \RR^n$ be a pointed rational cone.  The image $\cC A \subset
\RR^m$ of $\cC$ under $A$ is then again a pointed rational cone.  Suppose that
$\cC \cap \Ker(A) = \{ 0\}$.  Then $(\blank)^A\colon \QQ[\bm\lambda^{\pm
  1}]\to\QQ[\bm\xi^{\pm 1}]$ extends to a $\QQ$-algebra homomorphism
$(\blank)^A\colon \cB \to \cB'$, where $\cB$ is the $\QQ$-algebra generated by
$\QQ[\bm\lambda^{\pm 1}]$ and all $(1-\bm\lambda^\omega)^{-1}$ with $0\not=
\omega\in \cC\cap \ZZ^n$ and $\cB'$ is generated by $\QQ[\bm\xi^{\pm 1}]$ and
all $(1-\bm\xi^{\omega A})^{-1}$, again for $0\not= \omega\in \cC\cap \ZZ^n$.
In particular, $\genfun{\cC}^A = \genfun{\cC}(\bm\xi^{A_1},\dotsc,\bm\xi^{A_n})$
is a well-defined element of $\cB' \subset \QQ(\bm\xi)$.

\begin{lemma}
  \label{lem:sub}
  Let $\cC \subset \RR^n$ be a pointed rational cone and let $\cC_0\subset \cC$
  be a half-open cone with $\overline{\cC_0} = \cC$.  Let $A\in \Mat_{n\times
    m}(\ZZ)$ with $\cC \cap \Ker(A) = \{ 0\}$.  Then $\sum_{\omega\in
    {\cC_0}\cap \ZZ^n} \yy^{\omega A}$ converges absolutely and compactly on
  $\ConeRegion(\cC A)$ (see Definition~\ref{d:ConeRegion}).  The resulting
  function $\ConeRegion(\cC A)\to\CC$ is given by $\yy \mapsto
  \genfun{\cC_0}^A(\yy)$.
\end{lemma}
\begin{proof}
  If $\cC = \cone(\varrho_1,\dotsc,\varrho_r)$ for $0\not= \varrho_i \in \ZZ^n$,
  then $\cC A = \cone(\varrho_1 A, \dotsc, \varrho_r A)$ and the $\varrho_i A$
  are non-zero too.  As $\yy^{\omega A} = (A\yy)^\omega$ for $\yy\in
  \Torus^m(\CC)$ and $\omega\in \ZZ^n$, we conclude that $A(\blank)\colon
  \Torus^m(\CC) \to \Torus^n(\CC)$ maps $\ConeRegion(\cC A)$ to a subset of
  $\ConeRegion(\cC)$.  Now apply Theorem~\ref{thm:genfun}.
\end{proof}

\subsection{Rational function functions from cones and polytopes}
\label{ss:zeta_cone_polytopes}

Let $\cC_0 \subset \Orth^n$ be a half-open rational cone with closure $\cC =
\overline{\cC_0}$ and let $\cP_1,\dotsc,\cP_m \subset \Orth^n$ be non-empty
lattice polytopes.  Let $\xi_0,\dotsc,\xi_m$ be algebraically independent over
$\QQ$.  We now construct a rational function
$\cZ^{\cC_0,\cP_1,\dotsc,\cP_m}(\xi_0,\dotsc,\xi_m)$, which, as we will see in
Proposition~\ref{prop:monomial_integral}, essentially specialises to
\eqref{eq:cone_polytope_zeta}.

For $\cC_0 = \emptyset$, define
$\cZ^{\cC_0,\cP_1,\dotsc,\cP_m}(\xi_0,\dotsc,\xi_m) := 0$. Henceforth, let
$\cC_0 \not=\emptyset$.  Define $\cP := \cP_1 + \dotsb + \cP_m$ to be the
Minkowski sum of $\cP_1,\dotsc,\cP_m$.  It is well-known that the normal fan of
$\cP$ is precisely the coarsest common refinement of the normal fans of
$\cP_1,\dotsc,\cP_m$ (see \cite[Prop.\ 7.12]{Zie95}).  Specifically, each face
$\tau$ of $\cP$ admits a unique decomposition $\tau = \tau_1 + \dotsb + \tau_m$
for suitable faces $\tau_j \subset \cP_j$ (see e.g.\ \cite[Prop.\ 2.1]{Fuk04})
and
\begin{equation}
  \label{eq:intersection_normal_cones}
  \NormalCone_{\tau}(\cP) = \NormalCone_{\tau_1}(\cP_1) \cap \dotsb \cap \NormalCone_{\tau_m}(\cP_m).
\end{equation}

Write $\one = (1,\dotsc,1)$.  For a vertex $\vv$ of $\tau$, decomposed as $\vv =
\vv_1 + \dotsb + \vv_m$ for vertices $\vv_j$ of $\tau_j$, let $A(\vv) =
[\one^\top, \vv_1^\top, \dotsc, \vv_m^\top] \in \Mat_{n\times(m+1)}(\ZZ)$.  If
$\omega\in \NormalCone_{\tau}(\cP)$, then \eqref{eq:intersection_normal_cones}
shows that $\omega A(\vv) = \omega A(\vv')$ for all vertices $\vv,\vv'\in\tau$;
by continuity, this identity extends to the closure of
${\NormalCone_{\tau}(\cP)}$.  Recall the notation for monomial substitutions
from \S\ref{ss:monomial_sub}.  In view of the remarks following
Theorem~\ref{thm:genfun}, and since $\cC_0 \subset \Orth^n$, we obtain a
well-defined rational function
$\cZ^{\cC_0,\cP_1,\dotsc,\cP_m}_\tau(\xi_0,\dotsc,\xi_m) := \genfun{\cC_0\cap
  \NormalCone_\tau(\cP)}^{A(\vv)}(\xi_0,\dotsc,\xi_m)$ which does not depend on
the choice of a vertex $\vv$ of $\tau$.

\begin{defn}
  \label{def:Z}
  $\cZ_{\phantom\tau}^{\cC_0,\cP_1,\dotsc,\cP_m}(\xi_0,\dotsc,\xi_m) :=
  \sum_{\tau} \cZ^{\cC_0,\cP_1,\dotsc,\cP_m}_\tau(\xi_0,\dotsc,\xi_m)$, the sum
  being taken over the $\cC_0$-visible faces $\tau$ of $\cP$ (see
  Definition~\ref{d:visible}).
\end{defn}

\begin{rem}
  \label{r:Zdenom}
  By construction, $\cZ^{\cC_0,\cP_1,\dotsc,\cP_m}(\xi_0,\dotsc,\xi_m)$ can be
  written over a denominator of the form $\prod_{i=1}^r (1-\xi_0^{a_{i0}}\dotsb
  \xi_m^{a_{im}})$ for $a_{ij}\in \NN_0$ and $a_{i0} > 0$.
\end{rem}

\begin{prop}
  \label{prop:Zseries}
  Let $$U = \{ \zz \!=\! (z_0,\dotsc,z_m) \in \CC^{m+1} : 0 < \abs{z_0} < 1
  \text{ and } 0 < \abs{z_1},\dotsc,\abs{z_m} \le 1 \}.$$ Then, for $\zz\in U$,
  \begin{equation}
    \label{eq:polycone}
    \cZ^{\cC_0,\cP_1,\dotsc,\cP_m}(\zz)
    =
    \sum_{\omega \in \cC_0 \cap \ZZ^n}
    z_0^{\bil\one\omega}
    \prod_{j=1}^m
    z_j^{\min\left( \bil\alpha\omega : \alpha\in \cP_j\right)},
  \end{equation}
  and the series on the right-hand side converges absolutely and compactly on
  $U$.
\end{prop}
\begin{proof}
  Fix a $\cC_0$-visible face $\tau$ of $\cP_1+\dotsb+\cP_r$ and a vertex $\vv =
  \vv_1 + \dotsb + \vv_m$ of $\tau$.  Let $\cC_0^\tau := \cC_0 \cap
  \NormalCone_{\tau}(\cP)$.
  Since $\overline{\cC_0^\tau}A(\vv)$ is contained in the cone, $\cC'$ say,
  spanned by the rows of $A(\vv)$, we have
  $\ConeRegion(\overline{\cC_0^\tau}A(\vv)) \supset \ConeRegion(\cC')$ (see
  Definition~\ref{d:ConeRegion}).  Moreover, $\ConeRegion(\cC') \supset U$ since
  $\vv_j \in \Orth^n$ for $j=1,\dotsc,m$.
  By Lemma~\ref{lem:sub}, $\cZ^{\cC_0,\cP_1,\dotsc,\cP_m}_\tau(\zz) =\!\!
  \sum_{\omega \in \cC_0^\tau \cap \ZZ^n} \! z_0^{\bil\one\omega}
  \!\prod_{j=1}^m z_j^{\bil{\vv_j}\omega}$ for $\zz\in U$, the convergence being
  as stated.  The claim follows since $\cC_0 = \bigcup_{\tau} \cC_0^\tau$
  (disjoint) and $\bil{\vv_j}\omega = \min(\bil\alpha\omega \!:\! \alpha\in
  \cP_j)$ for $\omega \in \cC_0^\tau$.
\end{proof}

\subsection{Computing \texorpdfstring{$\fP$}{P}-adic integrals associated with cones and polytopes}
Let $K$ be a $p$-adic field with associated objects as in
Notation~\ref{not:local}.

\begin{prop}
  \label{prop:monomial_integral}
  Let $\cC_0\subset \Orth^n$ be a half-open rational cone and let
  $\cP_1,\dotsc,\cP_m \subset \Orth^n$ be non-empty lattice polytopes.  Let
  $\cZ^{\cC_0,\cP_1,\dotsc,\cP_m}(\xi_0,\dotsc,\xi_m)$ be as in
  Definition~\ref{def:Z}.
  Then
  \begin{equation}
    \label{eq:monomial_integral}
    \int_{\{\xx \in K^n : \nu(\xx)\in \cC_0 \}}
    \!
    \prod_{j=1}^m \norm{\cP_j(\xx)}^{s_j}
    \!\dd\mu(\xx)
    =
    (1-q^{-1})^n \dtimes
    \cZ^{\cC_0,\cP_1,\dotsc,\cP_m}(q^{-1},q^{-s_1},\dotsc,q^{-s_m})
  \end{equation}
  for $s_1,\dotsc,s_m \in \CC$ with $\Real(s_j) \ge 0$.
\end{prop}
\begin{proof}
  For $\omega \in \ZZ^n$, write $\Torus^n_\omega(K) := \{ \xx\in
  \Torus^n_{\phantom o}(K) : \nu(\xx) = \omega\}$.  If $\xx\in
  \Torus^n_\omega(K)$, then $\norm{\cP_j(\xx)} =
  q^{-\min\bigl(\bil\alpha\omega:\alpha\in \cP_j\bigr)}$.
  Note that $F_\omega\colon \Torus^n(\fO) \xto{\approx} \Torus^n_\omega(K), \uu
  \mapsto \pi^\omega \uu := (\pi^{\omega_1} u_1,\dotsc,\pi^{\omega_n} u_n) $
  satisfies $\abs{\det\!\left( F_\omega'(\uu)\right)} = q^{-\bil\one\omega}$ for
  all $\uu\in \Torus^n(\fO)$.  Since $\mu(\Torus^n(\fO)) = (1-q^{-1})^n$, using
  the $\fP$-adic change of variables formula \cite[Prop.\ 7.4.1]{Igu00}, we
  obtain
  \[
  \int_{\{\xx \in K^n : \nu(\xx)\in \cC_0 \}} \prod_{j=1}^m
  \norm{\cP_j(\xx)}^{s_j}\dd\mu(\xx) = (1-q^{-1})^n \sum\limits_{\omega\in \cC_0
    \cap \ZZ^n} q^{-\bil\one\omega - \sum_{j=1}^m s_j \min\bigl(\bil\alpha\omega
    : \alpha\in\cP_j\bigr)}
  \]
  whence the claim follows from Proposition~\ref{prop:Zseries}.
\end{proof}

\begin{rem}
  Note that Proposition~\ref{prop:monomial_integral} behaves well under base
  change. Namely, replacing $K$ by a finite extension $K'\supset K$ simply
  amounts to replacing $q$ by $q^f$ on the right-hand side of
  \eqref{eq:monomial_integral}, where $f$ is the residue class degree of $K'/K$.
\end{rem}

\section{Non-degeneracy I: computing \texorpdfstring{$\fP$}{P}-adic integrals}
\label{s:padic_eval}

In this section, we study a class of $\fP$-adic integrals which includes
Igusa-type zeta functions associated with polynomial mappings (see \cite{VZG08})
as well as the integrals in Theorem~\ref{thm:coneint}.  In
Theorem~\ref{thm:evaluate}, we give an explicit formula for the integrals
considered under suitable non-degeneracy assumptions.  In addition to being
applicable to the study of zeta functions of groups, algebras, and modules, our
findings generalise various existing applications of non-degeneracy in the
context of Igusa's local zeta function.

\subsection{Newton polytopes and initial forms}
\label{ss:newton}

The following is mostly folklore.  Let $R$ be a commutative ring and let $\XX =
(X_1,\dotsc,X_n)$ be indeterminates over $R$.  Let $f = \sum_{\alpha\in \ZZ^n}
c_\alpha \XX^\alpha \in R[\XX^{\pm 1}]$, where $c_\alpha\in R$, almost all of
which are zero.  The \emph{support} of $f$ is $\supp(f) := \{ \alpha \in \ZZ^n :
c_\alpha \not= 0\}$.  The \emph{Newton polytope} $\Newton(f)$ of $f$ is the
convex hull of $\supp(f)$ within $\RR^n$.  If $R$ is a domain, then $\Newton(fg)
= \Newton(f) + \Newton(g)$ for $f,g\in R[\XX^{\pm 1}]$ (see
\cite[Lemma~2.2]{Stu96}).  For $\omega \in \RR^n$, define the \emph{initial
  form} $\init_\omega(f)$ of $f$ in the direction $\omega$ to be the sum of all
those monomials $c_\alpha \XX^\alpha$ (with $\alpha \in \supp(f)$) where $\bil
\alpha \omega$ attains its minimum.  Let $f\not= 0$.  We then have
$\face_\omega(\Newton(f)) = \Newton(\init_\omega(f))$ for $\omega\in \RR^n$
\cite[(2.5)]{Stu96}.  The equivalence classes of $\sim$ defined on $\RR^n$ via
$\omega \sim \omega'$ if and only if $\init_{\omega}(f) = \init_{\omega'}(f)$
are precisely the normal cones of the faces of $\Newton(f)$ as defined in
\S\ref{ss:cones}.

Let $Y$ be another variable over $R$.  For $\omega \in \ZZ^n$, we write
$Y^\omega \XX = (Y^{\omega_1} X_1, \dotsc, Y^{\omega_n} X_n)$.  Let $f\in
R[\XX^{\pm 1}]$ be non-zero and let $\tau$ be a face of $\Newton(f)$.  Write $f
= \sum_{\alpha\in \ZZ^n} c_\alpha \XX^\alpha$ as above and let $\omega \in
\NormalCone_\tau(\Newton(f)) \cap \ZZ^n$.  Choose a vertex $\Lambda(\tau)$ of
$\tau$.  We may write
\begin{equation*}
  f(Y^\omega \XX) =
  Y^{\bil{\Lambda(\tau)}\omega} \dtimes
  \Bigl(
  \init_\omega(f) +
  Y \dtimes {\Bigl(
    \sum_{\alpha \text{ s.t.\ }\! \bil{\alpha-\Lambda(\tau)}\omega > 0}  
    Y^{\bil{\alpha-\Lambda(\tau)}\omega -1}
    c_\alpha \XX^\alpha
    \Bigr)}
  \Bigr).
\end{equation*}
Note that $\bil{\Lambda(\tau)} \omega$ is independent of the choice of
$\Lambda(\tau)$.
\begin{notation}
  \label{not:fomega}
  $f_\omega(\XX,Y) := Y^{-\bil{\Lambda(\tau)}\omega} \dtimes f(Y^\omega \XX) \in
  R[\XX^{\pm 1}, Y]$.
\end{notation}

\subsection{Non-degeneracy}
\label{ss:nd}

Let $k$ be a field and let $X_1,\dotsc,X_n$ be algebraically independent over
$k$.  Write $\XX = (X_1,\dotsc,X_n)$.  Let $\ff = (f_i)_{i\in I} $ be a finite
family of non-zero elements of $k[\XX^{\pm 1}]$ and let $\cC_0 \subset \RR^n$ be
a half-open rational cone.  Let $\cN := \Newton\bigl(\prod_{i\in I}f_i\bigr)
=\sum_{i\in I}\Newton(f_i)$.  For a face $\tau \subset \cN$, let $\cC_0^\tau :=
\cC_0^{\phantom 1} \cap \NormalCone_{\tau}(\cN)$.  By
\eqref{eq:intersection_normal_cones} and \S\ref{ss:newton}, we may unambiguously
define $f_i^\tau := \init_{\omega}(f_i)$ for $i\in I$ and an arbitrary $\omega
\in \NormalCone_\tau(\cN)$.  Let $\bar k$ be an algebraic closure of $k$.

\begin{defns}
  \label{d:nd}
  \quad
  \begin{enumerate}
  \item
    \label{d:nd1}
    We say that $\ff$ is \emph{non-degenerate relative to $\cC_0$} if the
    following holds:

    for all $\cC_0$-visible faces $\tau \subset \cN$ and all $J\subset I$, if
    $\uu\in \Torus^n(\bar k)$ satisfies $f_j^\tau(\uu) = 0$ for all $j\in J$,
    then the Jacobian matrix $\begin{bmatrix}\frac{\partial
        f_j^\tau\!(\uu)}{\partial X_i}\end{bmatrix}_{i=1,\dotsc,n; j\in J}$ has
    rank $\noof J$.
  \item
    \label{d:nd2}
    We say that $\ff$ is \emph{globally non-degenerate} if it is non-degenerate
    relative to $\RR^n$.
  \end{enumerate}
\end{defns}

\begin{rems}
  \label{r:nd}
  \quad
  \begin{enumerate}
  \item Non-degeneracy of $\ff$ relative to $\cC_0$ is preserved under extension
    of $k$.
  \item
    \label{r:nd2}
    Let $f\in k[\XX^{\pm 1}]$ and write $f' = \left(\frac{\partial f}{\partial
        X_1},\dotsc,\frac{\partial f}{\partial X_n}\right)$.  Then $(\XX^\gamma
    f)' \equiv \XX^\gamma \dtimes f'\bmod f$ for $\gamma \in \ZZ^n$, the
    congruence being understood componentwise and within $k[\XX^{\pm 1}]$.
    In particular, whether $\ff$ is non-degenerate relative to $\cC_0$ or not is
    invariant under rescaling of the elements of $\ff$ by Laurent monomials.
  \item
    \label{r:nd3}
    Let $k$ be a number field with ring of integers $\fo$.  By the (weak)
    Nullstellensatz, if $\ff$ is non-degenerate relative to $\cC_0$ over $k$,
    then, for almost all primes $\fp$ of $k$, the reduction of $\ff$ modulo
    $\fp$ (that is, the image of $\ff$ under the natural map $\fo_{\fp}[\XX^{\pm
      1}] \to (\fo/\fp)[\XX^{\pm 1}]$) is non-degenerate relative to $\cC_0$
    over $\fo/\fp$.
  \item
    \label{r:nd4}
    If all polynomials in $\ff$ are homogeneous and $(1,\dotsc,1)$ is an
    interior point of $\cC_0$ (which thus needs to be full-dimensional), then
    every face of $\cN$ is $\cC_0$-visible whence $\ff$ is non-degenerate
    relative to $\cC_0$ if and only if it is globally non-degenerate.
  \end{enumerate}
\end{rems}

\begin{rem}
  \label{r:famset}
  In most cases below, we will consider sets instead of families of
  polynomials. When we speak of non-degeneracy of a finite set $\ff \subset
  k[\XX^{\pm 1}]$ (with $0\not\in \ff$), we refer to properties of the family
  $(f)_{f\in \ff}$.
\end{rem}

Khovanskii \cite{Kho77,Kho78}, Kushnirenko \cite{Kus76}, Varchenko \cite{Var76},
and others investigated complex varieties defined by suitably non-degenerate
systems of polynomials (see e.g.\ Theorem~\ref{thm:bkk}).  We may rephrase
Khovanskii's notion of non-degeneracy \cite[\S 2]{Kho77} as follows.

\begin{defn}
  \label{d:capdegen}
  We say $\ff = (f_i)_{i\in I}$ is \emph{\Classically-non-degenerate} if the
  following holds: for every face $\tau \subset \cN$ and $\uu\in \Torus^n(\bar
  k)$, if $f_i^\tau(\uu) = 0$ for all $i\in I$, then the Jacobian matrix
  $\begin{bmatrix}\frac{\partial f_j^\tau\!(\uu)}{\partial
      X_i}\end{bmatrix}_{i=1,\dotsc,n;j\in I}$ has rank $\noof I$.
\end{defn}
Hence, using the conventions from Remark~\ref{r:famset}, a finite set $\ff
\subset k[\XX^{\pm 1}]$ (with $0\not\in \ff$) is globally non-degenerate in our
sense if and only if each subset $\bm g \subset \ff$ is
\Classically-non-degenerate.

\subsection{Computing non-degenerate $\fP$-adic integrals}
\label{ss:eval}

Let $k$ be a number field with ring of integers $\fo$.  As before, write $\XX =
(X_1,\dotsc,X_n)$.  Let $\cC_0 \subset \Orth^n$ be a half-open rational cone and
let $\ff_0 \subset k[\XX^{\pm 1}]$ and $\ff_1,\dotsc,\ff_m \subset k[\XX]$ be
non-empty finite sets of non-zero (Laurent) polynomials.

\begin{defn}
  \label{d:Zeta}
  For a $p$-adic field $K \supset k$ with associated data as in
  Notation~\ref{not:local}, let
  \begin{equation*}
    \Zeta^{\cC_0, \ff_0,\dotsc,\ff_m}_K(s_1,\dotsc,s_m)
    :=
    \!\!\!\!
    \int_{\{\xx\in \Torus^n(K): \nu(\xx)\in \cC_0, \norm{\ff_0(\xx)} \le 1\}}
    \!\!\!\!
    \!\!\!\!
    \!\!\!\!
    \!\!\!\!
    \norm{\ff_1(\xx)}^{s_1} \dotsb \norm{\ff_m(\xx)}^{s_m} \dd\mu(\xx),
  \end{equation*}
  where $s_1,\dotsc,s_m\in \CC$ with $\Real(s_j)\ge 0$.
\end{defn}

In \cite{dSG00}, integrals of the form \eqref{eq:coneint} were studied within
the framework of {cone integrals} introduced there.  As we will explain in
Remark~\ref{r:eval_and_algmod},
$\Zeta^{\cC_0,\ff_0,\dotsc,\ff_m}_K(s_1,\dotsc,s_m)$ specialises to the integral
in \eqref{eq:coneint}.  (More generally, every cone integral with monomial
left-hand sides of divisibility conditions arises as a specialisation of an
integral $\Zeta^{\cC_0, \ff_0,\dotsc,\ff_m}_K(s_1,\dotsc,s_m)$.)  Furthermore,
taking $\cC_0 = \Orth^n$, $m = 1$, and $\ff_0 = \{ 1\}$ in
Definition~\ref{d:Zeta}, we recover the Igusa-type integral $\int_{\fO^n}
\norm{\ff_1(\xx)}^s\dd\mu(\xx)$ considered in \cite{VZG08}.

Under the assumption that the set $\ff := \ff_0 \cup \dotsb \cup\ff_m$ in
Definition~\ref{d:Zeta} is non-degenerate relative to $\cC_0$, in this
subsection, we derive an explicit formula for
$\Zeta^{\cC_0,\ff_0,\dotsc,\ff_m}_K(s_1,\dotsc,s_m)$ (see
Theorem~\ref{thm:evaluate}) which is valid for all $p$-adic fields $K\supset k$
such that the associated prime $\fP \cap \fo$ of $k$ does not belong to some
finite exceptional set (depending on $\ff$ and $\cC_0$).  Since we are willing
to ignore finitely many primes of $k$, we may assume that $\ff \subset
\fo[\XX^{\pm 1}]$.

\begin{rem}
  \label{rem:principal}
  Although we will not need this in the sequel, we note that it is possible to
  produce an ``explicit formula'' (in Denef's sense \cite[\S 3]{Den91a}) for
  $\Zeta^{\cC_0, \ff_0,\dotsc,\ff_m}_K(s_1,\dotsc,s_m)$ in terms of a
  principalisation of ideals over $k$, cf.\ \cite[Prop.\ 4.1]{AKOV13} where a
  similar class of integrals is studied using techniques going back to
  \cite{VZG08,Vol10}.  In practice, finding a principalisation of ideals is is
  closely related to finding an embedded resolution of singularities (see
  \cite[Prop.\ 2.5.1]{Wlo05}) and thus equally impractical, in general.
\end{rem}

\begin{rem}
  The role played by the ambient half-open cone $\cC_0$ in
  Definition~\ref{d:Zeta} might seem artificial.  Namely, if $\cC_0$ is closed,
  then by suitably modifying $\ff_0$, we may reduce the computation of
  $\Zeta^{\cC_0, \ff_0,\dotsc,\ff_m}_K(s_1,\dotsc,s_m)$ to the case $\cC_0 =
  \Orth^n$; the case of a general $\cC_0$ then follows by the
  inclusion-exclusion principle.  However, the more practically-minded
  extensions in \cite{topzeta2} of the techniques described in the present
  article rely on successive refinements of the partition $\cC_0 = \bigcup_\tau
  \cC_0^\tau$ used here and the author found these to be most conveniently
  expressed in terms of partitions of $\cC_0$ itself.
\end{rem}

\paragraph{Setup: associated cones, polytopes, and sets.}
\label{p:setup}
As in \S\ref{ss:nd}, let $\cN := \Newton(\prod\ff) = \sum_{f\in \ff} \Newton(f)$
and $f^\tau := \init_\omega(f)$ for $f\in \ff$, a $\cC_0$-visible face $\tau
\subset \cN$, and an arbitrary $\omega \in \cC_0^\tau := \cC_0 \cap
\NormalCone_\tau(\cN)$.  Given $\tau$, let $\tau = \sum_{f\in \ff}\tau(f)$ be
the decomposition of $\tau$ into faces $\tau(f) \subset \Newton(f)$.  Given $f$
and $\tau$, we choose, once and for all, a vertex $\Lambda(f,\tau)$ of
$\tau(f)$.

\begin{defns}
  \label{d:players}
  Let $\tau \subset\cN$ be a $\cC_0$-visible face (see
  Definition~\ref{d:visible}) and let $\bm g \subset \ff$.
  \begin{enumerate}
  \item
    \label{d:players1}
    For $j = 0,\dotsc,m$, define a lattice polytope
    \begin{align*}
      \cP^\tau_j(\bm g) := \conv\Bigl(\!\Bigl\{
      \bigl(\Lambda(g,\tau),\ee_g\bigr) : g \in \ff_j\cap \bm g \Bigr\} \cup
      \Bigl\{ \bigl(\Lambda(f,\tau),0\bigr) : f\in \ff_j\setminus\bm
      g\Bigr\}\!\Bigr),
    \end{align*}
    within $\RR^n \times \RR^{\bm g}$, where $(\ee_g)_{g\in \bm g}$ is the
    standard basis of $\RR^{\bm g} = \{ (x_g)_{g\in \bm g} : \forall g\in \bm
    g\colon x_g \in \RR\} \approx \RR^{\card{\bm g}}$.
  \item
    \label{d:players2}
    Define a half-open cone
      $$\cC^\tau_0(\bm g) := (\cC^\tau_0 \times \StrictOrth^{\bm g}) \cap
      \cP^\tau_0(\bm g)^*,$$ where $\cP_0^\tau(\bm g)^*$ denotes the dual cone
      of $\cP^\tau_0(\bm g)$ as in \S\ref{ss:cones}.
    \item
      \label{d:players3}
      For a a prime $\fp$ of $k$ and a field extension $\fk$ of $\fo/\fp$, let
      \begin{align*}
        \bar V^\tau_{\bm g}(\fk) & := \{ \uu \in \Torus^n(\fk) : \forall f\in
        \ff\colon f^\tau(\uu) = 0 \iff f\in \bm g \}.
      \end{align*}
    \end{enumerate}
  \end{defns}
  \begin{thm}
    \label{thm:evaluate}
    Let $k$ be a number field with ring of integers $\fo$.  Let $\cC_0 \subset
    \Orth^n$ be a half-open rational cone and let $\ff_0 \subset \fo[\XX^{\pm
      1}]$ and $\ff_1,\dotsc,\ff_m \subset \fo[\XX] = \fo[X_1,\dotsc,X_n]$ be
    non-empty finite sets with $0\not\in \ff := \ff_0 \cup\dotsb \cup \ff_m$.
    Define local zeta functions $\Zeta^{\cC_0,
      \ff_0,\dotsc,\ff_m}_K(s_1,\dotsc,s_m)$ as in Definition~\ref{d:Zeta}.  For
    a $\cC_0$-visible face $\tau$ of $\cN := \Newton(\prod \ff)$ (see
    Definition~\ref{d:visible} and \S\ref{ss:newton}) and a subset $\bm g
    \subset \ff$, define a half-open cone $\cC_0^\tau(\bm g)$, a lattice
    polytope $\cP^\tau_j(\bm g)$, and finite sets $\bar V^\tau_{\bm g}(\fk)$ as
    in Definition~\ref{d:players}(\ref{d:players1})--(\ref{d:players3}).
    Finally, define $\cZ^{\cC^\tau_0(\bm g), \cP^\tau_1(\bm
      g),\dotsc,\cP^\tau_m(\bm g)}(\xi_0,\dotsc,\xi_m) \in
    \QQ(\xi_0,\dotsc,\xi_m)$ as in Definition~\ref{def:Z}.

    Suppose that $\ff$ is non-degenerate relative to $\cC_0$ over $k$ in the
    sense of Definition~\ref{d:nd}(\ref{d:nd1}).  Then there exists a finite set
    $S$ of primes of $k$ such that the following holds: if $K$ is a
    non-Archimedean local field extending $k$ such that the associated prime
    ideal $\fP$ of its valuation ring $\fO$ satisfies $\fP \cap \fo \not\in S$,
    then, writing $q = \card{\fO/\fP}$, we have { \small
      \[
      \Zeta^{\cC_0, \ff_0,\dotsc,\ff_m}_K(s_1,\dotsc,s_m) = \sum_{\substack{\bm
          g \subset \ff,\\\tau\subset \cN}} \noof{\bar V^\tau_{\bm g}(\fO/\fP)}
      \dtimes \frac{(q-1)^{\card{\bm g}}}{q^{n}} \dtimes \cZ^{\cC^\tau_0(\bm g),
        \cP^\tau_1(\bm g),\dotsc,\cP^\tau_m(\bm g)}(q^{-1},
      q^{-s_1},\dotsc,q^{-s_m})
      \]
    } for all $s_1,\dotsc,s_m\in \CC$ with $\Real(s_j) \ge 0$ and $\tau$ ranging
    over the $\cC_0$-visible faces of $\cN$.
  \end{thm}
  \begin{rem*}
    Following \cite{DH01}, the established method of computing $\fP$-adic
    integrals associated with non-degenerate polynomials in the literature (see
    \S\ref{ss:ndlit}) is to express them as countable sums of integrals over
    copies of $\Torus^n(\fO)$, each of which can then computed using Hensel's
    lemma.  Our proof of Theorem~\ref{thm:evaluate} proceeds along the same
    lines.
  \end{rem*}

\begin{proof}[Proof of Theorem~\ref{thm:evaluate}]
  \quad
  \begin{enumerate}
  \item {Restrictions on the local field:}

    Choose $\gamma \in \ZZ^n$ such that $\tilde f := \XX^\gamma f$ belongs to
    $\fo[\XX]$ for all $f\in \bm f$.  By Remark~\ref{r:nd}(\ref{r:nd2}), $\tilde
    \ff := \{\tilde f: f\in \ff\}$ is non-degenerate relative to $\cC_0$.
    According to Remark~\ref{r:nd}(\ref{r:nd3}), for almost all primes $\fp$ of
    $k$, non-degeneracy of $\tilde \ff$ relative to $\cC_0$ is preserved under
    reduction modulo $\fp$.
    We assume that $\fP \cap \fo$ is among these primes.
  \item {Breaking up the integral:}
   
    Write $T(\bar\uu)$ for the fibre of $\bar\uu \in \Torus^n(\fO/\fP)$ under
    the natural map $\Torus^n(\fO) \to \Torus^n(\fO/\fP)$.
    By decomposing $\Torus^n(K)$ into subsets $\Torus^n_\omega(K) \approx
    \Torus^n(\fO)$ as in the proof of Proposition~\ref{prop:monomial_integral}
    and by further decomposing $\Torus^n(\fO) = \bigcup_{\bar u\in
      \Torus^n(\fO/\fP)}T(\bar u)$, we obtain
    \begin{equation}
      \label{eq:breakup}
      \Zeta^{\cC_0, \ff_0,\dotsc,\ff_m}_K(s_1,\dotsc,s_m) =
      \sum_\tau
      \sum_{\bm g \subset \ff}
      \sum_{\bar\uu\in \bar V^\tau_{\bm g}(\fO/\fP)}
      \sum_{\omega\in \cC_0^\tau \cap \ZZ^n}
      q^{-\bil\one\omega}
      I_\omega(\bar\uu),
    \end{equation}
    where $I_\omega(\bar\uu) := \int_{\{\ww\in T(\bar\uu) :
      \norm{\ff_0(\pi^\omega \ww)} \le 1\}} \prod_{j=1}^m \norm{\ff_j(\pi^\omega
      \ww)}^{s_j} \dd\mu(\ww)$.
  \item {Computing the pieces:}

    We write $f_\omega$ and $\tilde f_\omega$ for $f_\omega(\XX,\pi)$ and
    $\tilde f_\omega(\XX,\pi)$ (see Notation~\ref{not:fomega}), respectively;
    note that $\init_\omega(\tilde f) = \XX^{\gamma} \init_\omega(f)$ and
    $\tilde f_\omega = \XX^{\gamma} f_\omega$.
    For $\omega\in \cC_0^\tau \cap \ZZ^n$, $\uu\in \Torus^n(\fO)$, and $f\in
    \ff$, we have $\bigl\lvert f(\pi^\omega \uu)\bigr\rvert = \bigl\lvert
    \pi^{\bil{\Lambda(f,\tau)}\omega} \dtimes \tilde f_\omega(\uu)\bigr\rvert$.
    Fix $\tau \subset \cN$, $\omega \in \cC_0^\tau \cap \ZZ^n$, $\bm g\subset
    \ff$, and $\bar\uu \in \bar V^\tau_{\bm g}(\fO/\fP)$.  Let $e = \card{\bm
      g}$, $r = \card{\ff}$, $\ff = \{ f_1,\dotsc,f_r\}$, and $\bm g = \{
    f_1,\dotsc,f_e\}$.  Let $\uu = (u_1,\dotsc,u_n) \in \Torus^n(\fO)$ be an
    arbitrary lift of $\bar\uu$.  By non-degeneracy of $\tilde \ff$ over
    $\fo/\fp$ (hence over $\fO/\fP$) and after suitably renumbering variables,
    we may assume that $\det \bm h'(\uu) \not\equiv 0 \bmod {\fP}$, where $\bm h
    := (\tilde f_{1,\omega},\dotsc,\tilde f_{e,\omega},
    X_{e+1}-u_{e+1},\dotsc,X_n-u_n)$.  Hensel's lemma for polynomial mappings
    \cite[Ch.\ III, \S 4.5, Cor.\ 2]{Bou89} allows us to replace
    $u_1,\dotsc,u_e$ by elements in the same residue classes modulo $\fP$ in
    such a way that $\bm h(\uu) = 0$ holds.  The $\fP$-adic inverse function
    theorem {\cite[Ch.\ III, \S 4.5, Prop.\ 7]{Bou89}} (applied to $\bm h(\uu +
    \XX)$) then shows that $\bm h$ induces a bianalytic and measure-preserving
    bijection from $T(\bar \uu) = \uu + \fP^{n}$ onto $\fP^{n}$.
    Let $\varepsilon_{ij} = 1$ if $f_i \in \ff_j$ and $\varepsilon_{ij} = 0$
    otherwise.  For $\yy\in \fP^n$, $\vv := \bm h^{-1}(\yy) - \uu$, and $0\le j
    \le m$, we have
    \begin{align*}
      \norm{\ff_j(\pi^\omega(\uu + \vv))} & = \Bigl\lVert \varepsilon_{1j}
      \pi^{\bil{\Lambda(f_1,\tau)}\omega} y_1, \dotsc, \varepsilon_{ej}
      \pi^{\bil{\Lambda(f_e,\tau)}\omega} y_e,
      \\
      &\quad\quad\quad \varepsilon_{e+1,j}
      \pi^{\bil{\Lambda(f_{e+1},\tau)}\omega}, \dotsc, \varepsilon_{rj}
      \pi^{\bil{\Lambda(f_r,\tau)}\omega}\Bigr\rVert =: a_{j\omega}(\yy).
    \end{align*}
    Hence, a change of variables yields
    \begin{align*}
      I_\omega(\bar\uu) & = \mu_{n-e}(\fP^{n-e}) \dtimes \int_{\left\{\yy\in
          \fP^e : a_{0\omega}(\yy) \le 1\right\}}
      a_{1\omega}(\yy) ^{s_1} \dotsb a_{m\omega}(\yy)^{s_m} \dd\mu_e(\yy) \\
      & = \frac{(1-q^{-1})^e}{q^{n-e}} \sum_{\substack{\psi\in
          \NN^e,\\a_{0\omega}(\pi^\psi)\le 1}} q^{-\bil\one\psi}
      a_{1\omega}(\pi^\psi)^{s_1} \dotsb a_{m\omega}(\pi^\psi)^{s_m}
      \\
      & = \frac{(q-1)^e}{q^n} \!\!\!\! \!\!\!\! \!\!\!\! \!\!\!\!
      \sum_{\substack{\psi\in \NN^e, \\
          \varepsilon_{i0} \dtimes \left( \bil{\Lambda(f_i,\tau)}\omega +
            \psi_i\right) \ge 0 \,\,(1\le i\le e), \\
          \varepsilon_{i0} \dtimes \bil{\Lambda(f_{i},\tau)}\omega \ge
          0\,\,(e+1\le i\le r) }} \!\!\!\! \!\!\!\! \!\!\!\!  q^{-\bil\one\psi}
      \prod_{j=1}^m q^{-s_j \min\left( \substack{ \varepsilon_{ij} \dtimes
            \left( \bil{\Lambda(f_i,\tau)}\omega +
              \psi_i\right) \,\,(1\le i\le e), \\
            \varepsilon_{ij} \dtimes \bil{\Lambda(f_{i},\tau)}\omega \,\,(e+1\le
            i\le r) } \right)},
    \end{align*}
    where we wrote $\mu_d$ for the normalised Haar measure on $\fO^d$ and
    $\pi^\psi = (\pi^{\psi_1},\dotsc,\pi^{\psi_e})$.
  \end{enumerate}
  Using Proposition~\ref{prop:Zseries}, the claim now follows from
  \eqref{eq:breakup}.
\end{proof}

\begin{ex}[Igusa's local zeta function]
  Let $f\in k[\XX]$ be non-constant.  Denef and Hoornaert \cite[\S 4]{DH01} gave
  an explicit formula for the Igusa zeta function $\int_{\fO^n}
  \abs{f(\xx)}^s\dd\mu(\xx)$ (for $k = \QQ$, $K = \QQ_p$) if $(f)$ is
  non-degenerate relative to $\Orth^n$ in our terminology.  As we will explain
  in \S\ref{ss:ndlit}, Theorem~\ref{thm:evaluate} essentially generalises their
  result.  In fact, even in the very special case of a single polynomial, our
  method applies to a larger class of polynomials.  Namely, let $f =
  f_1^{e_1}\dotsb f_m^{e_m}$ be the factorisation into powers of distinct
  irreducibles and suppose that $(f_1,\dotsc,f_m)$ is non-degenerate relative to
  $\Orth^n$.  By applying Theorem~\ref{thm:evaluate} with $\ff_0 = \{ 1 \}$,
  $\ff_1 = \{f_1\},\dotsc,\ff_m = \{f_m\}$, we may compute $ \int_{\fO^n}
  \abs{f_1(\xx)}^{s_1} \dotsb \abs{f_m(\xx)}^{s_m} \dd\mu(\xx) $ and hence
  $\int_{\fO^n} \abs{f(\xx)}^s\dd\mu(\xx)$ via $s_j = e_j s$ for $j=1,\dotsc,m$.
  Non-degeneracy of $(f)$ implies that of $(f_1,\dotsc,f_m)$ but the converse is
  false.
\end{ex}

\begin{rem}
  \label{r:eval_and_algmod}
  There are various ways of interpreting the integral in
  Theorem~\ref{thm:coneint} within the framework of Theorem~\ref{thm:evaluate}.
  For instance, we can regard the integrand as a specialisation of
  $\abs{x_{11}}^{s_1} \dotsb \abs{x_{dd}}^{s_d}$.  On the level of the rational
  functions in Theorem~\ref{thm:evaluate}, such specialisations amount to
  monomial substitutions as in \S\ref{ss:monomial_sub}.  Specifically, let
  $\gamma$ and $\delta$ be independent variables over $\QQ$.
  Remark~\ref{r:Zdenom} shows that we may define $\tilde \cZ_{\bm
    g}^\tau(\gamma,\delta) := \cZ^{\cC^\tau_0(\bm g), \cP^\tau_1(\bm
    g),\dotsc,\cP^\tau_m(\bm
    g)}(\gamma,\gamma^{-1}\delta,\dotsc,\gamma^{-m}\delta) \in
  \QQ(\gamma,\delta)$ for $\bm g\subset \ff$ and $\tau \subset \cN$.  Hence,
  under the same assumptions as in Theorem~\ref{thm:evaluate}, for $\Real(s) \ge
  m$, we have
  \[
  \Zeta^{\cC_0, \ff_0,\dotsc,\ff_m}_K(s-1,\dotsc,s-m) = \sum_{\substack{\bm g
      \subset \ff,\\\tau\subset \cN}} \noof{\bar V^\tau_{\bm g}(\fO/\fP)}
  \dtimes \frac{(q-1)^{\card{\bm g}}}{q^{n}} \dtimes \tilde\cZ_{\bm
    g}^\tau(q^{-1}, q^{-s}).
  \]
\end{rem}

The central feature of Theorem~\ref{thm:evaluate} is that it is entirely
explicit.  In contrast, the general formulae for subalgebra zeta functions given
in \cite{dSG00,Vol10} rely on the usually infeasible step of constructing a
resolution of singularities.  As an illustration of the key ingredients of
Theorem~\ref{thm:evaluate} and as a demonstration of its practical value, in
\S\ref{ss:sl2}, we use it to compute the known $\fP$-adic subalgebra zeta
functions of $\mathfrak{sl}_2(\ZZ)$.

\subsection{Non-degeneracy and \texorpdfstring{$\fP$}{P}-adic integration in the literature}
\label{ss:ndlit}

Let $k$ be a field with algebraic closure $\bar k$.  In applications to
Igusa-type zeta functions (see e.g.\ \cite{DH01,VZG08}), non-degeneracy of a
non-zero polynomial $f\in k[\XX]$ is usually expressed not in terms of the
Newton polytope $\cN := \Newton(f)$ but in terms of the Newton polyhedron $\cP
:= \cN + \Orth^n$.  This approach is subsumed by ours as follows.  The normal
fan of $\cP$ is the least common refinement of the normal fan of $\cN$ and the
fan of faces of $\Orth^n$.  Hence, the partition $\Orth^n = \bigcup_{\gamma}
\NormalCone_\gamma(\cP)$ into (relatively open) normal cones of non-empty faces
$\gamma \subset \cP$ refines our partition $\Orth^n =
\bigcup_\tau(\NormalCone_{\tau}(\cN) \cap \Orth^n)$ indexed by $\Orth^n$-visible
faces $\tau \subset \cN$.  The number of faces of $\cP$ will often greatly
exceed that of $\cN$ which leads to considerable practical limitations of the
former approach.  We note that by Lemma~\ref{lem:visible}, the compact faces of
$\cP$ often considered in the literature coincide with the
$\StrictOrth^n$-visible faces of $\cN$ in our terminology.  In this sense, in
addition to being better suited for explicit computations,
Theorem~\ref{thm:evaluate} provides a far-reaching generalisation of \cite[Thm\
4.3]{Bor12} (itself a generalisation of \cite[Thm\ 4.2]{DH01}) on integrals
$\int_{\ZZ_p^n} \abs{f(\xx)}^s\! \abs{g(\bm x)}\dd\mu(\xx)$; the assumptions
used in \cite{Bor12} are equivalent to $(f,g)$ being non-degenerate relative to
$\Orth^n$ in our sense.

Veys and \Zuniga{}-Galindo \cite{VZG08} used yet another notion of
non-degeneracy for $\ff = (f_i)_{i\in I}$ to compute integrals
$\int_{\fO^n}\norm{\ff(\xx)}^s\dd\mu(\xx)$.  Their notion is based on the
polyhedron $\conv\bigl(\bigcup_{i\in I}\Newton(f_i)\bigr) + \Orth^n$.  The
aforementioned computational difficulties surrounding unbounded polyhedra apply
here too.  More importantly, we found the non-degeneracy assumption in
\cite[Def.\ 3.1]{VZG08} to be too restrictive for our intended applications to
subalgebra and submodule zeta functions.  The Newton polytope $\cN =
\Newton(\prod_i f_i)$ central to our approach previously featured in work of
\Zuniga-Galindo \cite{ZG09}; it frequently appears in applications of
non-degeneracy to problems in Archimedean settings, see e.g.\ \cite{Oka90}.

\section{Topological zeta functions}
\label{s:topzeta}

The topological zeta function $\Zeta_{f,\topo}(\ess)$ associated with a
polynomial $f\in k[\XX]$ over a number field, say, is usually defined in terms
of numerical data extracted from an embedded resolution of singularities of
$\Spec(k[\XX]/f) \subset \Spec(k[\XX])$.  There are two established ways, both
of them due to Denef and Loeser, of proving independence of the chosen
resolution (for a third approach, see \cite[\S 1.5]{Loe09}).  In \cite{DL92},
Denef and Loeser used $\ell$-adic interpolation techniques to express
$\Zeta_{f,\topo}(\ess)$ as a limit of Igusa zeta functions associated with $f$,
thus giving meaning to the limit ``$p\to 1$'' from the introduction.  A more
geometric approach is given in \cite{DL98}, where the motivic zeta function
associated with $f$ is introduced and shown to specialise to
$\Zeta_{f,\topo}(\ess)$.

In this section, based on ideas from \cite{DL92}, we give a self-contained
introduction to topological zeta functions using the original $\ell$-adic
approach.  We attach topological zeta functions to systems of local zeta
functions of a certain shape arising in $\fP$-adic integration.  Our approach
therefore applies to a range of algebraic counting problems.  In particular, in
\S\ref{ss:topsub}, we give rigorous definitions of the topological subalgebra
and submodule zeta functions informally introduced in \S\ref{s:intro}.  The
first of these zeta functions was defined in \cite[\S 8]{dSL04} (in greater
generality and slightly differently, see Remark~\ref{r:shift}) as a
specialisation of a suitable motivic one.  In contrast to the more recent theory
of motivic integration, $\fP$-adic integration is firmly established as a tool
in asymptotic algebra.  By relying exclusively on $\fP$-adic methods, the
machinery in this section can, for instance, immediately be used to define
topological representation zeta functions associated with finitely generated
torsion-free nilpotent groups via \cite[Prop.\ 3.1]{Vol10}.

\subsection{Formal and \texorpdfstring{$\ell$}{l}-adic binomial expansions}
\label{ss:binomial}

In this subsection, which formalises ideas used in \cite[(2.4)]{DL92}, we
explain how certain rational functions $W(\qq,\tee)$ admit an algebraic
``reduction modulo $\qq-1$'' which corresponds to taking an $\ell$-adic limit
$x\to 1$ of $W(x,x^{-s})$.

Let $k$ be a field of characteristic zero and $\zz$ be an indeterminate over
$k$.  Given a formal power series $f \in k \llb \zz \rrb$ without constant term,
it is well-known that formal powers defined using the binomial series $(1+f)^a =
\sum_{d=0}^\infty {\binom a d} f^d$ with $a\in k$ and $\binom a d = \frac {a
  (a-1) \dotsb (a-d+1)}{d!}$ satisfy $(1+f)^a (1+f)^b = (1+f)^{a+b}$ for $a,b\in
K$, see e.g.\ \cite[\S 1.1.6]{GJ83}.  Henceforth, let $\qq$, $\tee_1$, $\dotsc$,
$\tee_m$, $\ess_1$, $\dotsc$, $\bm s_m$ be algebraically independent over $\QQ$.
Let $\tee = (\tee_1,\dotsc,\tee_m)$ and $\ess = (\ess_1, \dotsc, \ess_m)$.  For
$a\in \QQ[\ess]$, write $\qq^a := (1 + (\qq-1))^a \in \QQ[\ess]\llb \qq-1\rrb$,
and let $\qq^{-\ess} := (\qq^{-\ess_1},\dotsc,\qq^{-\ess_m})$.

\begin{lemma}
  The homomorphism $f(\qq,\tee) \mapsto f(\qq,\qq^{-\ess})$ embeds
  $\QQ[\qq,\tee]$ into $\QQ[\ess]\llb \qq-1\rrb$.
\end{lemma}
\begin{proof}
  Let $0\not= f(\qq,\tee) \in \QQ[\qq,\tee]$.  By induction, we find $r\in
  \NN^m$ with $f(\qq,\qq^{r_1},\dotsc,\qq^{r_m})\not= 0$.
  The binomial theorem shows that if we substitute $\ess_j \gets -r_j$ in the
  coefficients of $f(\qq,\qq^{-\ess})$ as a series in $\qq-1$, we obtain the
  coefficients of $f(\qq,\qq^{r_1}, \dotsc, \qq^{r_m})$ as a polynomial in
  $\qq-1$.  As $f(\qq,\qq^{r_1},\dotsc,\qq^{r_m})\not= 0$, we conclude that
  $f(\qq,\qq^{-\ess})\not= 0$.
\end{proof}

Hence, $f(\qq,\tee)\mapsto f(\qq,\qq^{-\ess})$ extends to an embedding of
$\QQ(\qq,\tee)$ into the field $\QQ(\ess)\llp \qq-1\rrp$ of formal Laurent
series in $\qq-1$ over $\QQ(\ess)$.  Consider the $\QQ$-subalgebra, $\cA$ say,
of $\QQ(\qq,\tee)$ generated by $\qq^{\pm 1}, \tee_1^{\pm 1}, \dotsc,
\tee_m^{\pm 1}$, and all $(\qq^a\tee^{b} -1)^{-1}$ with $(0,0) \not=(a,b) \in
\ZZ^{1+m}$.
\begin{defn}
  \label{d:MM}
  Let $\MM$ be the $\QQ$-subalgebra of $\QQ(\qq,\tee)$ consisting of those
  $W(\qq,\tee)\in \cA$ with $W(\qq,\qq^{-\ess}) \in \SuperMM :=
  \QQ(\ess)\llb\qq-1\rrb$.
\end{defn}

We henceforth regard $\MM$ as a subalgebra of $\SuperMM$ via $W(\qq,\tee)\mapsto
W(\qq,\qq^{-\ess})$.  For added clarity, we sometimes write
$\MM(\qq;\tee_1,\dotsc,\tee_m)$ instead of $\MM$.
\begin{notation}
  \label{not:red}
  Let $\red{\blank}\!\colon \SuperMM \to \QQ(\ess)$ denote reduction modulo
  $\qq-1$.
\end{notation}
For $W^\sharp \in \SuperMM$, we sometimes write $\red{W^\sharp(\ess)}$ instead
of $\red{W^\sharp\!}(\ess)$.
\begin{exs}
  \quad
  \begin{enumerate}
  \item $\qq^{-1} = \sum_{d=0}^\infty (-1)^d (\qq-1)^d$ and $\tee_j^{-1} =
    \sum_{d=0}^\infty \binom{\ess_j}d (\qq-1)^d$ so that $\red{\qq^{\pm 1}} =
    \red{\tee_j^{\pm 1}} = 1$.
  \item $\frac{\qq-1}{\qq^a\tee^{b}-1}\in \MM$ and $
    \red*{\frac{\qq-1}{\qq^a\tee^b-1}} = \frac 1 {a-\bil{b} {\ess}} = \frac 1 {a
      - b_1 \ess_1 - \dotsb - b_m \ess_m} $ for $(0,0)\not= (a,b)\in \ZZ^{1+m}$.
  \end{enumerate}
\end{exs}

Fix a rational prime $\ell$.  For $v\ge 1$, let $U_v = \{ x \in \ZZ_\ell :
\abs{x -1}_\ell \le \ell^{-v}\}$, a subgroup of $\ZZ_\ell^\times$.  It is
well-known that each $U_v$ admits a unique continuous $\ZZ_\ell$-module
structure $U_v\times \ZZ_\ell \to U_v, (x,e) \mapsto x^e$ extending the
canonical $\ZZ$-action.  The following is a multivariate version of an argument
from \cite[(2.4)]{DL92}.
\begin{prop}
  \label{prop:Winterpol}
  Let $W(\qq,\tee)\in \MM$.  There exists a finite union $\cH$ of affine
  hyperplanes (defined over $\ZZ$, independently of $\ell$) in $\QQ_\ell^m$ such
  that $ U_2 \!\setminus\!\{1\} \times\ZZ^m\!\setminus\!\cH \to \QQ_\ell, (x,s)
  \mapsto W(x,x^{-s_1},\dotsc,x^{-s_m}) $ extends to a continuous map $F\colon
  U_2\times \ZZ_\ell^m\!\setminus\!\cH\to\QQ_\ell$ with $F(1,s) = \red{W(s)}$
  for $s\in \ZZ_\ell^m\!\setminus\!\cH$.
\end{prop}
Note that $\ZZ^m\!\setminus\!\cH$ is dense in $\ZZ_\ell^m$ whence $F$ is unique.
Proposition~\ref{prop:Winterpol} is a consequence of the following more
technical statement.

\begin{lemma}
  \label{lem:Winterpol}
  Let $W(\qq,\tee) \in \MM$ and $v\ge 1$ with $\ell^v \ge 3$.  Suppose that
  $\prod_{(a,b)\in \Lambda} (\qq^a\tee^{b}-1)^{e(a,b)}\dtimes W(\qq,\tee)\in
  \QQ[\qq^{\pm 1}, \tee^{\pm 1}]$, where $\Lambda = \ZZ^{1+m}\setminus\{0\}$ and
  $e\colon \Lambda \to \NN_0$ has finite support.
  \begin{enumerate}
  \item
    \label{it:Winterpol1}
      
    Define $\cH = \{ s\in \ZZ^m : a = \bil {b} s \text{ for some } (a,b)\in
    \Lambda \text{ with } e(a,b)> 0\}$.  Then $$G_0\colon
    U_v\setminus\{1\}\times \ZZ^m\setminus\cH \to \QQ_\ell, \quad (x,s) \mapsto
    \prod_{(a,b)\in \Lambda} (a-\bil b s)^{e(a,b)} \dtimes
    W(x,x^{-s_1},\dotsc,x^{-s_m})$$ extends to a continuous map $G\colon
    U_v\times \ZZ_\ell^m \to \QQ_\ell$.
  \item
    \label{it:Winterpol2}
    $g(\ess) \!:=\! \prod_{(a,b)\in \Lambda}(a\!-\!\bil{b}{\ess})^{e(a,b)}
    \red{W(\ess)}$ belongs to $\QQ[\ess]$ and $G(1,s) \!=\! g(s)$ for $s\!\in\!
    \ZZ_\ell^m$.
  \end{enumerate}
\end{lemma}
\begin{proof}
  First, there exists a continuous map $h\colon U_v \times \ZZ_\ell \to \ell
  \ZZ_\ell$ such that (a) $x^e - 1 = e (x-1) \dtimes (1+h(x,e))$ for all $x\in
  U_v$ and $e\in \ZZ_\ell$ and (b) $h(1,\blank) = 0$.
  Indeed, for $e\not= 0$, the $\ell$-adic binomial series \cite[Cor.\
  4.2.18]{Coh07} yields $ x^e - 1 = e(x-1)\dtimes \left( 1 + \sum_{d=1}^\infty
    \binom{e-1}{d} \dtimes \frac{(x-1)^d}{d+1} \right) $ and $\sup_{x\in
    U_v,e\in \ZZ_\ell} \abs{\binom{e-1}{d} \dtimes \frac{(x-1)^d}{d+1}}_\ell \le
  \sup_{x \in U_v}\!\left( (d+1) \dtimes \abs{x-1}_\ell^d\right) \le
  \frac{d+1}{3^d} < 1$ for $d\ge 1$.

  Let $f(\qq,\tee) = \prod_{(a,b)\in \Lambda} (\qq^a\tee^b-1)^{e(a,b)}\dtimes
  W(\qq,\tee)$, say $f(\qq,\tee) = \sum_{(a,b)\in \ZZ^{1+m}} c_{ab}\qq^a\tee^b$
  with $c_{ab}\in \QQ$, almost all of which are $0$.
  Let $M = \sum_{(a,b)\in \Lambda}e(a,b)$.
  As an element of $\SuperMM$, the series $W(\qq,\qq^{-\ess})$ has nonnegative
  $(\qq-1)$-adic valuation.  Hence, $u^\sharp := f(\qq,\tee)(\qq-1)^{-M}$ is
  contained in $\SuperMM$ and thus in $\QQ[\ess]\llb \qq-1\rrb$ since
  $f(\qq,\qq^{-\ess})\in \QQ[\ess]\llb \qq-1\rrb$.  Note that $W(\qq,\tee) =
  u^\sharp \dtimes \prod_{(a,b)\in \Lambda}
  \left(\frac{\qq-1}{\qq^a\tee^b-1}\right)^{e(a,b)}.$ Thus, $g(\ess)$ in part
  (\ref{it:Winterpol2}) of the lemma coincides with $\red{u^\sharp} \in
  \QQ[\ess]$.  The $\ell$-adic binomial series shows that the map
  $U_v\!\setminus\!\{1\}\times \ZZ^m \to \QQ_\ell, (x,s) \mapsto
  f(x,x^{-s_1},\dotsc,x^{-s_m})(x-1)^{-M}$ admits the continuous extension $
  F\colon U_v\times \ZZ_\ell^m\to \QQ_\ell, (x,s) \mapsto \sum_{d=0}^\infty
  \left(\sum_{(a,b)\in \Lambda} c_{ab}\binom{a-\bil b s}{d+M}\right) (x-1)^d $
  which satisfies $F(1,s) = \sum_{(a,b)\in\Lambda}c_{ab}\binom{a-\bil b s}{M} =
  \red{u^\sharp(s)}$ for $s\in \ZZ_\ell^m$.
  We thus obtain an extension of $G_0$
  of the desired form, namely $(x,s) \mapsto F(x,s) \dtimes \prod_{(a,b)\in
    \Lambda} (1+h(x,a-\bil b s))^{-e(a,b)}$.
\end{proof}

\subsection{Local zeta functions of Denef type}
\label{ss:denef_type}

Throughout this subsection, let $k$ be a number field with ring of integers
$\fo$.

\begin{defn}
  \label{d:local_system}
  By a \emph{system of local zeta functions} (in $m$ variables) over $k$, we
  mean a family $\Zeta = (\Zeta_K)$ such that
  \begin{enumerate}
  \item $K$ ranges over all non-Archimedean local fields endowed with an
    embedding $k \subset K$,
  \item
    \label{d:local_system2}
    each $\Zeta_K$ is a meromorphic function of the form
    $$\Zeta_K(s_1,\dotsc,s_m) = W_K(q_K^{-s_1},\dotsc,q_K^{-s_m})$$ for 
    complex $s_1,\dotsc,s_m$ and a rational function $W_K(\tee_1,\dotsc,\tee_m)
    \in \QQ(\tee_1,\dotsc,\tee_m)$, where $q_K$ denotes the residue field size
    of $K$, and
  \item $\Zeta_{K_1} = \Zeta_{K_2}$ whenever non-Archimedean local fields
    $K_1,K_2 \supset k$ are topologically isomorphic over $k$.
  \end{enumerate}
\end{defn}

The final condition is included to avoid set-theoretic difficulties.

\begin{defn}
  \label{d:equivalent}
  Let $\Zeta$ and $\Zeta'$ be systems of local zeta functions over $k$.  We say
  that $\Zeta$ and $\Zeta'$ are \emph{equivalent} if there exists a finite set
  $S$ of primes of $k$ such that $\Zeta^{\phantom o}_K = \Zeta'_K$ whenever
  $\fP_K \cap \fo \not\in S$, where $\fP_K$ is the maximal ideal of the
  valuation ring of $K$.
\end{defn}

We now consider systems of local zeta functions given (up to equivalence) by
formulae similar to those obtained by Denef for Igusa's local zeta function
\cite[Thm\ 3.1]{Den87}.  Thus, let $I$ be a finite set and suppose that for each
$i\in I$, we are given a variety $V_i$ over $k$ and a rational function
$W_i(\qq,\tee_1,\dotsc,\tee_m) \in \MM$ (see Definition~\ref{d:MM}).  In the
cases of interest to us, each $V_i$ will be quasi-projective, i.e.\ a locally
closed (but not necessarily reduced) subscheme of projective space over $k$.
Given these data, we obtain an associated system $\Denef = (\Denef_K)$ of local
zeta functions as follows: for each prime $\fp$ of $k$ and non-Archimedean local
field $K \supset k$ such that the maximal ideal $\fP_K$ of its valuation ring
$\fO_K$ satisfies $\fP_K \cap \fo = \fp$, we let
\[
\Denef_K(s_1,\dotsc,s_m) := \sum_{i\in I} \noof{\bar V_i(\fO_K/\fP_K)} \dtimes
W_i(q_K^{\phantom{s_1}},q_K^{-s_1},\dotsc,q_K^{-s_m}),
\]
where $q_K$ is the residue field size of $K$ and $\bar\dtimes$ denotes
``reduction modulo $\fp$''.  A rigorous definition of ``reduction modulo $\fp$''
is a subtle matter, see \cite[\S 2]{Den87} for Denef's definition.  For our
purposes, the following naive approach will suffice.  Namely, we choose, once
and for all, for each $i\in I$, a model $X_i$ of $V_i$ over $\fo$, i.e.\ a
separated scheme $X_i$ of finite type over $\fo$ with $(X_i)_k \approx_k V_i$.
Given $\fp$, we then define $\bar V_i := X_i \times_{\Spec(\fo)}
\Spec(\fo/\fp)$.  While this definition of $\bar V_i$ is not intrinsic, for a
fixed rational prime $\ell$, the following numbers do not depend on our choice
of $X_i$, provided we are willing to ignore finitely many $\fp$:
\begin{enumerate}
\item[(N1)] the number of rational points of $\bar V_i$ over finite field
  extensions of $\fo/\fp$ and
\item[(\textlabel{N2}{redmodp2})] the $\ell$-adic Euler characteristic
  $\Euler_c((\bar V_i)_{\overline{\fo/\fp}},\QQ_\ell)$ with compact support of
  $\bar V_i$, taken after base changing to an algebraic closure
  $\overline{\fo/\fp}$ of $\fo/\fp$---indeed, $\Euler_c((\bar
  V_i)_{\overline{\fo/\fp}},\QQ_\ell) = \Euler_c((V_i)_{\bar k},\QQ_\ell)$;
\end{enumerate}
see e.g.\ \cite[\S 4.8.2]{Ser12} and cf.\ \cite[7.6]{dSL04}, \cite[\S
1.3]{Loe09}, \cite[\S 5.3]{Nic10}.  In particular, up to equivalence, $\Denef$
only depends on the collection
$\left(V_i,W_i(\qq,\tee_1,\dotsc,\tee_m)\right)_{i\in I}$.

\begin{rem}
  \label{r:artin}
  By Artin's comparison theorem and invariance of $\ell$-adic cohomology under
  extensions of algebraically closed fields (see \cite{Kat94} for both), for
  almost all $\fp$, the Euler characteristic in (\ref{redmodp2}) coincides with
  the topological Euler characteristic (with or without compact support
  \cite[A4, Prop.]{KP85})\!\! of the $\CC$-analytic space associated with $V_i$
  for any embedding $k\subset \CC$.
\end{rem}

\begin{defn}
  \label{d:denef_type}
  We say that a system of local zeta functions over $k$ is of \emph{Denef type}
  if it is equivalent to a system of the form $\Denef$ just constructed.
\end{defn}

For the following examples of zeta functions, we do not distinguish between
defining integrals (or Dirichlet series) and associated meromorphic
continuations.  Moreover, as we are only interested in systems of local zeta
functions up to equivalence, the ``explicit formulae'' cited below provide the
rationality properties required by
Definition~\ref{d:local_system}(\ref{d:local_system2}); we note that, instead of
proving rationality using model-theoretic techniques, the aforementioned
explicit formulae can be modified to cover exceptional primes, see \cite[\S
4.3]{AKOV13}.

\begin{exs}
  \label{exs:denef_type}
  \quad As before, let $k$ be a number field with ring of integers $\fo$.  We
  follow the conventions for local fields in Notation~\ref{not:local} with
  subscripts $(\blank)_K$ added for clarity.
  \begin{enumerate}
  \item
    \label{exs:denef_type1}
    The prime example of a system of local zeta functions of Denef type, and the
    reason for our choice of terminology, is given by Igusa's local zeta
    function.  Thus, for $f\in k[\XX]$, the family $\Zeta_f = (\Zeta_{f,K})$
    given by $\Zeta_{f,K}(s) := \int_{\fO_K^n}\abs{f(\xx)}_K^s\dd\mu_K(\xx)$ is
    of Denef type.  For a proof, take Denef's ``explicit formula'' \cite[Thm\
    3.1]{Den87} and apply \cite[Prop.\ 2.3, Thm\ 2.4]{Den87} to pass from
    completions $k_{\fp}$ of $k$ to finite extensions; see also \cite[Rem.\ 2.3
    \& \S 3]{Den91b}.
  \item
    \label{exs:denef_type2}
    More generally, for non-empty finite $\ff \subset k[\XX]$, the family
    $\Zeta_{\ff}$ with $\Zeta_{\ff,K}(s) :=
    \int_{\fO_K^n}\norm{\ff(\xx)}_K^s\dd\mu_K(\xx)$ is of Denef type.  Here, one
    takes the formula of Veys and \Zuniga{}-Galindo \cite[Thm\ 2.10]{VZG08} and
    uses stability under base change as in (\ref{exs:denef_type1}).
  \item
    \label{exs:denef_type3}
    Let $\cA$ be a non-associative $\fo$-algebra which is free of finite rank
    $d$ as an $\fo$-module.  Define a family $\Zeta_{\cA}$ by $\Zeta_{\cA,K}(s)
    := (1-q_K^{-1})^d \zeta_{\cA_{\fO_K}}(s)$ (see
    Definition~\ref{d:subzeta}(\ref{d:subzeta1})).  Similarly, for a free
    $\fo$-module $M$ of rank $d$ and a subalgebra $\cE$ of $\End_{\fo}(M)$,
    define $\Zeta_{\cE \acts M}$ by $\Zeta_{\cE\acts M,K}(s) := (1-q_K^{-1})^d
    \zeta_{\cE_{\fO_K}\acts M_{\fO_K}}(s)$ (see
    Definition~\ref{d:subzeta}(\ref{d:subzeta2})).  Then $\Zeta_{\cA}$ and
    $\Zeta_{\cE\acts M}$ are both of Denef type, see Theorem~\ref{thm:subdenef}.
  \item
    \label{exs:denef_type4}
    Taking $\Zeta_K(s_1,\dotsc,s_m)$ to be the integral in
    Proposition~\ref{prop:monomial_integral}, we obtain a system $(\Zeta_K)$ of
    local zeta functions of Denef type.  Indeed, the remarks following
    Theorem~\ref{thm:genfun} show that $(1-\qq^{-1})^n
    \cZ^{\cC_0,\cP_1,\dotsc,\cP_m}(\qq^{-1},\tee_1,\dotsc,\tee_m) \in \MM$ (see
    Definition~\ref{d:MM}).
  \item Further examples of multivariate local zeta functions of Denef type are
    provided by \cite[Thm\ 2.1]{Vol10}.  For the stability under base change,
    one again argues as in (\ref{exs:denef_type1}), while the interpretation of
    rational functions as elements of $\MM$ is similar to
    (\ref{exs:denef_type4}).
  \item
    \label{exs:denef_type6}
    Let $\cC_0 \subset \Orth^n$ be a half-open rational cone and let $\ff_0
    \subset \fo[\XX^{\pm 1}]$ and $\ff_1,\dotsc,\ff_m \subset \fo[\XX] =
    \fo[X_1,\dotsc,X_n]$ be non-empty finite sets with $0\not\in \ff := \ff_0
    \cup\dotsb \cup \ff_m$.  Disregarding finitely many primes $\fP_K \cap \fo$
    of $k$, if $\ff := \ff_0 \cup \dotsb \cup \ff_m$ is non-degenerate relative
    to $\cC_0$ (see Definition~\ref{d:nd}(\ref{d:nd1})), then
    Theorem~\ref{thm:evaluate} shows that $\Zeta^{\cC_0,\ff_0,\dotsc,\ff_m} :=
    \Bigl(\Zeta^{\cC_0,\ff_0,\dotsc,\ff_m}_K\Bigr)$ is a system of local zeta
    functions.  While $\Zeta^{\cC_0,\ff_0,\dotsc,\ff_m}$ is indeed of Denef
    type, Theorem~\ref{thm:evaluate} falls short of making this explicit because
    the rational functions $ \qq^{-n}(\qq-1)^{\card{\bm g}} \cZ^{\cC^\tau_0(\bm
      g), \cP^\tau_1(\bm g),\dotsc,\cP^\tau_m(\bm g)}(\qq^{-1},
    \tee_1,\dotsc,\tee_m) $ in Theorem~\ref{thm:evaluate} need not belong to
    $\MM$.  This ``flaw'' is a minor one which we will fix in
    Lemma~\ref{lem:nd_denef}.
    
    Although we will not use it in the following, we note that
    $\Zeta^{\cC_0,\ff_0,\dotsc,\ff_m}$ is a system of local zeta functions of
    Denef type without any non-degeneracy assumptions on $\ff$, as can be shown
    using Remark~\ref{rem:principal} and arguments similar to
    (\ref{exs:denef_type2}) and (\ref{exs:denef_type3}).
  \end{enumerate}
\end{exs}

\subsection{Topological zeta functions as limits of local ones}
Let $\Zeta$ be a system of local zeta functions in $m$ variables over a number
field $k$ (see Definition~\ref{d:local_system}).  Suppose that $\Zeta$ is of
Denef type (see Definition~\ref{d:denef_type}).  For a prime $\fp$ of $k$ and
$f\ge 1$, let $k_{\fp}^{(f)}$ denote the unramified extension of degree $f$ of
the completion $k_{\fp}$ of $k$ at $\fp$.  Write $\Zeta_{\fp}(f;s_1,\dotsc,s_m)
:= \Zeta_{k_{\fp}^{(f)}}(s_1,\dotsc,s_m)$.  Then, unless $\fp$ belongs to some
finite exceptional set of primes of $k$, for any local field $K\supset k_{\fp}$
with residue class degree $f$ over $k_{\fp}$, we have $\Zeta_K(s_1,\dotsc,s_m) =
\Zeta_{\fp}(f;s_1,\dotsc,s_m)$.

\begin{thm}[{Cf.\ \cite[(2.4)]{DL92}}]
  \label{thm:top}
  Let $\Zeta$ be a system of local zeta functions in $m$ variables and of Denef
  type over a number field $k$.  Then:
  \begin{enumerate}
  \item
    \label{thm:top1}
    There exists a finite union $\cH$ of affine hyperplanes in $\AA_{\ZZ}^{\!m}$
    such that for each rational prime $\ell$ and almost all primes $\fp$ of $k$
    (depending on $\ell$ and $\Zeta$), there exists $d \ge 1$ such that
    $$
    \NN \times \ZZ^m\!\setminus\!\cH(\ZZ) \to \QQ, (f,s) \mapsto
    \Zeta_{\fp}(df;s)
    $$
    extends to a continuous map $\ZZ_\ell^{\phantom o} \times
    \ZZ_\ell^m\!\setminus\!\cH(\ZZ_\ell) \to \QQ_\ell$.
  \item
    \label{thm:top2}
    Notation as in (\ref{thm:top1}).  There exists a rational function
    $\Zeta_{\topo}(\ess)\in \QQ(\bm s_1,\dotsc,\ess_m)$ such that for all $\ell$
    and almost all primes $\fp$ (depending on $\ell$ and $\Zeta$), there exists
    $d \ge 1$ such that $\lim\limits_{\substack{f\to 0,\\f\in \NN}}
    \Zeta_{\fp}(df;s) = \Zeta_{\topo}(s)\in \QQ$
    for all $s\in \ZZ^m\!\setminus\! \cH(\ZZ)$, the limits being $\ell$-adic.
  \item
    \label{thm:top3}
    Let $\Denef$ be defined in terms of $k$-varieties $V_i$ and
    $W_i(\qq,\tee)\in \MM$ as in Definition~\ref{d:denef_type}.  Then
    $\Denef_{\topo}(\ess) = \sum_{i\in I} \Euler(V_i(\CC)) \dtimes
    \red{W_i(\ess)}$ (see Notation~\ref{not:red}).
  \end{enumerate}
\end{thm}

Note that since $\ZZ^m\!\setminus\!\cH(\ZZ)$ is Zariski-dense in $\CC^m$, the
rational function $\Zeta_{\topo}(\ess)$ in (\ref{thm:top2}) is uniquely
determined (and independent of $\cH$).  Furthermore, by construction,
$\Zeta_{\topo}(\ess)$ only depends on the equivalence class of $\Zeta$ in the
sense of Definition~\ref{d:equivalent}.

\begin{defn}
  \label{d:topzeta}
  Let $\Zeta$ be a system of local zeta functions of Denef type over a number
  field.  The \emph{topological zeta function} associated with $\Zeta$ is the
  rational function $\Zeta_{\topo}(\ess)$ in
  Theorem~\ref{thm:top}(\ref{thm:top2}).
\end{defn}

This definition generalises the case of Igusa's local zeta function in
\cite{DL92}.  Denef and Loeser apparently chose the name owing to the appearance
of the topological Euler characteristics of the $\CC$-analytic spaces $V_i(\CC)$
(defined with respect to an arbitrary embedding $k\subset \CC$) in
Theorem~\ref{thm:top}(\ref{thm:top3}).

The following result needed to derive Theorem~\ref{thm:top} can be found in
\cite[(2.4)]{DL92}; we include a short group-theoretic proof.
\begin{lemma}
  \label{lem:NXp}
  Let $V$ be separated scheme of finite type over $\FF_q$.  Let $\ell$ be a
  rational prime with $\ndivides \ell q$.  There exists $d \in \NN$ such that
  $\NN \to \NN_0, f\mapsto \noof{V(\FF_{q^{d f}})}$ extends to a continuous map
  $\alpha\colon \ZZ_\ell\to\ZZ_\ell$ which satisfies $\alpha(0) =
  \Euler_c(V_{\bar \FF_q},\QQ_\ell)$.
\end{lemma}
\begin{proof}
  We use some basic properties of $\ell$-adic cohomology as e.g.\ explained in
  \cite{Kat94}.
  Let $\Gamma$ be the absolute Galois group of $\FF_q$ and let $\sigma \in
  \Gamma$ be the arithmetic Frobenius of $\FF_q$.
  It is well-known \cite[Thm\ 4.4]{Ser12} that the geometric Frobenius of $V$
  and $\sigma^{-1}$ induce the same automorphism, $\gamma_i$ say, of $W_i :=
  \HH^i_c(V_{\bar \FF_q},\QQ_\ell)$ for every $i\ge 0$.  Hence, by
  Grothendieck's trace formula, $\noof{V(\FF_{q^f})} = \sum_i (-1)^i
  \trace(\gamma_i^{f})$ for all $f\ge 1$.  Since the natural action $\Gamma \to
  G_i := \GL(W_i)$ is continuous and $G_i \approx \GL_{\dim(W_i)}(\QQ_\ell)$ is
  $\ell$-adic analytic, there exists $d \in \NN$ such that each $\gamma_i^d$ is
  contained in a pro-$\ell$ subgroup of $G_i$.
  It follows that each map $\ZZ \to G_i, f \mapsto \gamma_i^{d f}$ admits a
  continuous extension $\ZZ_\ell\to G_i$.
\end{proof}

\begin{proof}[Proof of Theorem~\ref{thm:top}]
  Let $\Denef$ be as in (\ref{thm:top3}) and suppose that $\Zeta$ is equivalent
  to $\Denef$.  Fix an arbitrary $\ell$.  We may choose a single $\cH$
  (independent of $\ell$) such that the conclusion of
  Proposition~\ref{prop:Winterpol} holds for all $W_i(\qq,\tee)$ (with
  $\cH(\ZZ_\ell)$ in place of $\cH$).  Let $\fp$ be a prime of $k$ and let $q$
  be the size of the residue field.  Suppose that $\ndivides \ell q$.  There
  exists $d\ge 1$ such that the conclusion of Lemma~\ref{lem:NXp} holds for all
  $\bar V_i$ and, additionally, $\abs{q^d - 1}\le \ell^{-2}$.
  Part~(\ref{thm:top1}) follows.  Next, $ \lim_{f\to 0} \Denef_{\fp}(df;s) =
  \sum_{i\in I} \chi_c((\bar V_i)_{\overline{\fo/\fp}},\QQ_\ell) \dtimes
  \red{W_i(s)} $ for $s\in \ZZ^m\!\setminus\!\cH(\ZZ)$.  As $\chi_c((\bar
  V_i)_{\overline{\fo/\fp}},\QQ_\ell) = \chi(V_i(\CC))$ for almost all $\fp$
  (see Remark~\ref{r:artin}), parts (\ref{thm:top2})--(\ref{thm:top3}) follow.
\end{proof}

\begin{rem}
  \label{r:subdenef}
  As we will now explain, passing from local to topological zeta functions is
  compatible with suitable affine specialisations.  Thus, let $\Denef$ be a
  system of local zeta functions of Denef type over a number field $k$, arising
  from varieties $V_i$ and $W_i(\qq,\tee) \in \MM(\qq;\tee_1,\dotsc,\tee_m)$ as
  in Definition~\ref{d:denef_type}.  Let $\bm c \in \ZZ^m$, $A \in \Mat_{m\times
    r}(\ZZ)$ with rows $A_1,\dotsc,A_m$, and let $\tilde\tee :=
  (\tilde\tee_1,\dotsc,\tilde\tee_r)$ and $\tilde{\bm s} := (\tilde{\bm
    s}_1,\dotsc,\tilde{\bm s}_r)$.  Suppose that $\prod_{(a,b)\in \Lambda}
  (\qq^a\tee^{b}-1)^{e(a,b)}\dtimes W_i(\qq,\tee)\in \QQ[\qq^{\pm 1}, \tee^{\pm
    1}]$ for all $i\in I$, where $\Lambda = \ZZ^{1+m}\!\setminus\!\{0\}$ and
  $e\colon \Lambda \to \NN_0$ has finite support.  If $\qq^{a+\bil b
    c}\tilde\tee^{b A}\not= 1$ for $e(a,b) > 0$, we may define
  $W_i'(\qq,\tilde\tee) := W_i(\qq,\qq^{c_1}\tilde\tee^{A_1}, \dotsc,
  \qq^{c_m}\tilde\tee^{A_m})$ and it is easily verified that each
  $W_i'(\qq,\tilde\tee)$ belongs to $\MM(\qq;\tilde\tee)$.  We thus obtain a
  system of local zeta functions in $r$ variables, $\Denef'$ say, associated
  with the $V_i$ and $W_i'(\qq,\tilde\tee)$.  We then have
  $\Denef'_{\topo}(\tilde{\bm s}_1,\dotsc,\tilde{\bm s}_r) = \Denef_{\topo}(
  \bil {A_1}{\tilde{\bm s}} - c_1, \dotsc, \bil {A_m}{\tilde{\bm s}} - c_m)$.
\end{rem}

\subsection{Topological subalgebra and submodule zeta functions}
\label{ss:topsub}

Let $k$ be a number field with ring of integers $\fo$.  Let $\cA$ be a
non-associative $\fo$-algebra which is free of finite rank $d$ as an
$\fo$-module.  Let $M$ be a free $\fo$-module of rank $d$ and let $\cE$ be an
associative subalgebra of $\End_{\fo}(M)$.  In
Example~\ref{exs:denef_type}(\ref{exs:denef_type3}), we defined systems of local
zeta functions $\Zeta_{\cA}$ and $\Zeta_{\cE \acts M}$ by rescaling the
subalgebra and submodule zeta functions associated with $\cA$ and $\cE$ acting
on $M$, respectively.

\begin{thm}[{Cf.\ \cite[\S\S 2--3]{dSG00}}]
  \label{thm:subdenef}
  $\Zeta_{\cA}$ and $\Zeta_{\cE \acts M}$ are both of Denef type.
\end{thm}
\begin{proof}
  We only consider $\Zeta_{\cA}$, the case of $\Zeta_{\cE \acts M}$ being
  essentially identical.  First, up to a shift $s\mapsto s+d$,
  Theorem~\ref{thm:coneint} expresses $\Zeta_{\cA,K}$ in terms of the ``cone
  integrals'' studied in \cite{dSG00}.  In particular, \cite[Thm\ 1.4]{dSG00}
  provides explicit formulae for $\Zeta_{\cA,K}(s)$ in terms of an embedded
  resolution of singularities over $k$.
  Strictly speaking, \cite[Thm\ 1.4]{dSG00} only covers the case $k = \QQ, K =
  \QQ_p$ but as \cite[Thm\ 1.4]{dSG00} draws upon and generalises Denef's
  explicit formula \cite[Thm\ 3.1]{Den87} and its proof, the extension to an
  arbitrary number field $k$ is straightforward.  Again, as in the context of
  Igusa's local zeta function (see
  Example~\ref{exs:denef_type}(\ref{exs:denef_type1}) and cf.\ \cite[\S
  4.2]{AKOV13}), one deduces the desired behaviour of \cite[Thm\ 1.4]{dSG00}
  under base change.  Given the shape of their explicit formula, it only remains
  to show that the rational functions occurring there give rise to elements of
  our ring $\MM$. This follows since, using their notation, we have $\card{M_k}
  = \dim(R_k) \le \card{I_k}$ for $k=0,\dotsc,w$.
\end{proof}

We thus obtain associated topological zeta functions (see
Definition~\ref{d:topzeta}).  In the univariate case, we do not distinguish
between $\ess$ and $\ess_1$ in our notation.

\clearpage

\begin{defns}
  \quad
  \begin{enumerate}
  \item The \emph{topological subalgebra zeta function} of $\cA$ is
      $$\zeta_{\cA,\topo}(\ess) := \Zeta_{\cA,\topo}(\ess).$$
    \item The \emph{topological submodule zeta function} of $\cE$ acting on $M$
      is
      $$\zeta_{\cE\acts M,\topo}(\ess) := \Zeta_{\cE\acts M,\topo}(\ess).$$
    \end{enumerate}
  \end{defns}
  We furthermore define the \emph{topological ideal zeta function}
  $\zeta_{\cA,\topo}^{\normal}(\ess)$ of $\cA$ in the evident way using
  Remark~\ref{r:algmod}(\ref{r:algmod2}).

  \begin{rem}
    \label{r:shift}
    The topological subalgebra zeta function of $\cA$ as defined by du Sautoy
    and Loeser \cite[\S 8]{dSL04} coincides with the function
    $\zeta_{\cA,\topo}(\ess + d)$ in our notation.  The author feels that
    avoiding this shift yields a more natural definition.
  \end{rem}

  Recall that we use subscripts to denote base change.
  \begin{prop}
    \label{prop:basetop}
    Let $\bar k$ be an algebraic closure of $k$.  Let $k' \subset \bar k$ be
    another number field and let $\fo'$ be its ring of integers.  Let $\cA'$ be
    a non-associative $\fo'$-algebra which is free of finite rank as an
    $\fo'$-module.
    \begin{enumerate}
    \item If $k = k'$ and $\cA_k{\phantom '} \!\approx_k \cA'_k$, then
      $\Zeta_{\cA}$ and $\Zeta_{\cA'}$ are equivalent (see
      Definition~\ref{d:equivalent}).
    \item
      \label{prop:basetop2}
      If $\cA_{\bar k} \approx_{\bar k} \cA'_{\bar k}$, then
      $\zeta_{\cA,\topo}(\ess) = \zeta_{\cA',\topo}(\ess)$.
    \end{enumerate}
    Corresponding statements hold for $\Zeta_{\cE\acts M}$ and $\zeta_{\cE\acts
      M,\topo}(\ess)$.
  \end{prop}
  \begin{proof}
    \quad
    \begin{enumerate}
    \item Every $k$-isomorphism $\cA_k{\phantom '} \to \cA'_k$ is defined over
      $\fo_{\fp}$ for almost all primes $\fp$ of $k$.
    \item Every $\bar k$-isomorphism $\cA_{\bar k}{\phantom '}\to \cA_{\bar k}'$
      is defined over some common finite extension of $k$ and $k'$ and
      $\Zeta_{\cA,\topo}(\ess)$ is easily seen to be invariant under finite
      extensions of $k$. \qedhere
    \end{enumerate}
  \end{proof}

  Note that by the weak Nullstellensatz, we may replace $\bar k$ by any
  algebraically closed extension of $k$ in
  Proposition~\ref{prop:basetop}(\ref{prop:basetop2}).

  \begin{rem}
    \label{r:justify_voodoo}
    Suppose that there exists $W(\qq,\tee)\in \MM$ such that $\Zeta_{\cA}$ is
    equivalent to $\bigl(W(q_K^{\phantom 1},q_K^{-s})\bigr)_K$.  In view of
    \cite[7.11--7.15]{dSL04}, it is presently unclear if the existence of such a
    $W(\qq,\tee)$ is equivalent to uniformity of the subalgebra zeta functions
    of $\cA$ in the sense of \cite[\S 1.2.4]{dSW08} (for $k = \QQ$).  Given the
    existence of $W(\qq,\tee)$, the informal approach for reading off
    $\zeta_{\cA,\topo}(\ess)$ from $\zeta_{\cA\otimes_{\fo}{\fo_{\fp}}}(s)$
    given in the introduction is actually correct.
  \end{rem}

\paragraph{Nilpotent groups.}
Let $G$ be a torsion-free finitely generated nilpotent group.  Choose an
arbitrary full Lie lattice $\cL$ inside $\fL(G)$.  By
Theorem~\ref{thm:nilpotent}, it is sensible to define the \emph{topological
  subgroup zeta function} of $G$ as $\zeta^{\phantom o}_{G,\topo}(\ess) :=
\zeta^{\phantom o}_{\cL,\topo}(\ess)$ and the \emph{topological normal subgroup
  zeta function} of $G$ as $\zeta^{\normal}_{G,\topo}(\ess) :=
\zeta^{\normal}_{\cL,\topo}(\ess)$.  By Proposition~\ref{prop:basetop}, these
rational functions are actually invariants of the $\CC$-algebra
$\fL(G)\otimes_{\QQ} \CC$.

\section{Non-degeneracy II: computing topological zeta functions}
\label{s:top_eval}

As our second main result (Theorem~\ref{thm:topeval}), assuming that the set
$\ff$ of Laurent polynomials in Theorem~\ref{thm:evaluate} is \itemph{globally}
non-degenerate (see Definition~\ref{d:nd}(\ref{d:nd2})), we give an explicit
convex-geometric formula for the topological zeta function associated with the
$\fP$-adic integral $\Zeta^{\cC_0,\ff_0,\dotsc,\ff_m}_K(s_1,\dotsc,s_m)$ in
Definition~\ref{d:Zeta}.

\subsection{Splitting off torus factors and explicit formulae for Euler
  characteristics}
\label{ss:euler}

Let $(f_i)_{i\in I}$ be a finite collection of non-zero Laurent polynomials from
$\CC[X_1^{\pm 1},\dotsc,X_n^{\pm 1}]$.  For $J\subset I$, let $V_J = \{ \uu\in
\Torus^n(\CC) : \forall i\in I\colon f_i(\uu) = 0 \iff i\in J\}$.  Such Boolean
combinations of subvarieties of complex tori are $\CC$-analogues of the $\bar
V_{\bm g}^{\tau}$ in Definition~\ref{d:players}(\ref{d:players3}) (for fixed
$\tau$).  Write $d = \dim(\Newton(\prod_{i\in I}f_i))$.
Lemma~\ref{lem:split}(\ref{lem:split1}) shows that there are $\tilde f_i\in
\CC[X_1^{\pm 1},\dotsc,X_d^{\pm 1}]$ such that $V_J \approx \tilde V_J \times
\Torus^{n-d}(\CC)$ for all $J\subset I$, where $\tilde V_J = \{ \uu\in
\Torus^d(\CC) : \forall i\in I\colon \tilde f_i(\uu) = 0 \iff i\in J\}$.  For
globally non-degenerate $(f_i)_{i\in I}$, using the
Bernstein-Khovanskii-Kushnirenko Theorem (Theorem~\ref{thm:bkk}), in
Proposition~\ref{prop:relbkk}, we then derive explicit formulae for the Euler
characteristics $\Euler(\tilde V_J)$ in terms of the Newton polytopes
$(\Newton(f_i))_{i\in I}$.  Part~(\ref{lem:split1}) of the following lemma is
implicit in the proof of \cite[Prop.\ 2.5]{Oka90} while (\ref{lem:split2}) is
similar to \cite[Prop.\ 5.1]{Oka90}.

\begin{lemma}
  \label{lem:split}
  Let $k$ be any field and let $\ff = (f_i)_{i\in I}$ be a finite family of
  nonzero elements of $k[X_1^{\pm 1},\dotsc,X_n^{\pm 1}]$.  Let $\cN =
  \Newton\bigl(\prod_{i\in I}f_i\bigr)$ and let $\cC_0\subset \RR^n$ be a
  half-open rational cone.  Let $\GL_n(\ZZ)$ act on $k[X_1^{\pm
    1},\dotsc,X_n^{\pm 1}]$ via $(f,A)\mapsto f^A$ as in \S\ref{ss:monomial_sub}
  (with $X_i = \lambda_i = \xi_i$).

  \begin{enumerate}
  \item
    \label{lem:split1}
    Write $d := \dim(\cN)$.  For $i\in I$, choose $\alpha_i \in \supp(f_i)$.
    Then there exists $A\in\GL_n(\ZZ)$ such that $(\XX^{-\alpha_i} f_i)^A\in
    k[X_1^{\pm 1},\dotsc,X_d^{\pm 1}]$ for all $i\in I$.
  \item
    \label{lem:split2}
    Let $A\in \GL_n(\ZZ)$. Write $\ff^A := (f_i^A)_{i\in I}$.  Then $\ff$ is
    non-degenerate relative to $\cC_0$ if and only if $\ff^A$ is non-degenerate
    relative to $\cC_0 (A^{-1})^\top$.
  \item
    \label{lem:split3}
    Let $\tau\subset \cN$ be a $\cC_0$-visible face (see
    Definition~\ref{d:visible}) and define $f_i^\tau$ as in \S\ref{ss:nd}.

    Then $\Newton(\prod_{i \in I }f_i^\tau) = \tau$.  If $\ff$ is non-degenerate
    relative to $\cC_0$, then so is $(f_i^\tau)_{i\in I}$.
  \end{enumerate}
\end{lemma}
\begin{proof}
  \quad
  \begin{enumerate}
  \item Let $\hat f_i = \XX^{-\alpha_i} f_i$.  If $(\cP_i)_{i\in I}$ is a finite
    family of polytopes in $\RR^n$ with $0\in \bigcap_{i\in I}\cP_i$, then
    $\bigcup_{i\in I}\cP_i$ and $\sum_{i\in I}\cP_i$ span the same subspace of
    $\RR^n$.  Let $\cP_i := \Newton(\hat f_i)$ so that $\sum_{i\in I}\cP_i =
    \Newton\bigl(\prod_{i\in I}\hat f_i\bigr)$ has dimension $d$.  We conclude
    that the $\ZZ$-submodule, $M$ say, of $\ZZ^n$ generated by $\bigcup_{i\in I}
    \supp(\hat f_i)$ has rank $d$.
    The claim follows since $M$ is contained in a direct summand of $\ZZ^n$ of
    the same rank, namely $\langle M\rangle_{\QQ} \cap \ZZ^n$.
  \item It suffices to prove one implication.  Let $\bm h = (h_1,\dotsc,h_r)$
    for $h_1,\dotsc,h_r\in k[\XX^{\pm 1}] = k[X_1^{\pm 1},\dotsc,X_n^{\pm 1}]$.
    Using \S\ref{ss:monomial_sub} (and a base change to $k$), each $A\in
    \GL_n(\ZZ)$ induces a $k$-automorphism $A\colon \Torus_k^n(B) \to
    \Torus_k^n(B)$ for every commutative $k$-algebra $B$.  Taking $B =
    k[\XX^{\pm 1}]$, a simple calculation reveals $\diag(A\XX) \dtimes A \dtimes
    \diag(\XX^{-1})$ to be the Jacobian matrix of $A\XX =
    (\XX^{A_1},\dotsc,\XX^{A_n})$, where $A_1,\dotsc,A_n$ denote the rows of
    $A$.
    As $\bm h^A = (h_1^A,\dotsc,h_r^A) = (h_1(A\XX),\dotsc,h_r(A\XX))$, we see
    that the Jacobian matrix of $\bm h^A$ is given by $(\bm h^A)' = \bm h'(A\XX)
    \diag(A\XX) A \diag(\XX^{-1})$
    whence $(\bm h^A)'(\uu)$ and $h'(A\uu)$ have the same rank for all $\uu \in
    \Torus^n(\bar k)$, where $\bar k$ is an algebraic closure of $k$.
    Finally, $\init_\omega(f^A) = (\init_{\omega A^\top}(f))^A$ for all
    $\omega\in \RR^n$ and $f\in k[\XX^{\pm 1}]$.  By applying what has been said
    about general $\bm h$ to the various families $(\init_{\omega
      A^\top}(f_j))_{j\in J}$
    for $\omega \in \cC_0 (A^{-1})^\top$ and $J\subset I$, the ``only if'' part
    follows.
  \item Let $\omega \in \cC_0 \cap \NormalCone_\tau(\cN)$ be arbitrary.  The
    first part follows from
    \begin{align*}
      \Newton\Bigl(\prod_{i\in I} f_i^\tau\Bigr) & = \Newton\Bigl(\prod_{i\in I}
      \init_\omega(f_i)\Bigr) = \sum_{i\in I} \Newton(\init_\omega(f_i)) =
      \sum_{i\in I} \face_\omega(\Newton(f_i)) = \\ & =
      \face_\omega\Bigl(\sum_{i\in I}\Newton(f_i)\Bigr) = \face_\omega(\cN) =
      \tau.
    \end{align*}
    For the final statement, by combining \cite[Eqns\ (2.3),(2.5)]{Stu96}, for
    $\omega'\in \RR^n$, we obtain $\init_{\omega'}(f_i^\tau) = \init_{\omega +
      \varepsilon \omega'}(f_i)$ for sufficiently small $\varepsilon > 0$.
    Hence, given $\omega'\in \cC_0$, there exists $\omega'' \in \cC_0$ with
    $\init_{\omega'}(f_i^\tau) = \init_{\omega''}(f_i) = f_i^{\tau'}$ for all
    $i\in I$, where $\tau' := \face_{\omega''}(\cN)$.  \qedhere
  \end{enumerate}
\end{proof}

Let $\Vol^n$ denote the Lebesgue measure on $\RR^n$.  Recall that the
\emph{mixed volume} of a collection $(\cQ_1,\dotsc,\cQ_n)$ of non-empty
polytopes in $\RR^n$ is given by
\[
\MixedVolume^n(\cQ_1,\dotsc,\cQ_n)= \frac{(-1)^n}{n!} \sum_{\emptyset\not=
  I\subset\{1,\dotsc,n\}} (-1)^{\card I} \dtimes \Vol^n\!\left(\sum_{i\in
    I}\cQ_i\right);
\]
see \cite[Ch.\ 7, \S 4]{CLO'S05} for details but note that we use the
\itemph{normalised} mixed volume which satisfies $\MixedVolume^n(\cQ,\dotsc,\cQ)
= \Vol^n(\cQ)$ for every non-empty polytope $\cQ\subset \RR^n$.  The function
$\MixedVolume^n$ is symmetric and invariant under translation in each argument.

\begin{defn}
  \label{d:khovanskii}
  Given non-empty polytopes $\cP_1,\dotsc,\cP_r \subset \RR^n$, let
  \[
  \Khovanskii^n(\cP_1,\dotsc,\cP_r) := (-1)^{n+r} n!  \sum_{\substack{c_1 +
      \dotsb + c_r = n,\\c_1,\dotsc,c_r\in \NN}} \MixedVolume^n(
  \underbrace{\cP_1,\dotsc,\cP_1}_{c_1}, \dotsc,
  \underbrace{\cP_r,\dotsc,\cP_r}_{c_r}).
  \]
\end{defn}

Using Khovanskii's notation {\cite[p.\ 44]{Kho78}}, we have
$\Khovanskii^n(\cP_1,\dotsc,\cP_r)= \prod_{i=1}^r \cP_i(1+\cP_i)^{-1}$.  Note
that in the special case $r = 0$, we have $\Khovanskii^n() = 0$ for $n > 0$ and
$\Khovanskii^0() = 1$.

\begin{thm}[{\cite[\S 3, Thm\ 2]{Kho78}}]
  \label{thm:bkk}
  Let $f_1,\dotsc,f_r\in \CC[X_1^{\pm 1},\dotsc,X_n^{\pm 1}]$ be non-zero and
  suppose that $(f_1,\dotsc,f_r)$ is \Classically-non-degenerate (see
  Definition~\ref{d:capdegen}).  Define $V = \{ \uu\in \Torus^n(\CC) : f_1(\uu)
  = \dotsb = f_r(\uu) = 0 \}$.  Then the topological Euler characteristic of $V$
  is given by $\Euler(V) = \Khovanskii^n(\Newton(f_1),\dotsc,\Newton(f_r))$.
\end{thm}

As $\MixedVolume^n$ is symmetric, we can naturally extend the domain of
$\Khovanskii^n$ to include arbitrary (unordered) finite families of polytopes in
$\RR^n$.  If $\dim(\cP_1 + \dotsb + \cP_r) < n$ in
Definition~\ref{d:khovanskii}, then we necessarily have
$\Khovanskii^n(\cP_1,\dotsc,\cP_r) = 0$.  We now consider a relative version of
$\Khovanskii^n$ that allows us to extract useful information in such
low-dimensional situations.
Thus, for a $d$-dimensional linear subspace $U\le \RR^n$ defined over $\ZZ$, let
$\Vol^U$ be the measure on $U$ obtained from $\Vol^d$ by identifying $U$ and
$\RR^d$ via an arbitrary choice of a $\ZZ$-isomorphism $U\cap \ZZ^n \approx
\ZZ^d$.  For non-empty polytopes $\cQ_1,\dotsc,\cQ_d,\cP_1,\dotsc,\cP_r$
contained in $U$, define $\MixedVolume^U(\cQ_1,\dotsc,\cQ_d)$ and
$\Khovanskii^U(\cP_1,\dotsc,\cP_r)$ in the evident way.  Let $\bm\cP =
(\cP_i)_{i\in I}$ be a finite family of non-empty lattice polytopes in $\RR^n$.
For each $i\in I$, choose an arbitrary point $\xx_i \in \cP_i \cap \ZZ^n$.
Define $\cL(\bm\cP)$ to be the linear subspace of $\RR^n$ associated with the
affine hull of $\sum_{i\in I}\cP_i$ within $\RR^n$.  Hence, $\cL(\bm\cP) =
\bigl\langle \sum_{i\in I}( \cP_i - \xx_i)\bigr\rangle_{\RR} = \bigl\langle
\bigcup_{i\in I} (\cP_i - \xx_i)\bigr\rangle_{\RR}$ and
$\dim\bigl(\cL(\bm\cP)\bigr) = \dim\bigl(\sum_{i\in I}\cP_i\bigr)$.
\begin{defn}
  \label{d:rel_khovanskii}
  Given a finite family $\bm\cP = (\cP_i)_{i\in I}$ of non-empty lattice
  polytopes in $\RR^n$ and a subset $J\subset I$, let
  $\Khovanskii^{\rel}_J(\bm\cP) := \sum\limits_{J\subset T\subset I}
  (-1)^{\card{T}+\card{J}} \dtimes \Khovanskii_{\phantom J}^{\mathcal L(\bm\cP)}
  \!\bigl(\cP_t - \xx_t\bigr)_{t\in T}$.
\end{defn}
By translation invariance of mixed volumes, $\Khovanskii_J^{\rel}(\bm\cP)$ does
not depend on our choices of the $\xx_i$.  For a family $\ff = (f_i)_{i\in I}$
of Laurent polynomials, write $\Newton(\ff) := \bigl(\Newton(f_i)\bigr)_{i\in
  I}$.

\begin{prop}
  \label{prop:relbkk}
  Let $\ff = (f_i)_{i\in I}$ be a finite family of non-zero $f_i\in \CC[X_1^{\pm
    1},\dotsc,X_n^{\pm 1}]$.  Write $\cN = \Newton(\prod_{i\in I}f_i)$ and $d =
  \dim(\cN)$.  Let $\alpha_i \in \supp(f_i)$ and $A\in \GL_n(\ZZ)$ be as in
  Lemma~\ref{lem:split}(\ref{lem:split1}).  Write $\tilde f_i :=
  (\XX^{-\alpha_i} f_i)^A \in \CC[X_1^{\pm 1},\dotsc,X_d^{\pm 1}]$ for $i\in I$.
  For $J\subset I$, let $\tilde V_J = \{ \uu\in \Torus^d(\CC) : \forall i\in
  I\colon \tilde f_i(\uu) = 0 \iff i\in J\}$.  Suppose that $\ff$ is globally
  non-degenerate in the sense of Definition~\ref{d:nd}(\ref{d:nd2}).  Then
  $\chi(\tilde V_J) = \Khovanskii_J^{\rel}(\Newton(\ff))$.
\end{prop}
\begin{proof}
  For $T\subset I$, let $\tilde W_T = \{ \uu\in \Torus^d(\CC) : \forall t\in
  T\colon \tilde f_t(\uu) = 0\}$.  The additivity of topological Euler
  characteristics for (the associated analytic spaces of) $\CC$-varieties (via
  $\Euler_{\phantom c}\!\!(\dtimes) = \Euler_c(\dtimes)$; cf.\
  Remark~\ref{r:artin}) and the inclusion-exclusion principle yield
  $\Euler(\tilde V_J) = \sum_{J \subset T\subset I} (-1)^{\card T + \card J}
  \dtimes \Euler(\tilde W_T)$.  By Remark~\ref{r:nd}(\ref{r:nd2}) and
  Lemma~\ref{lem:split}(\ref{lem:split2}), $(\tilde f_i)_{i\in I}$ is globally
  non-degenerate.  Theorem~\ref{thm:bkk} therefore implies that $\Euler(\tilde
  W_T) = \Khovanskii^d(\Newton(\tilde f_t)_{t\in T})$ for $T\subset I$.
  Next, we have $\cL(\Newton(\ff)) \dtimes A = \RR^d \times \{0\}^{n-d} \subset
  \RR^n$
  and hence $\Khovanskii^d(\Newton(\tilde f_t)_{t\in T}) =
  \Khovanskii^{\cL(\Newton(\ff))} ( \Newton(f_t) - \alpha_t)_{t\in T} $ whence
  $\Euler(\tilde V_J) = \sum_{J \subset T\subset I} (-1)^{\card T + \card J}
  \dtimes \Khovanskii^{\cL(\Newton(\ff))}(\Newton(f_t)-\alpha_t)_{t\in T} =
  \Khovanskii_J^{\rel}(\Newton(\ff))$.
\end{proof}

\subsection{Topological evaluation of non-degenerate
  \texorpdfstring{$\fP$}{P}-adic integrals}

Using \S\ref{ss:euler}, we now rewrite the formulae provided by
Theorem~\ref{thm:evaluate} using the language of systems of local zeta functions
of Denef type developed in \S\ref{ss:denef_type}.  As the main result of this
section, we then obtain Theorem~\ref{thm:topeval}, the topological counterpart
of Theorem~\ref{thm:evaluate}.

As in \S\ref{ss:eval}, let $k$ be a number field with ring of integers $\fo$,
let $\cC_0 \subset \Orth^n$ be a half-open rational cone, and let $\ff_0 \subset
\fo[\XX^{\pm 1}]$ and $\ff_1,\dotsc,\ff_m \subset \fo[\XX] =
\fo[X_1,\dotsc,X_n]$ be non-empty finite sets with $0\not\in \ff := \ff_0 \cup
\dotsb \cup \ff_m$.

\begin{defn}
  \label{d:W}
  For a $\cC_0$-visible face $\tau$ of $\cN := \Newton(\prod\ff)$ and a subset
  $\bm g \subset \ff$, let $\cZ^{\cC^\tau_0(\bm g), \cP^\tau_1(\bm
    g),\dotsc,\cP^\tau_m(\bm g)}(\xi_0, \dotsc,\xi_m) \in
  \QQ(\xi_0,\dotsc,\xi_m)$ be as in Theorem~\ref{thm:evaluate} and define
  \[
  W^\tau_{\bm g}(\qq,\tee_1,\dotsc,\tee_m) := \qq^{-n}
  (\qq-1)^{n-\dim(\tau)+\card{\bm g}} \dtimes \cZ^{\cC^\tau_0(\bm g),
    \cP^\tau_1(\bm g),\dotsc,\cP^\tau_m(\bm g)}(\qq^{-1}\!,
  \tee_1,\dotsc,\tee_m).
  \]
\end{defn}

We will see in Lemma~\ref{lem:nd_denef}(\ref{lem:nd_denef1}) that each
$W^\tau_{\bm g}(\qq,\tee_1,\dotsc,\tee_m)$ belongs to the ring $\MM$ from
Definition~\ref{d:MM}.
\begin{thm}
  \label{thm:topeval}
  Let $k$ be a number field with ring of integers $\fo$.  Let $\cC_0 \subset
  \Orth^n$ be a half-open rational cone and let $\ff_0 \subset \fo[\XX^{\pm 1}]$
  and $\ff_1,\dotsc,\ff_m \subset \fo[\XX] = \fo[X_1,\dotsc,X_n]$ be non-empty
  finite sets with $0\not\in \ff := \ff_0 \cup\dotsb \cup \ff_m$.  Define
  $\Zeta^{\cC_0, \ff_0,\dotsc,\ff_m}_K(s_1,\dotsc,s_m)$ as in
  Definition~\ref{d:Zeta}.  For a $\cC_0$-visible face $\tau$ of $\cN :=
  \Newton(\prod \ff)$ (see Definition~\ref{d:visible} and \S\ref{ss:newton}) and
  a subset $\bm g \subset \ff$, define a rational function $W^\tau_{\bm
    g}(\qq,\tee_1,\dotsc,\tee_m)$ as in Definition~\ref{d:W}.  As in
  \S\ref{ss:binomial}, let $\red{W(\ess_1,\dotsc,\ess_m)}$ denote the formal
  reduction of $W(\qq,\tee_1,\dotsc,\tee_m)$ modulo $\qq-1$ obtained by
  expanding $W(\qq,\qq^{-\ess_1},\dotsc,\qq^{-\ess_m})$ using the formal
  binomial series.  For $f\in \ff$ and a $\cC_0$-visible face $\tau\subset \cN$,
  let $f^\tau$ be as in \S\ref{ss:eval}.  Define $\Khovanskii^{\rel}$ as in
  Definition~\ref{d:rel_khovanskii}.

  Suppose that $\ff$ is globally non-degenerate in the sense of
  Definition~\ref{d:nd}(\ref{d:nd2}).  Then the system of local zeta functions
  $\Zeta^{\cC_0,\ff_0,\dotsc,\ff_m} := \left(
    \Zeta^{\cC_0,\ff_0,\dotsc,\ff_m}_K(s_1,\dotsc,s_m) \right) $ is of Denef
  type (see Definition~\ref{d:denef_type} and
  Example~\ref{exs:denef_type}(\ref{exs:denef_type6})) and the associated
  topological zeta function in the sense of Definition~\ref{d:topzeta} is given
  by
  \[
  \Zeta^{\cC_0,\ff_0,\dotsc,\ff_m}_{\topo}(\ess_1,\dotsc,\ess_m) =
  \sum_{\substack{\bm g \subset \ff,\\\tau\subset \cN}} \Khovanskii^{\rel}_{\bm
    g}\bigl(\Newton(f^\tau)\bigr)_{f\in \ff} \dtimes \red{W^{\tau}_{\bm
      g}(\ess_1,\dotsc,\ess_m)} \in \QQ(\ess_1,\dotsc,\ess_m),
  \]
  the sum being taken over all $\cC_0$-visible faces $\tau$ of $\cN$.
\end{thm}
\begin{rem}
  In the spirit of \S\ref{ss:ndlit}, we may regard Theorem~\ref{thm:topeval} as
  a generalisation of the combinatorial formula for the (global) topological
  zeta function of a globally non-degenerate polynomial given by Denef and
  Loeser \cite[Thm\ 5.3(i)]{DL92}.  In fact, by applying the approach explained
  in the following to the explicit formula of Denef and Hoornaert \cite[Thm\
  4.2]{DH01}, we may recover the formula of Denef and Loeser verbatim.
\end{rem}

We will now embark upon a proof of Theorem~\ref{thm:topeval}.  First, define
half-open cones $\cC_0^\tau(\bm g)$ and polytopes $\cP^\tau_j(\bm g)$ as in
Definition~\ref{d:players}(\ref{d:players1})--(\ref{d:players2}).  For a
$\cC_0$-visible face $\tau \subset \cN := \Newton(\prod\ff)$ and a subset $\bm
g\subset \ff$, let $V^\tau_{\bm g} := \Spec\Bigl(\frac{k[\XX^{\pm 1}]}{\langle
  g^\tau : g\in \bm g\rangle}\Bigr) \setminus \Spec\Bigl(\frac{k[\XX^{\pm
    1}]}{\langle \prod_{f\in \ff\setminus\bm g}f^\tau \rangle}\Bigr) \subset
\Torus^n_k$; note that $V^{\tau}_{\bm g} \approx_k \Spec\Bigl(\frac{k[X_1^{\pm
    1},\dotsc,X_n^{\pm 1},Y]} {\langle g^\tau : g \in \bm g \rangle + \langle 1
  - Y \prod_{f\in \ff\setminus\bm g} f^\tau \rangle }\Bigr)$ is affine.  We see
that, denoting reduction modulo $\fp$ by $\bar\dtimes$ as in
\S\ref{ss:denef_type}, for almost all primes $\fp$ of $k$ and all field
extensions $\fk \supset \fo/\fp$, the set $\bar V^{\tau}_{\bm g}(\fk)$ coincides
with its namesake in Definition~\ref{d:players}(\ref{d:players3}).

For each $\cC_0$-visible face $\tau \subset \cN$, using
Lemma~\ref{lem:split}(\ref{lem:split1}),(\ref{lem:split3}), choose $A^\tau\in
\GL_n(\ZZ)$ and a point $\alpha^\tau(f)\in \supp(f^\tau)$ for each $f\in \ff$
such that $(\XX^{-\alpha^\tau(f)} f^\tau)^{A^\tau} \in k\Bigl[X_1^{\pm
  1},\dotsc,X_{\dim(\tau)}^{\pm 1}\Bigr]$ for all $f\in \ff$.  Having made these
choices, for $\bm g\subset \ff$, we define
\[
\widetilde V_{\bm g}^\tau := \Spec\Biggl(\frac{k\Bigl[X_1^{\pm
    1},\dotsc,X_{\dim(\tau)}^{\pm 1}\Bigr]} {\bigl\langle
  (\XX^{-\alpha^\tau(g)}g^\tau)^{A^\tau} : g\in \bm g\bigr\rangle} \Biggr)
\setminus \Spec\Biggl(\frac{k\Bigl[X_1^{\pm 1},\dotsc,X_{\dim(\tau)}^{\pm
    1}\Bigr]} {\bigl\langle \prod_{f\in \ff\setminus\bm g}
  (\XX^{-\alpha^\tau(f)}f^\tau)^{A^\tau} \bigr\rangle}\Biggr) \subset
\Torus^{\dim(\tau)}_k.
\]

\begin{lemma}
  \label{lem:nd_denef}
  \quad
  \begin{enumerate}
  \item
    \label{lem:nd_denef1}
    $W_{\bm g}^\tau(\qq,\tee_1,\dotsc,\tee_m) \in \MM$
    for all $\cC_0$-visible faces $\tau\subset \cN$ and subsets $\bm g\subset
    \ff$.
  \item
    \label{lem:nd_denef2}
    Notation as in Theorem~\ref{thm:evaluate}.  Suppose that $\ff$ is
    non-degenerate relative to $\cC_0$.  If $K \supset k$ is a non-Archimedean
    local field, then, unless $\fP \cap \fo$ belongs to some finite exceptional
    set, we have
    \[
    \Zeta^{\cC_0,\ff_0,\dotsc,\ff_m}_K(s_1,\dotsc,s_m) = \sum_{\substack{\bm g
        \subset \ff,\\\tau\subset \cN}} \noof{\bar{\tilde V}^{\tau}_{\bm
        g}(\fO/\fP)} \dtimes W^{\tau}_{\bm g}(q,q^{-s_1},\dotsc,q^{-s_m})
    \]
    for $s_1,\dotsc,s_m\in \CC$ with $\Real(s_j)\ge 0$.

    In particular, the system of local zeta functions
    $\Zeta^{\cC_0,\ff_0,\dotsc,\ff_m}$ consisting of the meromorphic
    continuations of the $\Zeta^{\cC_0,\ff_0,\dotsc,\ff_m}_K(s_1,\dotsc,s_m)$ is
    of Denef type and
    $$\Zeta^{\cC_0,\ff_0,\dotsc,\ff_m}_{\topo}(\ess_1,\dotsc,\ess_m) =
    \sum\limits_{\substack{\bm g \subset \ff,\\\tau\subset \cN}} \Euler(\tilde
    V_{\bm g}^\tau(\CC)) \dtimes \red{W_{\bm g}^\tau(\ess_1,\dotsc,\ess_m)}.
    $$
  \end{enumerate}
  \end{lemma}
  \begin{proof}
    \quad
    \begin{enumerate}
    \item For a $\cC_0$-visible face $\tau \subset \cN$ and $\bm g \subset \ff$,
      we have $\cC_0^\tau(\bm g) \subset \NormalCone_{\tau}(\cN)
      \times\StrictOrth^{\bm g}$ and hence $\dim(\cC_0^\tau(\bm g)) \le n -
      \dim(\tau) + \card{\bm g}$.  The claim follows from the definition of
      $\cZ^{\cC^\tau_0(\bm g), \cP^\tau_1(\bm g),\dotsc,\cP^\tau_m(\bm
        g)}(\qq^{-1}\!, \tee_1,\dotsc,\tee_m)$ (see Definition~\ref{def:Z}), the
      remarks following Theorem~\ref{thm:genfun}, and \S\ref{ss:monomial_sub}.
      
    \item The first part follows from Theorem~\ref{thm:evaluate}
      and $V_{\bm g}^\tau \approx_k \widetilde V_{\bm g}^\tau \times_k
      \Torus^{n-\dim(\tau)}_k$.  For the remainder, we additionally invoke
      (\ref{lem:nd_denef1}) and Theorem~\ref{thm:top}(\ref{thm:top3}).  Strictly
      speaking, Theorem~\ref{thm:evaluate} does not prove rationality of
      $\Zeta_K^{\cC_0,\ff_0,\dotsc,\ff_m}(s_1,\dotsc,s_m)$ when $\fP \cap \fo$
      is an exceptional prime, but see the comments after
      Definition~\ref{d:denef_type} and note that exceptional primes are
      irrelevant for our purposes anyway.  \qedhere
    \end{enumerate}
  \end{proof}

\begin{proof}[Proof of Theorem~\ref{thm:topeval}]
  Fix a $\cC_0$-visible face $\tau \subset \cN$ and $\bm g\subset \ff$.  Let $d
  = \dim(\tau)$.  For $f\in \ff$, write $\tilde f^\tau := (\XX^{-\alpha^\tau(f)}
  f^\tau)^{A^\tau}$.  By Lemma~\ref{lem:split}(\ref{lem:split3}),
  $(f^\tau)_{f\in \ff}$ is globally non-degenerate.  Since $\tilde V_{\bm
    g}^\tau(\CC) = \{ \uu\in \Torus^d(\CC) : \forall f\in \ff\colon \tilde
  f^\tau(\uu) = 0 \iff f\in \bm g\}$, Proposition~\ref{prop:relbkk} shows that
  $\Euler(\tilde V_{\bm g}^\tau(\CC)) = \Khovanskii^{\rel}_{\bm
    g}\bigl(\Newton(f^\tau)\bigr)_{f\in \ff}$.  The claim now follows from
  Lemma~\ref{lem:nd_denef}(\ref{lem:nd_denef2}).
\end{proof}

\begin{rem}
  \label{r:top_easier}
  As far as explicit computations are concerned, evaluating the sum in
  Theorem~\ref{thm:topeval} is often considerably easier than $\fP$-adic
  computations using Theorem~\ref{thm:evaluate} for two reasons.  First, in
  contrast to the determination of the numbers $\noof{\bar V_{\bm
      g}^\tau(\fO/\fP)}$ (as $K$, hence $\fO$, varies) in
  Theorem~\ref{thm:evaluate}, the computation of the Euler characteristics
  $\Khovanskii^{\rel}_{\bm g}\bigl(\Newton(f^\tau)\bigr)_{f\in \ff}$ in
  Theorem~\ref{thm:topeval} is a purely mechanical process that can be performed
  by a computer.  Second, while the rational functions occurring in
  Theorem~\ref{thm:evaluate} admit concise descriptions in terms of
  convex-geometric objects, when written as quotients of polynomials, they often
  become too unwieldy for further computations.  Passing to the ``reductions
  modulo $q-1$'' of rational functions mitigates this to a large extent.
\end{rem}

\section{Applications}
\label{s:app}

We now apply the techniques for computing $\fP$-adic integrals and topological
zeta functions developed in the present article to some specific examples.  Our
primary focus will be on the case of topological subalgebra zeta functions.
However, to illustrate Theorem~\ref{thm:evaluate}, we also work through a known
$\fP$-adic example.

\subsection{Example: \texorpdfstring{$\Sl_2(\ZZ)$}{sl(2,ZZ)} -- revisited}
\label{ss:sl2}

Let $\mathfrak{gl}_n(\ZZ)$ be the Lie ring of all integral $n\times n$ matrices
with the usual Lie bracket $[A,B] = AB - BA$.  Let $\Sl_n(\ZZ)$ be the Lie
subring of $\mathfrak{gl}_n(\ZZ)$ consisting of traceless matrices.  We now
illustrate the key ingredients of Theorem~\ref{thm:evaluate} in the context of
the well-studied example $\Sl_2(\ZZ)$.

\paragraph{History.}
The subalgebra zeta function of $\Sl_2(\ZZ_p)$ was first recorded (for odd $p$)
by du Sautoy \cite{dS00} based, among other things, on elaborate computations
carried out by Ilani \cite{Il99}.  For $p\not= 2$, we have
\begin{equation}
  \label{eq:sl2p}
  \zeta_{\Sl_2(\ZZ_p)} = 
  \frac{\zeta_p(s) \zeta_p(s-1)\zeta_p(2s-2)\zeta_p(2s-1)}{\zeta_p(3s-1)},
\end{equation}
where $\zeta_p(s) = 1/(1-p^{-s})$.  Later, du Sautoy and Taylor \cite{dST02}
confirmed \eqref{eq:sl2p} by different means and also provided the missing case
$p = 2$.  Their approach proceeds by computing an associated $p$-adic integral
(see Theorem~\ref{thm:coneint}) via a manual resolution of singularities,
obtained as a sequence of blow-ups with judiciously chosen centres.  This
geometrically-minded computation was later reinterpreted in a motivic setting by
du Sautoy and Loeser, allowing them to deduce that
\begin{equation}
  \label{eq:sl2top}
  \zeta_{\Sl_2(\ZZ),\topo}(\ess) = 
  \frac{3\ess - 1}{2(2\ess - 1)(\ess - 1)^2\ess},
\end{equation}
see \cite[\S 9.3]{dSL04} and note that we already corrected the shift present in
\cite{dSL04}, see Remark~\ref{r:shift}.  For yet another $p$-adic approach,
Klopsch and Voll \cite{KV09} gave an explicit formula for $\zeta_{\cL}(s)$,
where $\cL$ is Lie $\ZZ_p$-algebra which is free of rank $3$ as a
$\ZZ_p$-module, in terms of Igusa's local zeta function associated with a
ternary quadratic form attached to $\cL$.

\paragraph{Setting up the integral.}
We now explain how the $p$-adic integral in \cite{dST02} can, for odd $p$, be
computed using our method.  Let $K$ be a finite extension of $\QQ_p$.  As
before, we follow the conventions explained in Notation~\ref{not:local}.  First,
by applying Theorem~\ref{thm:coneint} and
Remark~\ref{r:coneint}(\ref{r:coneint2}) with respect to the $\ZZ$-basis
$\bigl(\begin{mimat}0&0\\1&0\end{mimat},\begin{mimat}0&1\\0&0\end{mimat},\begin{mimat}1&0\\0&-1\end{mimat}\bigr)$
of $\Sl_2(\ZZ)$, as stated in \cite[\S 2]{dST02} (for $K = \QQ_p$), we obtain
\begin{equation}
  \label{eq:sl2}
  \zeta_{\Sl_2(\fO)}(s) = (1-q^{-1})^{-3} \int_{V(\fO)}
  \abs{x_{11}}^{s-1} \abs{x_{22}}^{s-2} \abs{x_{33}}^{s-3} \dd\mu(\xx),
\end{equation}
where
$$
V(\fO) = \left\{ \xx = \left[\begin{smallmatrix} x_{11} & x_{12} & x_{13} \\ &
      x_{22} & x_{23} \\ & & x_{33}
    \end{smallmatrix}\right] \in \Tr_3(\fO) \cap \Torus^6(K):
  \norm{ 4x_{12}^{\phantom 1}x_{22}^{-1} x_{23}^{\phantom 1}, \,\,
    4x_{12}^{\phantom 1} x_{22}^{-1} x_{33}^{\phantom 1}, \,\, f(\xx) }\le 1
\right\},
$$
$f(\XX) = X_{11}^{\phantom 1}X_{22}^{\phantom 1} X_{33}^{-1} + 4
X_{13}^{\phantom 1} X_{23}^{\phantom 1} X_{33}^{-1} - 4X_{12}^{\phantom 1}
X_{22}^{-1} X_{23}^2 X_{33}^{-1}$, and we identify $\Tr_3(K) \approx K^6$ via
$\xx \mapsto (x_{11},x_{12},x_{13},x_{22},x_{23},x_{33})$.  We henceforth assume
that $p\not= 2$.  Let
\[
\cC_0 = \left\{ (\omega_{11},\omega_{12},\dotsc,\omega_{33})\in \Orth^6 :
  \omega_{12} - \omega_{22} + \omega_{23} \ge 0, \,\, \omega_{12} - \omega_{22}
  + \omega_{33} \ge 0 \right\}
\]
so that $V(\fO) = \{ \xx \in \Torus^6(K) : \nu(\xx) \in \cC_0, \abs{f(\xx)} \le
1\}$.  In order to keep technicalities in our exposition at a minimum and in
view of the comparatively simple shape of the integral in \eqref{eq:sl2}, in the
following, we will directly carry out the relevant steps from the proof of
Theorem~\ref{thm:evaluate}.  In particular, we apply the relevant
specialisations from Remark~\ref{r:eval_and_algmod} right from the start and we
also describe all rational functions in terms of generating functions of
half-open cones instead of invoking the language of
\S\ref{ss:zeta_cone_polytopes}.

\paragraph{Breaking up the integral.}
The Newton polytope $\Delta := \Newton(f) \subset \RR^6$ is a triangle and one
checks that $f$ is globally non-degenerate over $\QQ$.  The same is true of the
reduction of $f$ modulo every odd prime.  As $(1,\dotsc,1)$ is an interior point
of $\cC_0$ and $f$ is homogeneous, each of the $7$ faces of $\Delta$ is
$\cC_0$-visible by Remark~\ref{r:nd}(\ref{r:nd4}).  Let us write $\bm a :=
X_{11}^{\phantom 1}X_{22}^{\phantom 1} X_{33}^{-1}$, $\bm b := 4X_{13}^{\phantom
  1} X_{23}^{\phantom 1} X_{33}^{-1}$, and $\bm c := -4X_{12}^{\phantom 1}
X_{22}^{-1} X_{23}^2 X_{33}^{-1}$ for the terms of $f$.  We further use the same
symbols to denote the corresponding vertices of $\Delta$ and let $\bm a\bm b$,
$\bm a\bm c$, and $\bm b\bm c$ denote the edges of $\Delta$ suggested by the
notation.  Note that in the present situation, each non-zero ``sub-polynomial''
of $f$ gives rise to a face of $\Delta$.  As in the proof of
Proposition~\ref{prop:Zseries}, for a face $\tau\subset \Delta$, define
$\cC_0^\tau := \cC_0 \cap \NormalCone_{\tau}(\Delta)$.  Let $V^\tau(\fO) := \{
\xx \in \Torus^6(K) : \nu(\xx) \in \cC_0^\tau, \abs{f(\xx)} \le 1\}$ and
$I^\tau(s) := \int_{V^\tau(\fO)} \abs{x_{11}}^{s-1} \abs{x_{22}}^{s-2}
\abs{x_{33}}^{s-3} \dd\mu(\xx)$ so that $(1-q^{-1})^3 \zeta_{\Sl_2(\fO)}(s) =
\sum_\tau I^\tau(s)$.  Write $\alpha,\beta,\gamma\in \ZZ^6$ for the exponent
vectors of the monomials in $\bm a$, $\bm b$, and $\bm c$, respectively.  We
thus, for example, have $\cC_0^{\bm a \bm b} = \{ \omega\in \cC_0^{\phantom a} :
\bil\alpha\omega = \bil\beta\omega < \bil\gamma\omega\}$.

\paragraph{Computing the pieces: vertices.}
We now describe the computation of $I^{\bm a}(s)$.  The other two vertices $\bm
b$ and $\bm c$ of $\Delta$ can be treated in the same way.  Thus, if $\omega \in
\cC_0^{\bm a}\cap\ZZ^6$ and $\xx\in \Torus^6(K)$ with $\nu(\xx) = \omega$, then
$\nu(f(\xx)) = \bil{\alpha}\omega$ as $\init_\omega(f) = \bm a$ does not vanish
on $\Torus^6(\fO/\fP)$.  Letting $(\dtimes)^*$ denote dual cones and writing
$\genfun{\dtimes}$ for generating functions of half-open cones (see
\S\ref{ss:cones}), as in the proof of Proposition~\ref{prop:monomial_integral},
we find that
\begin{align}
  \nonumber I^{\bm a}(s) & = \int_{\{ \xx\in \Torus^6(K) : \nu(\xx) \in
    \cC_0^{\bm a} \cap \{ \alpha\}^*\}} \abs{x_{11}}^{s-1} \abs{x_{22}}^{s-2}
  \abs{x_{33}}^{s-3} \dd\mu(\xx) \\
  \label{eq:Zas}
  & = (1-q^{-1})^6 \underbrace{ \genfun{\cC_0^{\bm a} \cap \{\bm a\}^*} \left(
      \left[
        \begin{smallmatrix}
          q^{-s} & q^{-1}&  q^{-1}\\
          & q^{1-s}& q^{-1}\\
          & & q^{2-s}
        \end{smallmatrix}
      \right] \right) }_{=: Z^{\bm a}(s)},
\end{align}
where the substitution is to be understood in the natural way---that is, within
$\genfun{\cC_0^{\bm a} \cap \{\bm a\}^*}(\xi_{11},\dotsc,\xi_{33})$, replace
$\xi_{ij}$ by the entry in position $(i,j)$ of the matrix in \eqref{eq:Zas}.

\paragraph{Computing the pieces: edges.}
We now consider the computation of $I^{\bm a\bm b}(s)$, the cases of the other
two edges $\bm a\bm c$ and $\bm b\bm c$ of $\Delta$ being essentially identical.
Let $\omega \in \cC_0^{\bm a\bm b} \cap \ZZ^6$ and $\uu \in \Torus^6(\fO)$ be
arbitrary; note that $\bil\alpha\omega = \bil\beta\omega$.  Using
\S\ref{ss:newton}, we see that $f(\pi^\omega \uu) = \pi^{\bil\alpha\omega}
\dtimes\bigl( (\bm a + \bm b)(\uu) + \cO(\pi)\bigr)$.  Therefore, if $(\bm a +
\bm b)(\uu) \not\equiv 0 \bmod\fP$, then $\nu(f(\pi^\omega \xx)) =
\bil\alpha\omega$ for all $\xx \in \uu + \fP^6$.  Suppose that $(\bm a + \bm
b)(\uu) \equiv 0 \bmod\fP$.  Then non-degeneracy of $f$ and Hensel's lemma allow
us to to replace $\uu$ by a possibly different element of $\uu+ \fP^6$ in such a
way that the congruence $(\bm a + \bm b)(\uu) \equiv 0 \bmod \fP$ becomes an
equality $f(\pi^\omega\uu) = 0$.  Next, after a suitable local change of
coordinates,
we may even assume that $f(\pi^\omega \xx) = \pi^{\bil\alpha\omega} (x_1-u_1)$
for all $\xx \in \uu + \fP^6$, see the proof of Theorem~\ref{thm:evaluate} for
details.  The number of solutions of $(\bm a + \bm b)(\bar\uu) \equiv 0\bmod
\fP$ for $\bar\uu\in \Torus^6(\fO/\fP)$ is $(q-1)^5$, as we can e.g.\ solve for
$\bar u_{11}$.  We conclude that {\small
  \begin{align*}
    I^{\bm a\bm b}(s) = & \frac{(q-1)^6 - (q-1)^5}{q^6} \dtimes
    \underbrace{\genfun{\cC_0^{\bm a\bm b} \cap \{\bm
        a\}^*}\left(\dotso\right)}_{=: Z^{\bm a\bm b}_{\off}(s)} +
    \frac{(q-1)^{5+1}}{q^6} \dtimes \underbrace{\genfun{ (\cC_0^{\bm a\bm b}
        \!\times\!  \StrictOrth) \cap \{ (\alpha,\!1)\}^*}(\dotso)}_{=:
      Z_{\on}^{\bm a \bm b}(s)},
  \end{align*}} where the substitutions are as in the case of $I^{\bm a}(s)$
above, except that the additional variable in the second summand is also
replaced by $q^{-1}$.  The ``$\on\!/\!\off$''-types of rational functions,
indicating the vanishing or non-vanishing of the initial form $\bm a + \bm b$ of
$f$, correspond to the two subsets $\bm g\subset \{ f\}$ in
Theorem~\ref{thm:evaluate}, see below.

\paragraph{Computing the pieces: the triangle.}
The computation for the triangle $\tau = \Delta$ is very similar to those for
the edges.  It is easy to see that the number of points $\bar\uu\in
\Torus^6(\fO/\fP)$ with $f(\bar \uu) \equiv 0 \bmod \fP$ is $(q-1)^4(q-2)$.

\paragraph{Summing up.}
In summary, we obtain rational functions $Z^{\tau}_{\on\!/\!\off}(s) \in
\QQ(q,q^{-s})$ with
\begin{align}
  \label{eq:sl2toric}
  \nonumber (1-q^{-1})^3 \zeta_{\Sl_2(\fO)}(s) = & (1-q^{-1})^6 \dtimes ( Z^{\bm
    a}(s) + Z^{\bm b}(s) + Z^{\bm c}(s) + Z^{\bm a\bm b}_{\on}(s) + Z^{\bm a\bm
    c}_{\on}(s) + Z^{\bm b \bm c}_{\on}(s)) +
  \\
  \nonumber & (1-q^{-1})^5(1-2q^{-1}) \dtimes (Z^{\bm a\bm b}_{\off}(s) + Z^{\bm
    a \bm
    c}_{\off}(s) + Z^{\bm b\bm c}_{\off}(s)) + \\
  & (1-q^{-1})^5(1-2q^{-1}) Z^{\Delta}_{\on}(s) +
  (1-q^{-1})^4(1-3q^{-1}+3q^{-2}) Z^{\Delta}_{\off}(s).
\end{align}
All that remains to be done to finish our computation of $\zeta_{\Sl_2(\fO)}(s)$
is to explicitly evaluate the right-hand side of \eqref{eq:sl2toric}.  In order
to do that, we only need to compute generating functions of certain half-open
cones which can e.g.\ be determined using the computer program \texttt{LattE}
\cite{DLHTY04}.  Carrying out this computation, after substantial cancellations,
we find that
\begin{equation}
  \label{eq:sl2O}
  \zeta_{\Sl_2(\fO)}(s) =
  \frac
  {1- q^{1-3s}}{(1- q^{1-2s})(1 - q^{2-2s})(1-q^{1-s})(1-q^{-s})},
\end{equation}
which of course agrees with \eqref{eq:sl2p} for $\fO = \ZZ_p$; recall that we
assumed that $\fP \cap \ZZ \not= \langle 2\rangle$.

\paragraph{}
Remark~\ref{r:eval_and_algmod} allows us to compare \eqref{eq:sl2toric} and the
formula in Theorem~\ref{thm:evaluate}.  Since the integrand in \eqref{eq:sl2}
consists of monomials, the only subsets $\bm g \subset \{ f,
X_{11},X_{22},X_{33}\}$ that can possibly give rise to a non-zero number
$\noof{\bar V^{\tau}_{\bm g}(\fO/\fP)}$ are precisely the subsets of $\{
f\}$---the ``$\on\!/\!\off$-types'' of rational functions correspond to those.
Note that if $\tau$ is a vertex, then the corresponding initial form of $f$ does
not vanish on the torus.  This explains the absence of an
``$\on\!/\!\off$''-distinction for vertices.

Having carried out the above computation for arbitrary $\fO$ with $\fP\cap
\ZZ\not= \langle 2\rangle$, it is a straightforward matter to confirm the
associated topological zeta function given in \eqref{eq:sl2top}.  For larger
examples, $\fP$-adic computations often become considerably more involved than
their topological counterparts (see Remark~\ref{r:top_easier}).  From now on, we
will focus exclusively on the topological case.

\subsection{Example: from the Heisenberg Lie ring to quadratic forms}
\label{ss:Heisenberg}

Let $k$ be a number field with ring of integers $\fo$.  Let $M$ be a free
$\fo$-module of finite rank endowed with a bilinear form $\beta\colon M\times M
\to \fo$.  We assume that $\beta$ is non-degenerate over $k$.  Define an
$\fo$-algebra $\cB(\beta)$ with underlying $\fo$-module $M \oplus \fo$ and
multiplication $(\xx,a)(\yy,b) = (0,\beta(\xx,\yy))$.  Such algebras were
considered in \cite[\S 3.3]{dSL96} in the context of a different type of zeta
function.

\begin{ex}
  \label{ex:Heisenberg}
  If $k = \QQ$, $M = \ZZ^{2m}$ and $\beta$ is the standard symplectic form
  $\beta(\xx,\yy) = \xx \dtimes \left[\begin{smallmatrix} 0_m & 1_m \\ -1_m &
      0_m \end{smallmatrix}\right] \dtimes \yy^\top$, then $\cB(\beta)$ is a
  central product $\cB(\beta) \approx \fh \curlyvee \dotsb \curlyvee \fh$ of $m$
  copies of the Heisenberg Lie ring $\fh = \left[\begin{smallmatrix} 0 & \ZZ &
      \ZZ \\ & 0 & \ZZ \\ & & 0
    \end{smallmatrix}\right] \le \mathfrak{gl}_3(\ZZ)$.
\end{ex}

\begin{rem}
  Using the conventions in Notation~\ref{not:local}, for a non-Archimedean local
  field $K\supset k$, the subalgebra zeta function $\zeta_{\cB(\beta)
    \otimes_{\fo}\fO}(s)$ is closely related to arithmetic properties of the
  bilinear forms induced by $\beta$ on submodules of $M_{\fO}$.  Namely, for
  each $\fO$-submodule $\Lambda \le M_{\fO}$ of finite index, define $m(\Lambda)
  := \min(\nu(\beta(\xx,\yy)) : \xx,\yy\in \Lambda) < \infty$ and consider the
  bivariate zeta function
  \[
  \Zeta^\beta(s_1,s_2) = \sum_{\Lambda} \idx{M_{\fO}:\Lambda}^{-s_1} q^{-
    m(\Lambda)s_2}.
  \]

  Then it is easy to see that $ (q^{d-s}-1) \zeta_{\cB(\beta)\otimes_{\fo}
    \fO}(s) = q^{d-s} \Zeta^\beta(s,s-d) - \zeta_{(\fO^d,0)}(s) $, where
  $(\fO^d,0)$ denotes $\fO^d$ endowed with the zero multiplication and $d =
  \rank_{\fo}(M)$.
\end{rem}

\paragraph{Heisenberg Lie rings.}
Let $K$ be a $p$-adic field with associated objects as in
Notation~\ref{not:local}.  The subalgebra zeta function of $\fh_{\fO}$ was among
the very first examples computed, see \cite[Prop.\ 8.1]{GSS88}; their formula
for $\fO = \ZZ_p$ and its proof carry over to the present case.  The topological
result, $\zeta_{\fh,\topo}(\ess) = \frac{3}{2(2\ess - 3)(\ess - 1)\ess}$, has
been recorded in \cite[\S 9.2]{dSL04} (up to the shift explained in
Remark~\ref{r:shift}).  The subalgebra zeta function of $\fh_{\ZZ_p} \curlyvee
\fh_{\ZZ_p}$ was first computed by Woodward, see \cite[p.~\!{57}]{Woo05} for the
result.

We now sketch the computation of $\zeta_{\fh \curlyvee \fh,\topo}(\ess)$ using
our method.  First, we choose a $\ZZ$-basis $(\ee_1,\dotsc,\ee_5)$ of $\fh
\curlyvee \fh$ which satisfies (a) $[\ee_1,\ee_3] = [\ee_2,\ee_4] = \ee_5$ and
(b) all other commutators of basis elements not implied by anti-symmetry are
zero.  Next, Theorem~\ref{thm:coneint} and
Remark~\ref{r:coneint}(\ref{r:coneint2}) provide us with a description of
$\zeta_{\fh_{\fO}\curlyvee \fh_{\fO}}(s)$ in terms of a $\fP$-adic integral.  As
illustrated in the case of $\Sl_2(\ZZ)$ in \S\ref{ss:sl2}, by discarding
finitely many primes, we can always move monomial divisibility conditions from
the domain of integration into the ambient cone $\cC_0$.  Moreover, since the
integrand in Theorem~\ref{thm:coneint} is always comprised of monomials, it is
without relevance for questions of non-degeneracy.  In the present context, a
simple computation reveals the essential ingredient of the $\fP$-adic integral
in Theorem~\ref{thm:coneint} to be the condition
\[
\norm{\underbrace{x_1^{\phantom 1} x_{10}^{\phantom 1} x_{15}^{-1}+
    x_2^{\phantom 1} x_{11}^{\phantom 1} x_{15}^{-1}}_{=:f_1(\xx)}, \quad
  \underbrace{x_4^{\phantom 1} x_6^{\phantom 1}x_{15}^{-1} - x_1^{\phantom 1}
    x_7^{\phantom 1}x_{15}^{-1} - x_2^{\phantom 1} x_8^{\phantom 1}
    x_{15}^{-1}}_{=:f_2(\xx)}} \le 1,
\]
where $\xx\in \Torus^{15}(K) \cap \fO^{15}$ and we identified $\Tr_5(K) \approx
K^{15}$.  One can check that $\Newton(f_1 f_2)$ is a $3$-dimensional polytope
with $6$ vertices and a total of $21$ faces.  Most importantly $(f_1,f_2)$ is
globally non-degenerate.  By following through our procedure
(Theorem~\ref{thm:topeval}) with the help of a computer, we find that
\[
\zeta_{\fh \curlyvee \fh,\topo}(\ess) = \frac{17\ess - 21}{3(3\ess - 4)(3\ess -
  7)(\ess - 3)(\ess-2)(\ess-1)\ess},
\]
which is exactly what one would expect given Woodward's $p$-adic formula and the
``informal approach'' from the introduction.

While we shall not concern ourselves with it in any detail, the complexity of
the integral under consideration here is sufficiently low to allow for
computer-assisted $\fP$-adic computations.  In order to adapt our computations
for $\Sl_2(\fO)$ in \S\ref{ss:sl2} to the case of $\fh_{\fO} \curlyvee
\fh_{\fO}$, one would replace the ``$\on\!/\!\off$''-distinction used there by
pairs $(\on\!/\!\off, \on\!/\!\off)$ indicating exactly which of the initial
forms of $f_1$ and $f_2$ vanish at a given point of $\Torus^{15}(\fO/\fP)$.
This suggests a total number of $4\dtimes 21$ cases but, as in \S\ref{ss:sl2},
this number can be reduced slightly because certain subvarieties of tori are a
priori known to be empty.

\paragraph{Commutative variants of Heisenberg rings.}
Let $\mathfrak{ch}_d$ be the $\ZZ$-algebra $\cB(\beta_d)$, where $\beta_d$ is
the standard inner product $(\xx,\yy) \mapsto \sum_{i=1}^d x_i y_i$ on $\ZZ^d$.
In view of Example~\ref{ex:Heisenberg}, we regard the rings $\mathfrak{ch}_d$ as
(non-unital) commutative analogues of the central products of Heisenberg Lie
rings considered above.  We will now describe how we computed
\begin{align*}
  \zeta_{\mathfrak{ch}_2,\topo}(\ess)
  & = \frac{3\ess - 2}{4(\ess - 1)^3\ess}, \\
  \zeta_{\mathfrak{ch}_3,\topo}(\ess)
  & = \frac{14}{(5\ess - 6)(5\ess - 7)(\ess - 2)\ess}, \text{ and}\\
  \zeta_{\mathfrak{ch}_4,\topo}(\ess) & = \frac{17\ess - 25}{(3\ess - 4)(3\ess -
    5)^2(\ess - 3)(\ess - 2)\ess}
\end{align*}
using Theorem~\ref{thm:topeval} and some elementary insight into the structure
of $\mathfrak{ch}_d$.  To the author's knowledge, these topological zeta
functions have not been previously computed, nor are their $\fP$-adic versions
known.

Write $\mathfrak{ch}_d = \ZZ \ee_1 \oplus \dotsb \oplus \ZZ \ee_d \oplus \ZZ
\vv$, where $\ee_i^2 = \vv$ and all other products of basis vectors are trivial.
We again follow the conventions in Notation~\ref{not:local}.  By using
Theorem~\ref{thm:coneint} and Remark~\ref{r:coneint}(\ref{r:coneint2}) relative
to the $\ZZ$-basis $(\ee_1,\ee_2,\vv)$ of $\mathfrak{ch}_2$, we express
$\zeta_{\mathfrak{ch}_2(\fO)}(s)$ in terms of a $\fP$-adic integral.  If we
identify $\Tr_3(K) \approx K^6$ as above, the domain of integration consists of
those $\xx\in \fO^6$ with $\norm{x_1^2x_6^{-1} + x_2^2x_6^{-1}, x_4^2x_6^{-1},
  x_2^{\phantom 1}x_4^{\phantom 1}x_6^{-1}} \le 1$.  As before, monomials do not
affect questions of degeneracy.  Moreover, binomials are always globally
non-degenerate so the computation of $\zeta_{\mathfrak{ch}_2,\topo}(\ess)$
becomes an essentially trivial matter.  In contrast, the $\ZZ$-basis
$(\ee_1,\ee_2,\ee_3,\vv)$ of $\mathfrak{ch}_3$ gives rise to a degenerate system
of polynomials.

In order to nonetheless be able to compute $\zeta_{\mathfrak{ch}_3,\topo}(\ess)$
and $\zeta_{\mathfrak{ch}_4,\topo}(\ess)$, we use some elementary facts about
quadratic forms to construct bilinear forms $\gamma_d$ with
$\zeta_{\mathfrak{ch}_d,\topo}(\ess) = \zeta_{\cB(\gamma_d),\topo}(\ess)$, where
$\cB(\gamma_d)$ is now amenable to our methods for $d = 3$ and $d = 4$.  Thus,
write $[a_1, \dotsc, a_r ]$ for the form $(x_1,\dotsc,x_r) \mapsto a_1 x_1^2 +
\dotsb + a_r x_r^2$, where $r\ge 0$.  Let $h(x_1,x_2) = x_1x_2$ be the standard
hyperbolic form and write $h^{\perp a} = h \perp \dotsb \perp h$ ($a$ copies).
Define a quadratic form $f_d$ in $d$ variables as follows: for $d = 4a + b$ with
$0\le b < 4$, let $f_d = h^{\perp 2a} \perp g_b$, where $g_0 = [\,]$, $g_1 =
[1]$, $g_2 = [1,1]$, and $g_3 = h \perp [-1]$.
\begin{lemma*}
  
  If $p\not= 2$, then $(x_1,\dotsc,x_d)\mapsto x_1^2 + \dotsb + x_d^2$ and $f_d$
  are equivalent over $\ZZ_p[\sqrt 2]$.
\end{lemma*}
\begin{proof}
  Let $\sim$ signify equivalence over $\ZZ_p[\sqrt 2]$.  By Hensel's lemma,
  there exist $x,y\in \ZZ_p$ with $x^2 + y^2 = -1$.  Let $A = \begin{mimat*} x &
    -y & \\ y & x \end{mimat*}$ and $B = \frac{1}{\sqrt{2}} \begin{mimat*} -1 &
    1 \\ 1 & 1 \end{mimat*}$.
  Then $A A^\top = \begin{mimat*} -1 & 0 \\ 0 & -1\end{mimat*}$ and $
  B \begin{mimat*} -1 & 0 \\ 0 & 1\end{mimat*} B^\top = \begin{mimat*} 0 & 1 \\
    1 & 0\end{mimat*}$.
  Hence, $[1,1,1] \sim [-1,-1,1] \sim g_3$ and $[1,1,1,1] \sim h\perp h$.
\end{proof}
By quadratic reciprocity, if $p\not= 2$, then $\ZZ_p$ contains $\sqrt 2$ if and
only if $p\equiv \pm 1 \bmod 8$.  Let $\gamma_d$ be the symmetric bilinear form
on $\ZZ^d$ with $\gamma_d(\xx,\xx) = f_d(\xx)$ for $\xx \in \ZZ^d$.
\begin{cor*}
  If $p \not= 2$, then ${\mathfrak{ch}_d\otimes \ZZ_p[\sqrt 2]}
  \approx_{\ZZ_p[\sqrt 2]} {B(\gamma_d)\otimes\ZZ_p[\sqrt 2]}$.  Hence,
  $\zeta_{\mathfrak{ch}_d,\topo}(\ess) =
  \zeta_{\cB(\gamma_d),\topo}(\ess)$. \qed
\end{cor*}

In contrast to $\mathfrak{ch}_3$ and $\mathfrak{ch}_4$, our methods do apply to
$\cB(\gamma_3)$ and $\cB(\gamma_4)$.  By choosing bases similar to the cases of
central products of Heisenberg Lie rings, disregarding monomials as before, the
computation of $\zeta_{\cB(\gamma_3),\topo}(\ess)$ involves a non-degenerate
pair of binomials, while for $\zeta_{\cB(\gamma_4),\topo}(\ess)$, we face a
non-degenerate triple consisting of a trinomial and two binomials.  The
situation for $\cB(\gamma_5)$ is more complicated and, in particular,
degenerate.

\subsection{Outlook: conquering degeneracy}
\label{ss:outlook}

The examples provided in \S\S\ref{ss:sl2}--\ref{ss:Heisenberg} show that the
techniques developed in the present article can be used to compute interesting
examples of $\fP$-adic and topological subalgebra (and, similarly, submodule)
zeta functions.  Unfortunately, for typical examples of algebras and modules of
interest, more often than not, our non-degeneracy assumptions tend to be
violated.  In \cite{topzeta2}, we will describe a refined form of the method
developed here and we will design additional techniques to overcome certain
forms of degeneracy.  These improvements and extensions considerably extend the
scope of our method.  To provide some further motivation for the conjectures
given in \S\ref{s:conjectures}, we now provide some examples computed with the
help of \cite{topzeta2}.

\subsubsection*{Example: $\mathfrak{gl}_2(\ZZ)$}

To the author's knowledge, up until now, $\Sl_2(\ZZ)$ remained the sole example
of an insoluble Lie ring whose topological or $\fP$-adic subalgebra zeta
function has been computed.  Using \cite{topzeta2}, we can add the topological
subalgebra zeta function of $\mathfrak{gl}_2(\ZZ)$ to the list:
\begin{equation}
  \label{eq:gl2}
  \zeta_{\mathfrak{gl}_2(\ZZ),\topo}(\ess) = 
  \frac{27\ess - 14}{6(6\ess - 7)(\ess - 1)^3 \ess}.
\end{equation}

\subsubsection*{Example: nilpotent groups of Hirsch length at most $5$}
\label{ss:Fil4}

Let $\Fil_4$ be the nilpotent Lie ring with $\ZZ$-basis $(\ee_1,\dotsc,\ee_5)$
such that $[\ee_1, \ee_2] = \ee_3$, $[\ee_1,\ee_3] = \ee_4$, $[\ee_1,\ee_4] =
\ee_5$, and $[\ee_2,\ee_3] = \ee_5$, where the remaining commutators
$[\ee_i,\ee_j]$ not implied by anti-commutativity are taken to be zero.  The
ideal zeta function of $\Fil_4 \otimes \ZZ_p$ was computed by Woodward
\cite[Thm\ 2.39]{dSW08}.  ``Despite repeated efforts'' \cite[p.~\!54]{dSW08} of
his to compute them, the local subalgebra zeta functions of $\Fil_4$ remain
unknown.  Using the methods from the present article and the techniques from
\cite{topzeta2}, we find that
\begin{align}
  \nonumber \zeta_{\Fil_4,\topo}(\ess) =\,\, & \bigl(392031360 \ess^9 -
  5741480808 \ess^8 + 37286908278 \ess^7 - \\&\,\, \nonumber 140917681751 \ess^6
  + 341501393670 \ess^5 - 550262853249 \ess^4 + \\&\,\, \nonumber 589429290044
  \ess^3 - 404678115300 \ess^2 + 161557332768 \ess - \\&\,\, \nonumber
  28569052512\bigr) {\mathlarger{/}} \bigl(3(15\ess - 26)(7\ess - 12)(7\ess -
  13)(6\ess - 11)^3 \\&\quad\quad (5\ess- 8)(5\ess - 9)(4\ess - 7)^2(3\ess -
  4)(2\ess - 3)(\ess - 1)\ess\bigr).
  \label{eq:Fil4}
\end{align}
Our conjectures in \S\ref{s:conjectures} suggest that the numbers in
\eqref{eq:Fil4} are far from random.

As we will now explain, our computation of $\zeta_{\Fil_4,\topo}(\ess)$
concludes the classification of topological subgroup zeta functions (see the end
of \S\ref{ss:topsub}) of all finitely generated torsion-free nilpotent groups
$G$ with $\Hirsch(G) \le 5$, where $\Hirsch(G)$ denotes the Hirsch length of
$G$.  Recall that the topological zeta function $\zeta_{G,\topo}(\ess)$ of such
a group $G$ only depends on the $\CC$-isomorphism type of $\mathfrak{L}(G)
\otimes_{\QQ} \CC$, where $\fL(G)$ is the $\Hirsch(G)$-dimensional Lie algebra
over $\QQ$ associated with $G$ under the Mal'cev correspondence.  It is
well-known that, up to $\CC$-isomorphism, there are precisely $16$ non-trivial
nilpotent Lie algebras of dimension at most $5$ over $\CC$ (see e.g.\ \cite[\S
4]{dG07}) and all these Lie algebras admit $\ZZ$-forms.  With the sole exception
of $\Fil_4 \otimes \CC$, each of these Lie algebras over $\CC$ admits a
$\ZZ$-form whose $p$-adic subalgebra zeta functions have been recorded in
\cite{dSW08}.  We can use the techniques from \cite{topzeta2} to compute the
associated topological zeta functions of these Lie rings. As expected, it turns
out that in all cases, the results coincide with those obtained by naively
applying the non-rigorous ``informal approach'' from the introduction to the
$p$-adic formulae in \cite{dSW08}.  In this sense, $\zeta_{\Fil_4,\topo}(\ess)$
was indeed the only missing case which we now managed to compute without
resorting to a computation of the associated $p$-adic zeta functions.

\subsubsection*{Example: submodules for unipotent representations}

The study of submodule zeta functions arising from orders in semisimple
associative $\QQ$-algebras was initiated by Solomon \cite{Sol77}.  In
particular, he gave explicit formulae, valid for almost all primes, for the
local factors in terms of the Wedderburn decomposition of the associated
$\QQ$-algebra.  At the other end of the spectrum lie ideal zeta functions of
nilpotent Lie rings (cf.\ Remark~\ref{r:algmod}(\ref{r:algmod2})), a great
number of which have been recorded in \cite[Ch.\ 2]{dSW08}.  We now consider
topological zeta functions arising from the enumeration of submodules of $\ZZ^d$
under the action of the unipotent group
 $$\mathrm U_d(\ZZ) = \begin{bmatrix} 
   1 & \ZZ & \dotsb & \ZZ \\
   0 & \ddots & \ddots & \vdots \\
   \vdots & \ddots & \ddots  & \ZZ \\
   0 & \dotsb & 0 & 1
 \end{bmatrix} \le \GL_d(\ZZ)$$ for $d\le 5$.  That is, using the language of
 \S\ref{s:groups}, we consider $\zeta_{\cU_d(\ZZ)\acts \ZZ^d,\topo}(\ess)$,
 where $\cU_d(\ZZ)$ is the associative $\ZZ$-subalgebra of $\Mat_d(\ZZ)$
 generated by $\mathrm U_d(\ZZ)$.  The author is not aware of any previous
 computations of these zeta functions or their $\fP$-adic versions.

 It is a trivial matter to determine the first two examples
 \begin{equation}
   \label{eq:U2U3}
   \zeta_{\cU_2(\ZZ) \acts \ZZ^2,\topo}(\ess) = \frac{1}{(2\ess - 1)\ess},
   \qquad
   \zeta_{\cU_3(\ZZ) \acts \ZZ^3,\topo}(\ess) = \frac{4\ess - 1}{2(3\ess - 1)(2\ess - 1)^2\ess}.
 \end{equation}
 For $\mathrm U_4(\ZZ)$, a computation similar to the case of $\mathfrak{ch}_2$
 in \S\ref{ss:Heisenberg} shows that
 \begin{equation}
   \label{eq:U4}
   \zeta_{\cU_4(\ZZ)\acts \ZZ^4,\topo}(\ess) = 
   \frac{3360 \ess^5 - 5192 \ess^4 + 3139 \ess^3 - 930\ess^2 + 136\ess - 8}{8(7\ess - 3)(5\ess - 2)(4\ess - 1)(3\ess - 1)^2(2\ess - 1)^3\ess}.
 \end{equation}

 The computation for $\mathrm U_5(\ZZ)$ is noticeably more involved.  Using
 \cite{topzeta2}, we obtain
 \begin{align}
   \nonumber \zeta_{\cU_5(\ZZ)\acts \ZZ^5,\topo}(\ess) =\,\, & \bigl(
   435891456000 \ess^{14} - 1957609382400 \ess^{13} + 4053337786080 \ess^{12} -
   \\&\,\,\nonumber 5128457632240\ess^{11} + 4429884320966 \ess^{10} -
   2763848527457 \ess^9 + \\&\,\,\nonumber 1284781950540 \ess^8 - 452226036325
   \ess^7 + 121188554644 \ess^6 - \\&\,\,\nonumber 24624421916 \ess^5 +
   3737984412 \ess^4 - 411498360 \ess^3 + 31087152 \ess^2 - \\&\,\,\nonumber
   1443744\ess + 31104 \bigr){\mathlarger{/}} \bigl( 36(13\ess - 6)(11\ess -
   4)(10\ess - 3)(9\ess - 4)(8\ess - 3) \\&\quad\quad (7\ess - 2)(7\ess -
   3)(5\ess - 1)(5\ess - 2)^2(4\ess - 1)^2(3\ess - 1)^2(2\ess - 1)^4\ess\bigr).
   \label{eq:U5}
 \end{align}

\section{Conjectures}
\label{s:conjectures}

Based on substantial experimental evidence gathered using the method described
in this article and its extension in \cite{topzeta2}, we now state and discuss
a series of conjectures on topological zeta functions of algebras and modules.
We first explain our conjectures in the context of subalgebra zeta functions.

 \subsection{General conjectures}
 \label{ss:genconj}

 Let $\cA$ be a non-associative $\fo$-algebra which is free of rank $d$ as an
 $\fo$-module, where $\fo$ is the ring of integers in a number field $k$.
 Experimental evidence (for $k = \QQ$) suggests that the following three
 conjectures hold without any further assumptions on $\cA$.

 \begin{conj}
   \label{conj:degree}
   $\deg_{\ess}( \zeta_{\cA,\topo}(\ess)) = -d$.
 \end{conj}

 This is in contrast to topological zeta functions of polynomials in the sense
 that for $f\in k[X_1,\dotsc,X_n]$, the degree of $\Zeta_{f,\topo}(\ess)$ in
 $\ess$ does not merely depend on $n$.

 If true, Conjecture~\ref{conj:degree} has the following $\fP$-adic consequence
 which, to the author's knowledge, has not been recorded before.  Suppose that
 there are numbers $a_i\in \ZZ$, $b_i\in \ZZ\setminus\{0\}$, and $\varepsilon_i
 = \pm 1$ ($i\in I$, for finite $I$) with the following property:
 \begin{itemize}
 \item[(\textlabel{$\clubsuit$}{PRT})] For every $p$-adic field $K\supset k$
   with valuation ring $\fO$, unless $p$ belongs to some finite exceptional set,
   we have
   \[
   \zeta_{\cA_{\fO}}(s) = \prod_{i\in I} (1-q^{a_i - b_i s})^{\varepsilon_i},
   \]
   where $q$ is the residue field size of $K$.
 \end{itemize}
 For instance, \eqref{eq:sl2O} shows that (\ref{PRT}) is satisfied for $\cA =
 \Sl_2(\ZZ)$.  Furthermore, various examples of local ideal zeta functions of
 nilpotent Lie rings given in \cite[Ch.\ 2]{dSW08} satisfy the evident analogue
 of (\ref{PRT}) (at least when $K = \QQ_p$, see \S\ref{ss:topsubmodule} and cf.\
 Remark~\ref{r:justify_voodoo}).

 Assuming that (\ref{PRT}) holds, one can deduce from Theorems~\ref{thm:top} and
 \ref{thm:subdenef} that $\sum_{i\in I}\varepsilon_i \ge -d$.
 Conjecture~\ref{conj:degree} now predicts that, in fact, the equality
 $\sum_{i\in I}\varepsilon_i = -d$ holds.

 \begin{conj}
   \label{conj:zero}
   $\zeta_{\cA,\topo}(\ess)$ has a pole at zero.
 \end{conj}

 Again, topological zeta functions of polynomials behave more erratically as the
 simple example of $\Zeta_{X^e,\topo}(\ess) = 1/(e\ess+1)$ shows.  On the
 $\fP$-adic side, the meromorphic continuations of local subalgebra zeta
 functions seem to have poles at zero but the author is not aware of an
 explanation as to why this is so.

 \begin{conj}
   \label{conj:roots}
   If $s\in \CC$ satisfies $\zeta_{\cA,\topo}(s) = 0$, then $0 < \Real(s) <
   d-1$.
 \end{conj}
 The ``$0 < \Real(s)$'' part of Conjecture~\ref{conj:roots} would imply that the
 numerator of $\zeta_{\cA,\topo}(\ess) \in \QQ(\ess)$ has alternating signs
 which is indeed the case for the examples given in \S\ref{s:app}.  We note that
 in all examples known to the author, the upper bound of $d-1$ for the real
 parts of the roots of $\zeta_{\cA,\topo}(\ess)$ is rather pessimistic.

 \subsection{The nilpotent case: behaviour at zero}
 \label{ss:conjnilpotent}

 Let $\cA$ be as in \S\ref{ss:genconj}. In particular, let $d$ denote the rank
 of $\cA$ as an $\fo$-module.  If $\cA$ is nilpotent, then experimental evidence
 suggests a considerably strengthened form of Conjecture~\ref{conj:zero} to be
 true.  We shall henceforth assume that $\cA$ is either Lie or (non-unital)
 associative so that there is no ambiguity as to what we mean by nilpotency of
 $\cA$.

 \begin{conj}[topological form]
   \label{conj:comagic}
   If $\cA$ is nilpotent, then $\zeta_{\cA,\topo}(\ess)$ has a simple pole at
   zero with residue $(-1)^{d-1} /(d-1)!$.
 \end{conj}

 For example, in the case of $\Fil_4$ from \S\ref{ss:Fil4}, we have
$$\ess \zeta_{\Fil_4,\topo}(\ess)\big\vert_{\ess=0} = 
\frac{28569052512}{3(-26)(-12)(-13)(-11)^3(-8)(-9)(-7)^2(-4)(-3)(-1)} = -1/24.$$

Some assumption on $\cA$ is certainly required in Conjecture~\ref{conj:comagic}.
For example, while $\zeta_{\mathfrak{gl}_2(\ZZ),\topo}(\ess)$ (see
\eqref{eq:gl2}) does have a simple pole at zero, the residue is $-1/3$ instead
of $-1/6$.  Also, zero need not be a simple pole in general.  For example,
consider the non-associative ring $\cR = \ZZ \ee_1 \oplus \ZZ \ee_2 \oplus \ZZ
\ee_3$ with $\ee_1^2 = \ee_2^{\phantom 1}$, $\ee_2^2 = \ee_3^{\phantom 1}$,
$\ee_3^2 = \ee_1^{\phantom 1}$, and $\ee_i^{\phantom 1} \ee_j^{\phantom 1} = 0$
for $i\not= j$.  Using the techniques from \cite{topzeta2}, we compute $
\zeta_{\cR,\topo}(\ess) = \frac{(15\ess - 7)(5\ess - 1)}{3(7\ess - 3)(2\ess -
  1)^2\ess^2}.  $ Such examples can be found more easily for ideal zeta
functions.  For instance, if $(\ZZ^d,\!\dtimes)$ denotes $\ZZ^d$ endowed with
component-wise multiplication, then $\zeta_{(\ZZ^d,\!\dtimes),\topo}(\ess) =
\ess^{-d}$.

\addtocounter{conj}{-1}

Conjecture~\ref{conj:comagic} might merely be the topological shadow of a
$\fP$-adic conjecture which, to the author's knowledge, has not been previously
noted:
 
\begin{conj}[$\fP$-adic form]
  \label{conj:comagic_P}
  Suppose that $\cA$ is nilpotent.  Let $K \supset k$ be a $p$-adic field with
  valuation ring $\fO$.  Then $\zeta_{\cA_{\fO}}(s)$ has a simple pole at zero
  and
  \[
  \frac{\zeta_{\cA_{\fO}}(s)}{\zeta_{(\fO^d,0)}(s)}\Bigg\vert_{s=0} = 1.
  \]
\end{conj}
Here, $(\fO^d,0)$ denotes $\fO^d$ endowed with the zero multiplication---it is
well-known that $\zeta_{(\fO^d,0)}(s) = \prod_{i=0}^{d-1} (1-q^{i-s})^{-1}$,
where $q$ is the residue field size of $K$.  Also, we identified the function
$\frac{\zeta_{\cA_{\fO}}(s)}{\zeta_{(\fO^d,0)}(s)}$ (which is defined for
$\Real(s) > d$, say) with its meromorphic continuation to the complex plane.

\subsection{On topological submodule zeta functions}
\label{ss:topsubmodule}

Some care has to be taken to adapt
Conjectures~\ref{conj:degree}--\ref{conj:comagic_P} to the case of submodule
zeta functions, primarily because our definition of topological submodule zeta
functions in \S\ref{ss:topsub} allows them to vanish.  For example, if $K$ is
any $p$-adic field with valuation ring $\fO$ and residue field size $q$, then
$\zeta_{\Mat_d(\fO) \acts \fO^d}(s) = 1/(1-q^{-ds})$, where $\Mat_d(\fO)$ acts
on $\fO^d$ in the natural way; cf.\ \cite[Prop.\ 4.1]{dS00}.  Hence, if $d > 1$,
then $\zeta_{\Mat_d(\ZZ) \acts \ZZ^d,\topo}(\ess) = 0$.  We predict that
Conjectures~\ref{conj:degree}--\ref{conj:comagic_P} also hold for submodule zeta
functions whenever they are non-zero.  The nilpotency condition in
Conjecture~\ref{conj:comagic_P} of course needs to be adapted: for a free
$\fo$-module $M$ of finite rank, we only consider nilpotent (associative)
subalgebras $\cE$ of $\End_{\fo}(M)$.  We note that, although we do not have a
proof of this, nilpotency of $\cE \subset \End_{\fo}(M)$ seems to imply that
$\zeta_{\cE \acts M,\topo}(\ess) \not= 0$.

The submodule zeta functions arising from nilpotent subalgebras
$\cE\subset \End_{\fo}(M)$ include those enumerating ideals in nilpotent Lie
rings.  Indeed, if $\cL$ is a nilpotent Lie algebra over $\fo$ which is free of
finite rank as an $\fo$-module, then, by the standard proof of Engel's theorem
(see e.g.\ \cite[Ch.\ II, \S\S 2--3]{Jac79}), $\mathrm{ad}(\cL)$ generates a
nilpotent associative subalgebra $\cE$ of $\End_{\fo}(\cL)$ and the
$\cE$-submodules of $\cL$ are precisely the ideals of $\cL$.

{\footnotesize \bibliography{topzeta}}

\end{document}